\documentclass[phd]{cornell}

\usepackage{amsmath,amssymb}
\usepackage{xcolor}
\usepackage{mathabx}
\usepackage[margin=3.5cm]{geometry}

\usepackage[english]{babel}
\usepackage{amsthm}
\usepackage{tikz-cd}
\usepackage[]{mathrsfs}
\usepackage{mathtools}
\usepackage{todonotes}
\usepackage{calc}
\usepackage{graphicx}
\usepackage{faktor}
\usepackage{amsfonts}
\usepackage{hyperref}
\hypersetup{colorlinks=true, linkcolor=black}
\usepackage{verbatim}
\usepackage{pinlabel}
\usepackage[T1]{fontenc}
\usepackage{graphicx}
\usepackage{graphbox}
\usepackage{caption}
\usepackage{subcaption}
\usepackage{enumerate}
\usepackage{csquotes}
\usepackage{float}

\usepackage[backend=biber, style=numeric, sorting=none]{biblatex}
\usepackage{graphicx}%
\usepackage{multirow}%
\usepackage{amsmath,amssymb,amsfonts}%
\usepackage{amsthm}%
\usepackage{mathrsfs}%
\usepackage[title]{appendix}%
\usepackage{xcolor}%
\usepackage{textcomp}%
\usepackage{manyfoot}%
\usepackage{booktabs}%
\usepackage{algorithm}%
\usepackage{algorithmicx}%
\usepackage{algpseudocode}%
\usepackage{listings}%

\usepackage{faktor}
\usepackage{hyperref}
\hypersetup{colorlinks=true, linkcolor=black}
\usepackage{verbatim}
\usepackage{pinlabel}
\usepackage{graphbox}
\usepackage{mathtools}
\usepackage{enumerate}

\usepackage{anyfontsize}

\usepackage{url}

\usepackage[backend=biber, style=numeric]{biblatex}
\theoremstyle{plain}
\newtheorem{theorem}{Theorem}[section]
\newtheorem{corollary}[theorem]{Corollary}
\newtheorem{lemma}[theorem]{Lemma}
\newtheorem{proposition}[theorem]{Proposition}

\newtheorem{introthm}{Theorem}          
\newtheorem{introcor}[introthm]{Corollary}

\newtheorem*{claim}{Claim}
\newtheorem*{lemma*}{Lemma}
\newtheorem*{convention}{Convention}

\theoremstyle{definition}
\newtheorem{definition}[theorem]{Definition}
\newtheorem{example}[theorem]{Example}
\newtheorem*{remark}{Remark}

\theoremstyle{remark}

\newtheorem*{question*}{Question}

\newtheorem*{defn*}{Definition}

\theoremstyle{remark}

\newcommand{\N}{\mathbb{N}}

\newcommand{\Z}{\mathbb{Z}}

\newcommand{\R}{\mathbb{R}}

\renewcommand{\H}{\mathbb{H}}

\newcommand{\Homeo}{\mathrm{Homeo}}

\newcommand{\id}{\mathrm{id}}

\newcommand{\wF}{\widetilde{\mathcal{F}}}

\newcommand{\F}{\mathcal{F}}
\newcommand{\Q}{\mathbb{Q}}

\newcommand{\G}{\mathcal{G}}

\renewcommand{\wF}{\widetilde{\mathcal{F}}}
\newcommand{\wG}{\widetilde{\mathcal{G}}}
\newcommand{\walpha}{\widetilde{\alpha}} 
 
\newcommand{\wz}{\widetilde{z}}
\renewcommand{\wr}{\widetilde{r}}
\newcommand{\wLambda}{\widetilde{\Lambda}}
\newcommand{\hF}{\widehat{\F}}
\newcommand{\hz}{\widehat{z}}
\newcommand{\hx}{\widehat{x}}

\newcommand{\Ends}{\mathrm{Ends}}

\newcommand{\overbar}[1]{\mkern 1.5mu\overline{\mkern-1.5mu#1\mkern-1.5mu}\mkern 1.5mu}

\begin{document}

\title{On bifoliated planes, their structure and group actions}

\author{Mauro Camargo}
\conferraldate{August}{2025}  
\maketitle

\makecopyright

\begin{abstract}
    Bifoliated planes arise naturally in the study of Anosov flows on $3$-manifolds. To any Anosov flow on a $3$-manifold $M$, one can associate a bifoliated plane equipped with an action of the fundamental group of $M$ which encodes the topology of the flow. 

This thesis begins by showing left-orderability of any group acting faithfully on a bifoliated plane. We then describe the bifoliated planes associated with two families of Anosov flows which are constructed from algebraic and combinatorial data via gluing procedures. For one of these families, we show that all the resulting bifoliated planes are isomorphic. In contrast, for the other family, we show that the defining data can be recovered as a topological invariant of the bifoliated plane.
\end{abstract}







\begin{acknowledgements}

I would like to thank everyone from whom I have ever learned a piece of mathematics. This rather large list of people includes professors, classmates, teachers, TAs, colleagues, and students I have met (either in person or through their writing) not only during my graduate studies, but throughout all of my education. Among them, I would like to name some people or groups of people who I feel have been particularly influential.

First, I want to thank my advisor, Katie Mann, for her guidance and teaching during my PhD, as well as her encouragement and patience. Throughout these years, I have always found inspiration in the clarity and enthusiasm with which she discusses mathematics, and I have always left a meeting with her more motivated than when I went in, even at times when motivation was not at its highest.

Next, I would like to thank the topology faculty and students at Cornell, who made possible many courses and seminars which were fundamental to my mathematical education in this area. This group includes, in particular, the other members of my committee, Ben Dozier and Jason Manning, whom I additionally thank for their time and effort reviewing my work and providing feedback.

In October 2024, I visited Queen's University in Kingston, Canada, where I was invited by Thomas Barthelmé to speak at the Dynamics, Geometry, and Groups seminar. During my stay in Kingston, some questions regarding the structure of bifoliated planes associated with Franks-Williams Anosov flows were raised during discussions between Katie, Thomas, Lingfeng Lu (Thomas' student, and my coauthor for the work on the first chapter of this thesis), Isaac Broudy (also a student of Katie at Cornell) and myself, and these questions became the starting point for the third chapter of this thesis. Therefore, I would like to thank the math department at Queen's University, Thomas Barthelmé, and Lingfeng Lu for their hospitality.

Before coming to Cornell, I completed my undergraduate studies at Universidad de la República in Uruguay. I believe that the mathematical education I received during that time was an excellent preparation for my PhD. I am very thankful to the Uruguayan mathematical community for my years spent there, and for their continued support afterwards. I would especially like to thank professors and students in the dynamical systems research group for introducing me to the study of this subject. Special thanks to my undergraduate thesis advisor Rafael Potrie for suggesting a very formative undergraduate thesis topic, and for his encouragement and guidance in my graduate school applications.

Last but most definitely not least, I would like to thank everyone who has contributed non-mathematically to my life during these years. The list of such people dwarfs the list of mathematical contributors, since it must include not only those closest to me, but also people I have never met and may never meet, such as screenwriter Jesse Armstrong or the cooks at the popular Ithaca restaurant ``New Delhi Diamond's.'' For practical reasons it is not possible to list them all here. 

Therefore, I will thank the most important of them: my friends, both in Ithaca and in Uruguay, and my family (humans and non-humans), for being there to share the good times and to support me in the less-good times. I will also thank my girlfriend, Nathalia, for all our shared moments and for her love and support in the final stretch of my PhD. Finally, I will thank my mother, without whom none of this would have been possible (both in the tautological and in a deeper sense), for the past three decades of wanting the best for me, and for being an example of hard work and perseverance.

\end{acknowledgements}
\contentspage

\pagenumbering{arabic}
\setcounter{page}{1}

\chapter{Introduction}

\section{Introduction}

The main objects of interest in this thesis are \emph{bifoliated planes}.

\begin{definition}[\cite{barthelme2022orbit}]\label{definition:def_bifoliated_plane}
A \emph{bifoliated plane} $(P, \F_1, \F_2)$ consists of a topological plane $P$ equipped with $C^0$ foliations by lines $\F_1, \F_2$ such that:
    \begin{enumerate}
        \item $\F_1, \F_2$ are proper, that is, each leaf is properly embedded in $P$.
        \item $\F_1, \F_2$ are topologically transverse, i.e. around each point in $P$ there is a coordinate chart which maps leaves of $\F_1$ and $\F_2$ to horizontal and vertical straight lines in $\R^2$, respectively.
    \end{enumerate}    
\end{definition}

Here we will be assuming that the foliations $\F_1$ and $\F_2$ are non-singular, although it is also natural to allow (as is done in \cite{barthelme2022orbit}) certain types of singularities. Additionally, since $\F_1$ and $\F_2$ are foliations of a topological plane they must be orientable, and we will always assume that they are oriented. 

The classical examples of bifoliated planes arise from the study of Anosov flows, which are an important class of hyperbolic dynamical systems, deeply connected to $3$-dimensional geometry and topology.

The idea of understanding Anosov flows on $3$-manifolds via the study of orbit spaces was introduced by Barbot and Fenley (\cite{barbot_1995}, \cite{fenley1994anosov}), who showed that given an Anosov flow $\varphi$ on a $3$-manifold $M$, the orbit space $P_\varphi$ of its lift $\widetilde{\varphi}$ to the universal cover $\Tilde{M}$ of $M$ can be shown to be homeomorphic to $\R^2$ \cite{fenley1994anosov}.  A pair of transverse foliations of $P_\varphi$ by lines, $\F^-$ and $\F^+$, is obtained from the weak stable and weak unstable foliations of the flow $\widetilde{\varphi}$. Thus, $(P_\varphi, \F^+, \F^-)$ is a bifoliated plane, which is additionally equipped with an action of $\pi_1(M)$ coming from its action on $\widetilde{M}$ by deck transformations. It was shown in \cite{barthelme2022orbit}, \cite{barthelme2023anosov}, \cite{barthelme2024nontransitivepseudoanosovflows} that, by studying aspects of the action of $\pi_1(M)$ on $(P_\varphi, \F^+, \F^-)$, one can classify Anosov flows up to orbit equivalence. This motivates the study of bifoliated planes as well as groups acting on them.

In Chapter \ref{chapter:LO}, we study the left orderability of groups acting on bifoliated planes. All results in this chapter are joint work with Lingfeng Lu and appeared in \cite{camargo2025leftorderability}, except for those in Section \ref{subsection:LO_Anosov_like}.

A group is \emph{left-orderable} if it admits a linear order that is invariant under left multiplication. The definition is purely algebraic, but left-orderability appears frequently in topological and dynamical contexts. For example, one part of the \emph{L-space conjecture} by Boyer--Gordon--Watson (\cite{Boyer2013}) and Juhzász (\cite{juhasz2014surveyheegaardfloerhomology}) assert the equivalence between admitting a taut foliation on an irreducible 3-manifold and the fundamental group of the manifold being left-orderable. A link to dynamics is given by the classical result stating that a countable group is left-orderable if and only if it acts faithfully on $\R$ by orientation-preserving homeomorphisms; see \cite{Holland63} or \cite{ghys_1984}. 

We prove the following result regarding left-orderability of groups acting on bifoliated planes:
 
\begin{introthm}[\cite{camargo2025leftorderability}]\label{introthm:main}

    Let $P = (P, \F_1, \F_2)$ be a bifoliated plane. Let $\mathrm{Aut^+}(P, \F_1, \F_2) $ be the group of homeomorphisms of $P$ that preserves foliations $\F_1$ and $\F_2$ as well as their orientations. Then, $\mathrm{Aut}^+(P, \F_1, \F_2)$ is left-orderable.
    \end{introthm}

    Since left-orderability of a group is inherited by its subgroups, one can see as a consequence that any group acting faithfully on a bifoliated plane preserving the foliations and their orientations is left-orderable. Additionally, by the classical result in the theory of left orderable groups mentioned above, one can show that if the group is countable, then it can be realized as a subgroup of $\mathrm{Homeo^+}(\R)$.

    For bifoliated planes $(P_\varphi, \F^+, \F^-)$ associated with an Anosov flow $\varphi$ on a $3$-manifold $M$, the action of $\pi_1(M)$ on $(P_\varphi, \F^+, \F^-)$ has certain properties that constrain the topology of the bifoliated plane. These properties have been axiomatized in \cite{barthelme2024nontransitivepseudoanosovflows}, and group actions satisfying these axioms on bifoliated planes $(P, \F^+, \F^-)$ are known as \emph{Anosov-like} actions. In this setting, we show:
    
    \begin{introcor}\label{introcor:Anosov_like}
        Let $(P, \F^+, \F^-)$ be a bifoliated plane equipped with an Anosov-like action by a group $G$. Then, $G$ can be realized as a subgroup of $\mathrm{Homeo}^+(\R)$.
    \end{introcor}

    We also discuss how one can identify the ends of the leaf spaces $\Lambda_i = P/\F_i$ of the foliations $\F_1$ and $\F_2$ with certain subsets of the \emph{boundary at infinity} of the bifoliated plane $(P, \F_1, \F_2)$.

 In \cite{fenley2012ideal}, Fenley defined a boundary at infinity for the bifoliated plane associated to an Anosov flow. In the setting of general bifoliated planes (not necessarily coming from an Anosov flow), Bonatti \cite{bonatti} defined canonically the \emph{circle at infinity} via ideas similar to the work of Mather \cite{mather}. The boundary circle at infinity is a topological circle which compactifies the plane into a closed disk. Given an action by a group $G$ on a bifoliated plane, there is also a natural induced action on its circle at infinity --- this was a key ingredient in the study of Anosov-like actions on general bifoliated planes, including the particular case of actions on orbit spaces of Anosov flows, in \cite{barthelme2022orbit}. \par

Since rays of leaves limit onto points of the boundary circle, it is natural to ask whether rays in the leaf spaces of the foliations limit to points on the boundary circle in a meaningful way. We answer this question, describing a natural correspondence between the set of ends of leaf spaces of a bifoliated plane and a subset of the associated boundary circle at infinity --- we call this correspondence the \emph{realization} of ends.

We show that there are two \emph{types} of realizations: the point-type and the interval-type. The point-type realization comes from a sequence of leaves with their ideal points in $\partial P$ converging to the same limit, while the interval-type realization comes from a local bifoliation structure called an \emph{infinite product region}. 

As an application of this, we show that in the case of a general bifoliated plane $(P, \F_1, \F_2)$, if the leaf spaces $\Lambda_1$ and $\Lambda_2$ are particularly simple, one can directly prove that $\mathrm{Aut}^+(P, \F_1, \F_2)$ (and therefore any group acting faithfully and preserving orientations on $(P, \F_1, \F_2)$) is realizable as a subgroup of $\mathrm{Homeo}^+(\R)$.

\begin{introcor}\label{introcor_finite_ends}
If $\Ends_+(\Lambda_i)$ or $\Ends_-(\Lambda_i)$ is finite for some $i$, then $\mathrm{Aut}^+(P, \F_1, \F_2)$ has a global fixed point on $\partial P$. In particular, $\mathrm{Aut}^+(P, \F_1, \F_2)$ is isomorphic to a subgroup of $\Homeo^+(\R)$.
\end{introcor}

In Chapters \ref{chapter:BF} and \ref{chapter:FW}, we explicitly describe the bifoliated planes associated with two families of Anosov flows constructed via a surgery and gluing procedure. In general, there is no procedure that allows one to easily describe the bifoliated plane corresponding to a given Anosov flow, other than in the very simple cases of geodesic and suspension flows (see Sections \ref{subsection:background_Anosov} and \ref{subsection:prelim_bifoliated} below). Partial results by Bonatti and Iakovoglou in \cite{bonatti2023anosov}, unrelated to our work in Chapters \ref{chapter:BF} and \ref{chapter:FW}, show how some bifoliated planes change when doing certain surgeries on the corresponding Anosov flow. 

The two families of Anosov flows discussed in these chapters are parametrized by certain combinatorial and algebraic data. We show that the associated families of bifoliated planes differ fundamentally in how they depend on the choice of this data. In the case of the family studied in Chapter \ref{chapter:BF} we show that, up to isomorphism, the bifoliated plane does not depend on this choice. On the other hand, for the family considered in Chapter \ref{chapter:FW} we show that, generically, two different choices of initial data yield non-isomorphic bifoliated planes.

In Chapter \ref{chapter:BF}, we discuss in detail the bifoliated planes $(P, \F^+, \F^-)$ corresponding to a family of \emph{totally periodic} Anosov flows on graph $3$-manifolds constructed by Barbot and Fenley in \cite{barbot2013pseudo}. An Anosov flow $\varphi$ on a graph manifold $M$ is said to be totally periodic if for every Seifert piece of the JSJ decomposition of $M$, there exists a periodic orbit of the flow which (seen as a loop in $M$) has a power that is freely homotopic to a regular fiber of the Seifert piece.

The examples we work with in Chapter \ref{chapter:BF} are constructed by a gluing procedure due to Barbot and Fenley. 

First, one defines a \emph{building block}, which is a $3$-manifold with boundary together with a partial flow which is tangent to two boundary components and transverse to two boundary components of the block, entering the block in one transverse boundary component and exiting from the other one. 

Then, one chooses some finite number of \emph{fat graphs} $X_i \subset \Sigma_i$ which satisfy a certain admissibility property. A fat graph is a finite graph $X$ embedded in a surface with boundary $\Sigma$, such that $\Sigma$ deformation retracts onto the graph. For each edge on the graph $X_i$, one takes a copy of the building block described above. The blocks are glued along their transverse boundary components according to combinatorial information given by the graph, obtaining a number of circle bundles $N(X_i)$ over $\Sigma_i$, each equipped with a flow transverse to its boundary components, which are tori.

Finally, the manifolds $N(X_i)$ are glued together along their boundary components. For a choice of gluing maps satisfying a generic condition $(*)$, it was shown in \cite{barbot2013pseudo} that the flow obtained on the resulting manifold $M$ is an Anosov flow.

We show in this chapter the following result:

\begin{introthm}\label{introthm:totally_periodic_isomorphic}
    For any two choices of fat graphs $\{X_i \subset \Sigma_i\}$ and gluing maps $\{A_j : \partial N(X_{i_j}) \to \partial N(X_{i_j'})\}$ satisfying condition $(*)$, the bifoliated planes of the resulting Anosov flows are isomorphic.
\end{introthm}

That is, the combinatorial information of the fat graphs and the choice of gluing maps of the torus boundary components of the pieces $N(X_i)$ does not affect the topology of the resulting bifoliated planes.

In order to prove Theorem $\ref{introthm:totally_periodic_isomorphic}$, we will describe in detail the structure of these bifoliated planes. The main structural feature of these planes which we will need to understand are \emph{trees of scalloped regions}, which are the projection to the orbit space of a connected component of the lift of a JSJ piece of $M$ to $\widetilde{M}$. Trees of scalloped regions in the orbit spaces of Anosov flows were first defined by Barthelmé-Fenley-Mann in \cite{barthelme2023anosov}. The union of these projections is all of the orbit space, and therefore understanding each tree of scalloped regions and how they fit together is key to understanding the topology of the bifoliated plane. 

Once we understand how the different trees of scalloped regions fit together, we show that any two bifoliated planes obtained by this construction are isomorphic. We do this via an iterative procedure.

Additionally, understanding the structure of the plane allows us to explicitly describe the Anosov-like action of $\pi_1(M)$ on $(P, \F^+, \F^-)$. In principle, this is doable for any manifold $M$ obtained in this way. A presentation for the fundamental group of $M$ can be calculated from the combinatorial data of the fat graphs and the boundary gluing maps. Then, one can explicitly describe how each generator of the fundamental group acts on the plane.

Here, we do it for the case of the Bonatti-Langevin flow, initially defined in \cite{bonatti1994exemple} via an explicit construction. We define it in terms of the Barbot-Fenley construction described above, as a flow with a single Seifert piece $N(X)$, where the fat graph $X$ is a figure-8 graph embedded in a punctured Möbius band. The ideas used to understand the action in the case of the Bonatti-Langevin flow generalize to other choices of fat graphs and gluing maps, but computations are simplified in this case due to the low complexity of the manifold supporting the flow.

In Chapter \ref{chapter:FW}, we describe the bifoliated planes $\left(P_A, \F^+_A, \F^-_A\right)$ associated to a class of Anosov flows, the Franks-Williams flows $\varphi_A$ on a manifold $M_A$, obtained via a surgery technique.

The construction of the flow $\varphi_A$, described in Section \ref{section:FW_construction}, involves a choice of a hyperbolic matrix $A \in \mathrm{SL}(2,\Z)$. This construction was first introduced by Franks and Williams in \cite{FW1980anomalous} for a specific choice of $A$. This flow is commonly referred to as \emph{the} Franks-Williams flow. It was the first example of a non-transitive Anosov flow. 

The same arguments used by Franks and Williams work for any choice of matrix $A$ in $\mathrm{SL}(2,\Z)$. In \cite{YangYu2022classifying}, Yang and Yu studied the flows obtained when considering an arbitrary choice of hyperbolic matrix $A$. We refer to the flows obtained in this way as Franks-Williams flows.

An important feature of these flows is the existence of a torus in the manifold $M_A$ which is transverse to the flow. This torus is (up to isotopy) the only essential torus in the JSJ decomposition of $M_A$, splitting the manifold into two connected components, which are atoroidal pieces. Each of these components contains a basic set for the flow, with one basic set being an attractor and the other one being a repeller. These are compact sets with empty interior, and contain all the orbits that do not intersect the transverse torus. Therefore, an open and dense subset of the orbit space of the lift $\widetilde{\varphi_A}$ consists of orbits which intersect some lift of this torus.

Then, in order to understand these bifoliated planes, one (as in Chapter \ref{chapter:BF}) needs to understand certain chains of lozenges. In this case, these are the projections to the orbit space of lifts of the transverse torus. It is also important to understand how the different chains are organized in the plane, and this is what leads us to Theorem \ref{introthm:FW_main_theorem} below. This theorem answers a question of Barthelmé, who asked if all the bifoliated planes associated to these flows are isomorphic.

We show that many of the bifoliated planes associated to the flows $\varphi_A$ can be distinguished from each other via a topological invariant, closely related to the continued fraction expansion of the slopes of the eigenspaces of $A$. 

We prove: 
\begin{introthm}\label{introthm:FW_main_theorem}
     Let $A, B \in \mathrm{SL}(2,\Z)$, and let $\left(P_A, \F^+_A, \F^-_A\right)$, $\left(P_B, \F^+_B, \F^-_B\right)$ be the bifolated planes corresponding to the flows $\varphi_A$ and $\varphi_B$. Let the slopes $u_A, u_B$ of the expanding eigenspaces of $A$ and $B$ be given by
   \[
\pm u_A = 
a_0+ \cfrac{1}{\displaystyle a_1 +
  \cfrac{1}{\displaystyle a_2 + 
   \cfrac{1}{\ddots}}},
\,\, \pm u_B = 
b_0+ \cfrac{1}{\displaystyle b_1 +
  \cfrac{1}{\displaystyle b_2 + 
   \cfrac{1}{\ddots}}}, \]
   where $a_i, b_i \in \N$.

    If $\left(P_A, \F^+_A, \F^-_A\right)$ and $\left(P_B, \F^+_B, \F^-_B\right)$ are isomorphic, then there exist $m,n \geq 0$ such that $a_{m+i} = b_{n + i}$ for all $i\geq 0$.
\end{introthm}

As a consequence, one can show the following using classical results relating matrices in $\mathrm{GL}(2,\Z)$ to the theory of continued fractions, and the structure of $\mathrm{SL}(2,\Z)$.
\begin{introcor}\label{introcor:FW_cor}
    Let $A, B \in \mathrm{SL}(2,\Z)$, and let $\left(P_A, \F^+_A, \F^-_A\right)$, $\left(P_B, \F^+_B, \F^-_B\right)$ be the bifolated planes corresponding to the flows $\varphi_A$ and $\varphi_B$.

    If $\left(P_A, \F^+_A, \F^-_A\right)$ and $\left(P_B, \F^+_B, \F^-_B\right)$ are isomorphic, then there exist $k,l\in \Z$ such that $A^k$ and $\pm B^l$ are conjugate in $\mathrm{GL}(2,\Z)$.
\end{introcor}

This shows in particular that there are infinitely many non-isomorphic bifoliated planes of this form.

\section{Preliminaries}

\subsection{Anosov flows}\label{subsection:background_Anosov}

In this section, we present standard definitions and results regarding Anosov flows and their orbit spaces. In this thesis, $M$ will always denote a $3$-dimensional manifold.

\begin{definition}
    Let $M$ be a compact manifold. A flow $\varphi:\R\times M \to M$generated by a $C^1$ vector field $X$ is said to be Anosov if there exists a continuous splitting $TM = \R X\oplus E^u \oplus E^s $ and constants $C, \lambda >0$ such that:
    \begin{enumerate}
        \item For all $t\in \R$, the splitting is $d\varphi_t$-invariant.
        \item For all $v\in E^s$ and all $t\in \R$, we have $\|d\varphi_t(v) \|\leq C e^{-\lambda t}\|
        v \|$.
        \item For all $v \in E^u$ and all $t\in \R$, we have $\|d\varphi_{-t}(v) \|\leq C e^{-\lambda t}\| v\|$.
    \end{enumerate}
\end{definition}

\begin{remark}
    Compactness of $M$ ensures that the choice of norm in the definition above does not matter.
\end{remark}

\begin{example}
    \begin{enumerate}
        \item Let $S$ be a compact surface equipped with a hyperbolic metric and let $T^1S$ be the unit tangent bundle of $S$. Then, the geodesic flow $\varphi:\R\times M \to M$ is an Anosov flow.

        \item Let $f_A: T^2 \to T^2$ be induced by the linear map $A:\R^2\to\R^2$ given by multiplication by a hyperbolic matrix $A\in \mathrm{SL}(2,\Z)$. Let $M_A$ be the suspension of the map $f_A$, and let $\varphi$ be the flow induced in $M_A$ by the unit norm vertical vector field. Then, $\varphi:\R\times M_A \to M_A$ is an Anosov flow.
        \end{enumerate}
\end{example}

The examples given above are the so-called \emph{algebraic} Anosov flows. By performing surgeries and gluing constructions, one can construct other examples of Anosov flows on $3-$manifolds (see e.g. Chapters \ref{chapter:BF} and \ref{chapter:FW}).

The bundles $E^s \oplus \R X$ and $E^u \oplus \R X$ are known as the \emph{stable} and \emph{unstable} bundle respectively. It is a classical result (see for instance \cite{fisher2019hyperbolic}) that these bundles are uniquely integrable, yielding transverse $2$-dimensional foliations $\F^s$ and $\F^u$, known as the stable and unstable foliations of the flow $\varphi$. 

Given an Anosov flow $\varphi$ on a $3$-manifold $M$, one can lift $\varphi$ to the universal cover $\widetilde{M}$ of $M$, obtaining a flow $\widetilde{\varphi}$. The foliations $\F^s$ and $\F^u$ also lift to foliations $\wF^s, \wF^u$ of $\widetilde{M}$. 

Now, we describe the connection between Anosov flows and bifoliated planes.

\begin{definition}
Given an Anosov flow $\varphi$ on a $3$-manifold $M$, the \emph{orbit space} $\mathcal{O}_\varphi$ of $\varphi $ is defined to be the quotient space of $\widetilde{M}$ by the equivalence relation whose equivalence classes are the orbits of $\widetilde{\varphi}$.  
\end{definition}

It follows from results of Verjovsky and Palmeira that the universal cover $\widetilde{M}$ of a $3$-manifold supporting an Anosov flow is homeomorphic to $\R^3$.

\begin{theorem}[\cite{barbot_1995}, \cite{fenley1994anosov}]
    The orbit space $\mathcal{O}_\varphi$ of an Anosov flow is homeomorphic to a topological plane. The foliations $\wF^s, \wF^u$ descend to a pair of transverse $1$-dimensional foliations $\F^-, \F^+$ of $\mathcal{O}_\varphi$, whose leaves are properly embedded in $\mathcal{O}_\varphi$.
\end{theorem}

That is, the orbit space of an Anosov flow together with the foliations $\F^+, \F^-$ is a bifoliated plane as defined in Definition \ref{definition:def_bifoliated_plane}. From here on, we will use the notation $P_\varphi$ for the orbit space $\mathcal{O}_\varphi$ of $\varphi$.

\subsection{Bifoliated planes}\label{subsection:prelim_bifoliated}

Here we state important results and definitions relating to bifoliated planes. 

First, we give two examples of bifoliated planes which arise naturally from the study of Anosov flows.

\begin{figure}[h]
  \centering
  \begin{minipage}[b]{0.30\textwidth}
    \includegraphics[width=\textwidth]{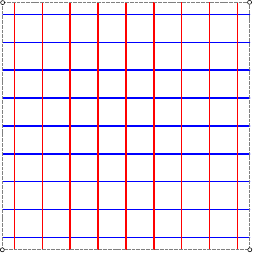}
    \caption{Trivial bifoliated plane.}
    \label{fig:image1}
  \end{minipage}
  \hspace{2cm}
  \begin{minipage}[b]{0.30\textwidth}
    \includegraphics[width=\textwidth]{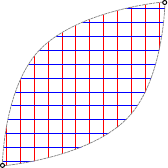}
    \caption{Skew bifoliated plane.}
    \label{fig:image2}
  \end{minipage}
\end{figure}

\begin{example}\label{example:background_skew_trivial}
    \begin{enumerate}
        \item Let $P = \R^2$, let $\F^+$ be the foliation of $P$ by vertical straight lines, and let $\F^-$ be the foliation of $P$ by horizontal straight lines. We call $(P, \F^+, \F^-)$ the \emph{trivial} bifoliated plane.

        \item Let $P = \{ (x,y) \in \R^2: x < y < x +1 \}$, and let $\F^+$ and $\F^-$ be the foliations of $P$ by vertical and horizontal straight lines. We call $(P, \F^+, \F^-)$ the \emph{skew} bifoliated plane.
    \end{enumerate}    
\end{example}

It will be important to have a notion of equivalence, or isomorphism, for bifoliated planes.

\begin{definition}
    We say that bifoliated planes $(P_1, \F^+_1, \F^-_1)$ and $(P_2, \F^+_2, \F^-_2)$ are isomorphic if there exists a homeomorphism $h: P_1 \to P_2$ that maps leaves of $\F^+_1$ to leaves of $\F^+_2$ and leaves of $\F^-_1$ to leaves of $\F^-_2$.
\end{definition}

\begin{remark}
    It's easy to see that the skew bifoliated plane and the trivial bifoliated plane are not isomorphic. For instance, it suffices to see that in the latter, each leaf in $\F^+$ intersects all leaves of $\F^-$. This is not the case in a skew bifoliated plane.
    
    One can show that the bifoliated plane coming from the geodesic flow on the unit tangent bundle of a hyperbolic surface is isomorphic to the skew plane, and that the bifoliated plane coming from the suspension flow of a hyperbolic linear map of the torus is isomorphic to a trivial plane.
\end{remark}

More examples can be explicitly constructed by generalizing the examples discussed above.

\begin{example}
    Let $U\subset \R^2$ be a connected open set, and let $\F^+, \F^-$ be the foliations on $U$ by straight vertical and horizontal lines. Then, if we let $P = \widetilde{U}$ be the universal cover of $U$ equipped with the lifted foliations $\wF^+, \wF^-$, we have that $(P, \wF^+, \wF^-)$ is a bifoliated plane.
\end{example}
\begin{figure}[h]
  \centering
  \includegraphics[width = 0.4\textwidth]{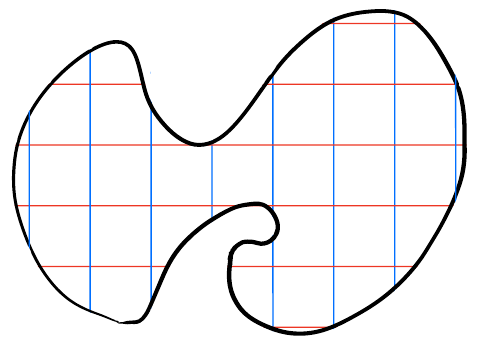}
  \caption{A simply connected open set with transverse foliations.}
  \label{fig:weird_open}
\end{figure}

It will be useful in the study of bifoliated planes to work with their associated leaf spaces. 

\begin{definition}
    Let the leaf spaces $\Lambda^+, \Lambda^-$ of $\F^+$ and $\F^-$ be defined as $\Lambda^+ = P /\F^+$ and $\Lambda^- = P/ \F^-$. That is, they are the spaces obtained by identifying each leaf of the foliations $\F^+, \F^-$ respectively to a point, equipped with the quotient topology.

\end{definition}

\begin{proposition}[\cite{candel2000foliations}]\label{proposition : leaf_spaces}
    If $(P, \F_1, \F_2)$ is a bifoliated plane, then the leaf spaces $\Lambda_1, \Lambda_2$ of $\F_1, \F_2$ respectively, are simply connected and second countable 1-manifolds. Additionally, once one has chosen orientations for $\F_1$ and $\F_2$, these orientations induce orientations on $\Lambda_2$ and $\Lambda_1$, respectively. 
\end{proposition}

This holds in fact for the leaf space of any proper foliation by lines in $\R^2$, since one only needs existence of local transversals to the foliation in order to show that these leaf spaces are locally Euclidean. Note that the leaf spaces are Hausdorff if and only if they are homeomorphic to $\R$.\par

In this thesis, 1-manifolds are always understood to be second countable (but not necessarily Hausdorff), connected, and simply connected. We also assume that they are oriented whenever it is required. \par

In the skew and trivial examples defined in Example \ref{example:background_skew_trivial}, we can see that both leaf spaces are homeomorphic to the real line. However, this is not the case in general, as seen in the example shown in Figure \ref{fig:weird_open}. In general, leaf spaces are not Hausdorff, and we will see below that for bifoliated planes coming from Anosov flows, this is true of all interesting examples.

The following definition describes how leaves of a foliation look in the bifoliated plane when its leaf space is non-Hausdorff.

\begin{definition}
    Two leaves of a foliation $\mathcal{F}$ are said to be \emph{non-separated} if they do not admit disjoint neighborhoods saturated by leaves of $\mathcal{F}$.
\end{definition}

A saturated neighborhood of a leaf $l \in \F_i$ is given by leaves in $\F_i$ that intersect a transversal to $l$. This implies that two leaves of $\mathcal{F}$ are non-separated if and only if there exists a sequence of leaves of $\mathcal{F}$ that accumulates on both of them. Given $l_1, l_2, l_3 \in \F_i$, if $l_1$ is non-separated from $l_2$ and $l_2$ is non-separated from $l_3$, it is not necessary for $l_1$ and $l_3$ to be non-separated (see Figure \ref{fig: non-separated leaves}).
\vspace{20 pt}
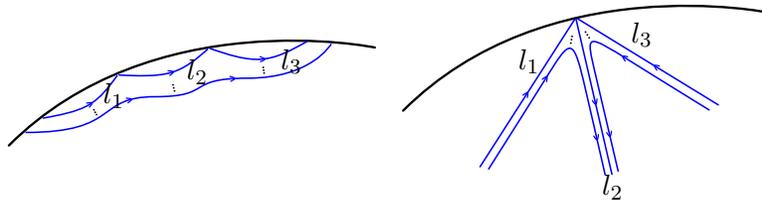
\begin{figure}[h]
    \labellist
    \small\hair 2pt
    \pinlabel $l_1$ at 80 60
    \pinlabel $l_2$ at 143 78
    \pinlabel $l_3$ at 214 86
    \pinlabel $l_1$ at 393 86
    \pinlabel $l_2$ at 456 -9
    \pinlabel $l_3$ at 478 106
     \endlabellist
    \centering
    \vspace{-2mm}\includegraphics[scale = 0.50]{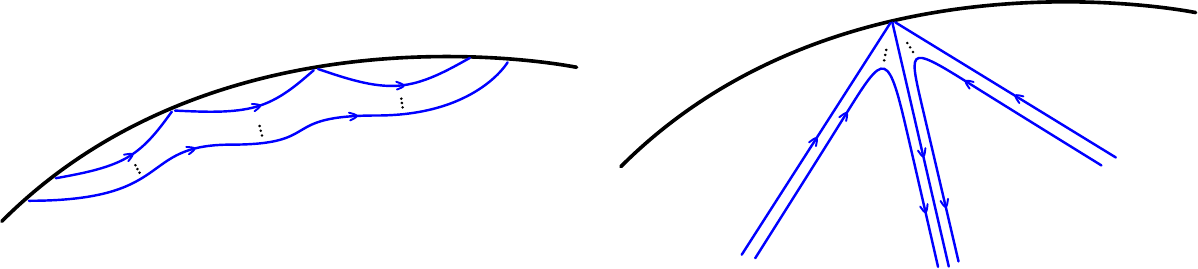}
    \caption{On the left, $l_1$, $l_2$, and $l_3$ are pairwise non-separated; on the right, $l_1$ and $l_2$ are non-separated, $l_2$ and $l_3$ are non-separated, while $l_1$ and $l_3$ are separated by $l_2$.}
    \label{fig: non-separated leaves}
\end{figure}   

\begin{remark}
For any bifoliated plane $(P_\varphi, \F^s, \F^u)$ coming from an Anosov flow $\varphi_t$ on a 3-manifold as described in the introduction, the leaf spaces $\Lambda^s, \Lambda^u$ of $\F^s, \F^u$ are either both Hausdorff (and hence homeomorphic to $\R$) or both non-Hausdorff. This was shown by Barbot and Fenley in \cite{barbot_1995, fenley1995sided}. There exist examples of both behaviors. In a general bifoliated plane this is not necessarily the case, and we will not assume it in Chapter \ref{chapter:LO}, where we study bifoliated planes that do not necessarily arise from Anosov flows.
\end{remark}

As we said in the introduction, the most important class of bifoliated plane is that of bifoliated planes $(P_\varphi, \F^+, \F^-)$ arising from an Anosov flow $\varphi$. In this case, the $\pi_1(M)$ action on $P_\varphi$ has special properties which come from the hyperbolicity properties of the flow. In \cite{barthelme2022orbit} and \cite{barthelme2024nontransitivepseudoanosovflows}, these properties have been axiomatized first by Barthelmé-Frankel-Mann (in the transitive case, generalizing $\pi_1(M)$ actions on $P_\varphi$ when the flow $\varphi$ is topologically transitive) and later by Barthelmé-Bonatti-Mann (in the general case) into what is known as \emph{Anosov-like} actions on bifoliated planes.

\begin{definition}[\cite{barthelme2024nontransitivepseudoanosovflows}]
    
An action of a group $G$ on a bifoliated plane $(P, \F^+, \F^-)$ which preserves both foliations is called Anosov-like if it satisfies the following:

\begin{enumerate}
\item[(A1)] If a nontrivial element of $G$ fixes a leaf $l \in \mathcal{F}^+$, then it has a fixed point $x \in l$ and is topologically expanding on one leaf through $x$ and topologically contracting on the other.

\item[(A2)] The union of leaves of $\mathcal{F}^+$ that are fixed by some element of $G$ is dense in $P$, as is the union of leaves of $\mathcal{F}^-$ that are fixed by some element of $G$.

\item[(A3)] Each singular point is fixed by some nontrivial element of $G$.

\item[(A4)] If $l$ is a leaf of $\mathcal{F}^+$ or $\mathcal{F}^-$ that is non-separated with some leaf $l'$ in the corresponding leaf space, then some nontrivial element $g \in G$ fixes $l$.

\item[(A5)] There are no totally ideal quadrilaterals in $P$. A totally ideal quadrilateral is a trivially foliated region bounded by four leaves $l_1, l_2\in \F^+,f_3,f_4\in\F^-$ such that $l_i$ makes a perfect fit with $f_j$, for $i, j = 1,2$.
\end{enumerate}
    
\end{definition}

\begin{proposition}[\cite{barthelme2022orbit}]
    The $\pi_1(M)$-action on the bifoliated plane $(P_\varphi, \F^+, \F^-)$ associated with an Anosov flow $\varphi$ on a $3$-manifold $M$ is an Anosov-like action.
\end{proposition}

There are certain structures on bifoliated planes associated with Anosov flows, and more generally for bifoliated planes which admit an Anosov-like action of a group $G$, which are important when studying these planes. We now define some of these, which will appear in Chapters \ref{chapter:BF} and \ref{chapter:FW}. 

\begin{definition}[\cite{barthelme2022orbit}]

Let $l^+ \in \F^+$ and $l^- \in \F^-$ be non-intersecting half-leaves. We say they make a perfect fit if there is an arc $\tau^+$ with an endpoint on $l^+$ and contained in a leaf of $\F^-$, and an arc $\tau^-$ with an endpoint on $l^-$ and contained in a leaf of $\F^+$, such that:
\begin{enumerate}
    \item Every leaf $s^+ \in \F^+$ that intersects the interior of $\tau^+$ intersects $l^-$.
    \item Every leaf $s^-\in \F^-  $ that intersects the interior of $\tau^-$ intersects $l^+$. 
\end{enumerate}  

If two leaves in $\F^+$ and $\F^-$ contain half-leaves making a perfect fit, we also say that the leaves make a perfect fit.
\end{definition}

\begin{definition}
Let $x \in P$, and let $r^{\pm}_x \in \F^{\pm}(x), r^{\pm}_y \in \F^{\pm}(y)$ be half-leaves with endpoints $x, y$ respectively. Suppose that the rays $r^{\pm}_x$ and $r^{\mp }_y$ make perfect fits. Then, the lozenge $L \subset P$ bounded by $r^{\pm}_x, r^{\pm}_y$ is the open region of $P$ whose boundary is the union of these half-leaves. That is,
\[
L = \{  z\in P: \F^{+}(z) \cap r^-_{x} \neq \emptyset \text{ and }  \F^{-}(z) \cap r^{+}_x \neq \emptyset  \}.
\]
We say that $x$ and $y$ are the \emph{corners} of $L$, and the half leaves $r^{\pm}_x, r^{\pm}_y$ are its sides. If $L$ and $L'$ are lozenges that share a side, we say that they are \emph{adjacent}.
\end{definition}

\begin{definition}\label{definition:chain_adjacent_lozenges}
    A chain of lozenges is the union $\mathcal{C}$ of a set $\{ L_\alpha\}_{\alpha \in I}$ of lozenges such that for any pair of lozenges $L, L'\in \{ L_\alpha\}_{\alpha \in I}$, there exist lozenges $L_{\alpha_0} = L, L_{\alpha_1}, \dots, L_{\alpha_n} = L' $ where $\alpha_i \in I$ such that for all $i = 0\dots, n-1$, $L_{\alpha_i}$ and $L_{\alpha_{i+1}}$ share a corner. 
    
    If given any two $L, L'$ in the chain we can choose the lozenges $L_{\alpha_0}, L_{\alpha_1}, \dots, L_{\alpha_n} $ so that $L_{\alpha_i}, L_{\alpha_{i+1}}$ additionally share a side for all $i=0,\dots,n-1$, we say that $\mathcal{C}$ is a chain of adjacent lozenges.
\end{definition}

For a bifoliated plane coming from an Anosov flow, there is a relationship between lozenges and free homotopies of periodic orbits of the flow. In the setting of Anosov-like actions, this generalizes to a relationship between lozenges and points on the plane which are fixed by the same element. One aspect of this relationship is described in Proposition 2.24 of \cite{barthelme2022orbit}:

\begin{proposition}[\cite{barthelme2022orbit}]\label{proposition:intro_fixed_points_joined_lozenges}

    Let $(P, \F^+, \F^-)$ be a bifoliated plane with an Anosov-like action of $G$. If two points $x, y\in P$ are fixed by the same element of $G$, then there exists a chain of lozenges $L_0, L_1, \dots, L_n$ such that $x$ is a corner of $L_0$ and $y$ is a corner of $L_n$.

\end{proposition}

Now, we describe a particular type of chain of lozenges which will be especially relevant in Chapter \ref{chapter:BF}.

\begin{definition}[\cite{barthelme2022orbit}]
     A scalloped region is an open, unbounded set $U \subset P$ with the following properties:

\begin{enumerate}
\item The boundary $\partial U$ consists of the union of four families of leaves $l_k^{1,+}, l_k^{2,+}$ in $\mathcal{F}^+$ and $l_k^{1,-}, l_k^{2,-}$ in $\mathcal{F}^-$, indexed by $k \in \mathbb{Z}$.

\item The leaves of each family $l_k^{i,\pm}$, $k \in \mathbb{Z}$ are pairwise nonseparated.

\item The boundary leaves are ordered so that there exists a (unique) leaf $l_k^{1,-}$ that makes a perfect fit with $l_k^{1,+}$ and $l_{k+1}^{1,+}$. Moreover, $l_k^{1,-}$ accumulates on the leaves $\bigcup_{i \in \mathbb{Z}} l_i^{1,-}$ as $k \to \infty$, and on $\bigcup_{i \in \mathbb{Z}} l_i^{2,+}$ as $k \to -\infty$. The analogous statement holds for leaves making perfect fits with the other families $l_k^{i,\pm}$.

\item The bifoliation is trivial inside $U$, i.e., for all $x \neq y \in U$, $\mathcal{F}^+(x) \cap \mathcal{F}^-(y) \neq \emptyset$ and $\mathcal{F}^+(y) \cap \mathcal{F}^-(x) \neq \emptyset$ and $U$ contains no singular points.
\end{enumerate}

\end{definition}

\begin{figure}[h]
  \centering
  \includegraphics[width=0.9\linewidth]{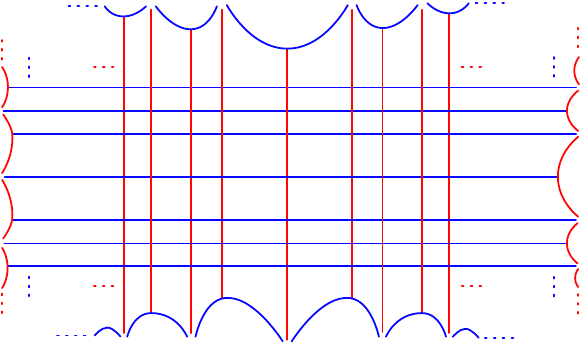}
  \caption{A scalloped region.}
  \label{fig:intro_scalloped}
\end{figure}

In particular, we see from the definition that each set of boundary leaves $\{ l^{i, \pm}_k : k\in \Z\}$ consist of infinitely many pairwise non-separated leaves. If the plane admits an Anosov-like action by some group, then this correspondence goes both ways.

\begin{proposition}[\cite{barthelme2022orbit}]\label{proposition:infinite_nonsep_implies_scalloped_boundary}
    Let $(P, \F^+, \F^-)$ be a bifoliated plane with an Anosov-like action of a group $G$. Suppose $P$ contains an infinite set of pairwise non-separated leaves. Then, this set is contained in the boundary of a single scalloped region in $P$.
\end{proposition}

For a bifoliated plane $(P_\varphi, \F^+, \F^-)$ associated to an Anosov flow $\varphi$ on a $3$-manifold $M$, there is a correspondence between chains of lozenges in the plane and \emph{Birkhoff annuli} in $M$. This will be of great importance in Chapters \ref{chapter:BF} and \ref{chapter:FW}, since it is by understanding certain chains of lozenges that we will be able to understand the planes studied in these chapters.

\begin{definition}[\cite{barbot1995mise}]
    A Birkhoff annulus $A\subset M$ for a flow $\varphi$ on $M$ is an immersed annulus which is transverse to $\varphi$, except on its boundary components which coincide with periodic orbits of the flow.
\end{definition}

Since the flow is transverse to the interior of a Birkhoff annulus $A$, the same is true of the stable and unstable foliations $\F^s, \F^u$. This induces two one-dimensional foliations $f^s, f^u$ on the interior of the annulus.

\begin{definition}
    A Birkhoff annulus $A\subset M$ is elementary if the induced foliations $f^s, f^u$ on its interior have no closed leaves.
\end{definition}

\begin{definition}
    A component of a Birkhoff annulus $A\subset M$ is a connected component of $\mathrm{int} (A) \setminus\{ \text{closed leaves of } f^s, f^u\}$. A component of a lift $\widetilde{A} \subset \widetilde{M}$ of a Birkhoff annulus is the lift of a component of $A$.
\end{definition}

In the proof of Proposition 5.1 of \cite{barbot1995mise}, Barbot shows that there exist finitely many such components for the lift Birkhoff annulus (and for the annulus itself), and that the interior of each of them projects to the interior of a lozenge. More precisely, he shows:

\begin{proposition}[\cite{barbot1995mise}]\label{proposition:birkhoff_projects_lozenges}

Let $A\subset M$ be a Birkhoff annulus, and let $\widetilde{A}\subset \widetilde{M}$ be a lift of $A $ to the universal cover $\widetilde{M}$. Then, the projection of $\widetilde{A}$ to $\widetilde{M}$ is a chain of lozenges. 

The interior of each lozenge in the chain is the projection of a component of $\widetilde{A}$, and two components $\widetilde{A}_1, \widetilde{A}_2 \subset \widetilde{A}$ are adjacent along a leaf of $\widetilde{f}^s$ (or $\widetilde{f}^u$) if and only if their projections are interiors of lozenges adjacent along leaves of $\F^-$ (resp. $\F^+$).    
\end{proposition}

\chapter{Left orderability of groups acting on bifoliated planes}\label{chapter:LO}

All original results in this chapter are joint work with Lingfeng Lu and appeared in \cite{camargo2025leftorderability}, with the exception of those in Section \ref{subsection:LO_Anosov_like}.

\section{Introduction}\label{section:LO_intro}

In this chapter, we study the left orderability of groups acting on bifoliated planes.

We begin by discussing background results in Sections \ref{subsection:zhao} and \ref{subsection:LO_prelim_circle}. In \ref{subsection:zhao}, we discuss ends of $1$-dimensional manifolds as well as \emph{orderable cataclysms}. These objects are crucial to the proof of Theorem \ref{introthm:main}. In \ref{subsection:LO_prelim_circle}, we briefly outline the definition of the circle at infinity of a bifoliated plane.

Our first goal is to prove Theorem \ref{introthm:main}, which we restate here as Theorem \ref{theorem:LO_main}.

\begin{theorem}\label{theorem:LO_main}
    Let $P = (P, \F_1, \F_2)$ be a bifoliated plane. Let $\mathrm{Aut^+}(P, \F_1, \F_2) $ be the group of homeomorphisms of $P$ that preserves foliations $\F_1$ and $\F_2$ as well as their orientations. Then, $\mathrm{Aut}^+(P, \F_1, \F_2)$ is left-orderable.
\end{theorem}

The proof of Theorem \ref{theorem:LO_main} in Section \ref{section:LO_proof_main_thm} involves introducing an order on the set of ends of the leaf spaces of the bifoliation. This was inspired by Zhao's work in \cite{zhao}, where he showed that for a $3$-manifold $M$ admitting a taut foliation with \emph{orderable cataclysms}, the fundamental group $\pi_1(M)$ is left-orderable. In general, given an action of a group $G$ on a non-Hausdorff $1$-manifold $\Lambda$, it is not always the case that $\Lambda$ has orderable cataclysms with respect to this action (see for instance Example 3.7 in \cite{calegari2003laminations}). However, with the extra structure coming from a bifoliation, there is a natural way of defining an order on every cataclysm in each leaf space that is preserved by actions of $G$. \par

The following is a well-known fact from the theory of left-orderable groups, which we will apply in several occasions.

\begin{theorem}[\cite{Holland63}]\label{theorem:LO_countable_ord_implies_homeo}
    A countable and left orderable group is isomorphic to a subgroup of $\mathrm{Homeo}^+(\R)$.
\end{theorem}

 It is not true in general that the group $\mathrm{Aut}^+(P, \F_1, \F_2)$ is a countable group. In Section \ref{subsection:LO_Anosov_like}, we prove Corollary \ref{introcor:Anosov_like}, which states that if the action of a group $G$ on a bifoliated plane $(P, \F^+, \F^-)$ is Anosov-like, then $G$ can be realized as a subgroup of $\mathrm{Homeo}^+(\R)$. We do this by showing in Corollary \ref{corollary:LO_Anosov_like_uncount_skew_or_triv} that if a bifoliated plane $(P, \F^+, \F^-)$ admits an Anosov-like action by an uncountable group $G$, then $(P, \F^+, \F^-)$ is either a trivial or a skew bifoliated plane.

Next, in section \ref{section:LO_realizing_ends} we address the question of how the ends of the leaf spaces $\Lambda_i = P/\F_i$ of the foliations $\F_i$ of $P$ relate to points in the circle at infinity $\partial P$ defined by Fenley (\cite{fenley2012ideal}) and Bonatti (\cite{bonatti}). The main result in this direction is Proposition \ref{proposition: mapping to ends}, which characterizes the possible \emph{realizations} of an end of a leaf space $\Lambda_i$ as a subset of the circle at infinity.

Finally, we prove some consequences of the correspondence between ends of leaf spaces and subsets of the circle at infinity. These include Corollary \ref{introcor_finite_ends} stated in the introduction to this thesis, and restated below as Corollary \ref{corollary: lift to left-orderability}, which shows that if the space of ends of at least one of the spaces $\mathrm{Ends}_\pm(\Lambda_i)$ is finite, then we can realize $\mathrm{Aut}^+(P, \F_1, \F_2)$ as a subgroup of $\mathrm{Homeo}^+(\R)$.

\section{Background}\label{section:LO_prelim}

\subsection{Ends of 1-manifolds and orderable cataclysms}\label{subsection:zhao}

Let $\Lambda$ be a non-Hausdorff and simply connected $1$-manifold, equipped with an action by a group $G$. Since simply connected manifolds are orientable, we will also assume $\Lambda$ to be oriented. We start by briefly stating some results concerning $\Lambda$. Most importantly, we will later make use of the fact that if $\Lambda$ has \emph{orderable cataclysms} with respect to the action of $G$, there exists a $G$-invariant linear order on $\mathrm{Ends}(\Lambda)$, the space of ends of $\Lambda$. This is a result of Zhao in \cite{zhao}.

We begin by stating the definition of (orderable) cataclysms. Recall that two points in a topological space are said to be \emph{non-separated} if any neighborhood of one of them intersects every neighborhood of the other. The following notion was originally definied by Calegari and Dunfield in \cite{calegari2003laminations} in the more general context of laminations.

\begin{definition}\label{def: cataclysm}
   A \emph{cataclysm} in $\Lambda$ is a maximal collection of pairwise non-separated points. Given an action of a group $G$ on $\Lambda$ by homeomorphisms, we say that $\Lambda$ has \emph{orderable cataclysms with respect to the action of} $G$ if there exists a $G$-invariant linear order in each cataclysm $\mu$ of $\Lambda$.
\end{definition}

It is a well-known fact that a group acting faithfully on a linearly ordered set in an order-preserving way must be left-orderable. In \cite{zhao}, Zhao showed that given a simply connected and orientable 1-manifold $\Lambda$ having orderable cataclysms with respect to an action by a group $G$, there exists a faithful order-preserving action of $G$ on a linearly ordered set. This linearly ordered set is the set of \emph{positive ends} of $\mathrm{\Lambda}$. The distinction between positive and negative ends of $\Lambda$ is possible since $\Lambda$ is assumed to be oriented. \par

\begin{definition}\label{def: rays and ends}
    A \emph{ray} in $\Lambda$ is a proper embedding $r: [0, +\infty) \to \Lambda$. Let $\mathcal{E}$ be the set of rays in $\Lambda$, and let $\sim$ be the equivalence relation on $\mathcal{E}$ such that two rays are equivalent if and only if they can be reparametrized so that they agree after a certain point. Then
    $$\Ends(\Lambda) \coloneqq \mathcal{E}/\sim$$
    is the set of \emph{ends} of $\Lambda$. An end is said to be \emph{positive} if a ray (hence all rays) representing it is orientation preserving, where we give $[0, +\infty)$ the standard orientation. Otherwise, the end is said to be \emph{negative}. Sets of positive and negative ends are denoted by $\Ends_+(\Lambda)$ and $\Ends_-(\Lambda)$, respectively.
\end{definition}
It is clear from the definition of $\mathrm{Ends}(\Lambda)$ that an action of $G$ on $\Lambda$ by homeomorphisms induces an action of $G$ on $\mathrm{Ends}(\Lambda)$. Moreover, if the original action is orientation preserving, then the set of positive ends is invariant under this action.

In the proof of our main theorem, we will make use of the following result.

\begin{proposition}[\cite{zhao}]\label{order_zhao}
    If $\Lambda$ has orderable cataclysms with respect to the action of a group $G$, then there exists a left-order on the set $\Ends_+(\Lambda)$ that is preserved by the induced action of $G$.
\end{proposition}

Using the above proposition together with 3-manifold topology, Zhao deduced left-orderability of certain 3-manifold groups:

\begin{theorem}[\cite{zhao}]
     Let $M$ be a closed, connected, orientable, irreducible $3$-manifold. If $M$ admits a taut foliation $\F$ with orderable cataclysm, then $\pi_1(M)$ is left-orderable.
\end{theorem}

We briefly explain how Zhao defines the left-order on $\mathrm{Ends}_+(\Lambda)$ mentioned in Proposition \ref{order_zhao}.

Given two distinct points $u, v$ in $\Lambda$, since $\Lambda$ may not be Hausdorff, it is possible that there does not exist an embedded path joining $u$ and $v$. However, for any $u,v$ in $\Lambda$ there exists a unique \emph{broken path} from $u$ to $v$, which consists of a sequence of embedded closed intervals $[a_0, a_1], [a_2, a_3], \dots, [a_{2n}, a_{2n+1}]$ such that $a_0 = u, a_{2n+1} = v$, and where $a_{2i - 1}, a_{2i }$ are contained in the same cataclysm (see \cite{calegari2007foliations}, Subsection 4.3 for details). Each pair of points $(a_{2i - 1}, a_{2i })$ is called a \emph{cusp} and can be either \emph{positive} or \emph{negative}, depending on how points in the pair are ordered within the cataclysm. \par

A broken path between two points of $\Lambda$ can be extended to a broken path between distinct ends. For any distinct positive ends $x_1$, $x_2$ in $\mathrm{Ends}_+(\Lambda)$, Zhao then defined the quantity $n(x_1, x_2)$ to be the difference between the number of positive cusps and the number of negative cusps on the broken path from $x_1$ to $x_2$. It is then shown that $n(x_1, x_2)$ possesses all necessary properties to define a linear order $ \overset{n}{<}$ on $\mathrm{Ends}_+(\Lambda)$ as follows: given $x_1, x_2 $ in $\mathrm{Ends}_+(\Lambda)$,  $x_1 \overset{n}{<} x_2$ if and only if $n(x_1, x_2) < 0$. Finally, it is easy to see from the definition of $n(x_1, x_2)$ that, since $G$ acts on $\Lambda$ by orientation preserving homeomorphisms and $\Lambda$ has orderable cataclysms with respect to this action, this order is $G$-invariant.

\subsection{The circle at infinity of a bifoliated plane}\label{subsection:LO_prelim_circle}

Given a bifoliated plane $P = (P, \F_1, \F_2)$, the circle at infinity $\partial P$ compactifies $P$ to a closed disk $P\cup \partial P$, where the set of endpoints of leaves of the foliations $\F_1, \F_2$ form a dense subset of the \emph{circle at infinity} $\partial P$.

The circle at infinity associated to a bifoliated plane was defined initially by Fenley (\cite{fenley2012ideal}) for the case of the bifoliated plane $(P_\varphi, \F^+, \F^-)$ associated to an Anosov flow. Later, Bonatti in \cite{bonatti} gave a definition of the circle at infinity for a general bifoliated plane (in fact, for a plane with a countable number of transverse foliations or potentially singular foliations whose singular points are $k$-prongs for $k > 2$, but we will not need this here). 

In Section \ref{subsection:LO_prelim_circle} we state the definition and the main properties of the circle at infinity for a bifoliated plane, following Bonatti's work in \cite{bonatti}. To streamline the exposition, we use the notation due to Mather in \cite{mather}, where ideal boundaries are defined for certain foliations of surfaces via an essentially identical procedure. For details and proofs of the results we will state, see \cite{bonatti}.

The main idea in the construction of the circle at infinity  associated to a bifoliated plane $(P, \F_1, \F_2)$ is the following: given a foliation by lines $\mathcal{F}$ of the plane, let $\mathrm{Ends}(\mathcal{F})$ denote the set of all ends of leaves of $\mathcal{F}$, that is
\[
\mathrm{Ends}(\mathcal{F}) = \bigcup_{l \text{ leaf of } \mathcal{F}} \mathrm{Ends}(l).
\]

One can then give a canonical \emph{circular order} $\mathcal{O}$ to the set $\mathrm{Ends}(\F_1)\cup \mathrm{Ends}(\F_2)$ when leaves are properly embedded lines in $P$ and the foliations are transverse to each other. This is possible since in this case any two leaves (belonging to the same or different foliations) can intersect at most once, and they must leave any compact subset of $P$. Given any three ends $\{x_1, x_2, x_3\}$ of leaves of $\mathcal{F}_1$ or $\mathcal{F}_2$, one can find a simple closed curve that intersects the rays corresponding to these ends exactly once. Giving this curve the boundary orientation corresponding to the compact region it bounds allows one to declare the triple $(x_1, x_2, x_3)$ as either positively or negatively oriented, and this can be shown to be independent of the choice of simple closed curve.

\begin{remark}
    Given a set $S$ equipped with a circular order $\mathcal{O}$, there are two operations one can perform on $(S, \mathcal{O})$ which yield circularly ordered sets related to $(S, \mathcal{O})$.
    \begin{enumerate}
        \item \emph{Identification of equivalent points}: we say two points $x, y\in S$ are equivalent if there are only finitely many points of $S$ between them on either side. The circularly ordered set $(\hat{S}, \hat{\mathcal{O}})$ is obtained by identifying all equivalent points. 

        \item \emph{Completion}: analogously to the Dedekind completion of a linearly ordered set, a circularly ordered set $(S, \mathcal{O})$ can be completed in a natural way in order to yield a complete circularly ordered set $(\Tilde{S}, \tilde{\mathcal{O}})$ that contains $(S,\mathcal{O})$. More precisely, Calegari showed in \cite{calegari_groups_order} that a circular order on a set $S$ consists of a collection of linear orders on all subsets of the form $S\setminus \{ x\}$, subject to some compatibility conditions. It can be seen that the usual Dedekind completions of these linear orders also satisfy the compatibility conditions, and yield a complete circular order on $S$.
    \end{enumerate}
\end{remark}

\begin{figure}[h!]
  \centering
\begin{minipage}[b]{0.9\textwidth}
  \includegraphics[width=\textwidth]{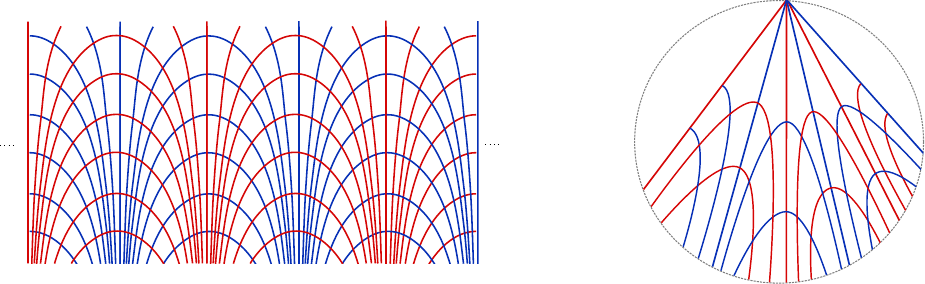}
    \caption{Part of a bifoliated plane consisting of two transverse Reeb foliations (left), and its compactification (right).}
    \label{fig:image8}
  \end{minipage}
 \end{figure}

This allows one to define the circle at infinity associated to a bifoliated plane:

\begin{definition}
    Let $(P, \F_1, \F_2)$ be a bifoliated plane, let $E =  \mathrm{Ends}(\F_1)\cup \mathrm{Ends}(\F_2) $ be the set of ends of leaves of $\F_1$ and $\F_2$, and let $\mathcal{O}$ be the natural circular order on $E$. Then, the \emph{circle at infinity} associated to $(P, \F_1, \F_2) $ is the circularly ordered set $\partial P =  \Tilde{\hat{E }  } $ with order $\Tilde{ \hat{ \mathcal{O} } }$. 
\end{definition}

Elements of $\partial P$ are called \emph{ideal points}. By construction, each end of a leaf corresponds to an ideal point. We adopt the following convention to simplify notation when we treat endpoints of leaves as elements on the circle at infinity.

\begin{convention} 
If $l$ is a leaf of $\F_i$, then $\partial l$ denotes the set of (two) ideal points of $l$.
\end{convention}

\begin{theorem}[\cite{bonatti}]
      The set $P \cup \partial P$ admits a natural topology that makes it homeomorphic to a closed disk which is the compactification of the plane $P$ by the circle $\partial P \cong S^1$. This topology is such that any homeomorphism of the plane preserving the foliations $\F_1, \F_2$ induces a homeomorphism of $P\cup \partial P$. 
      
\end{theorem}

\section{Proof of Theorem \ref{introthm:main}}\label{section:LO_proof_main_thm}

In this section, we prove Theorem \ref{theorem:LO_main}, which tells us that for a bifoliated plane $(P, \F_1, \F_2)$ we have that $\mathrm{Aut}^+(P, \F_1, \F_2)$ is left orderable.

\begin{remark}
    For the sake of brevity, when the foliations $\F_1$ and $\F_2$ on the plane $P$ are clear from the context, we will refer to $\mathrm{Aut}^+(P, \F_1, \F_2)$ as $\mathrm{Aut}^+(P)$.
\end{remark}

Given a bifoliated plane $P = (P, \F_1, \F_2)$ and a group $G \leq \mathrm{Aut}^+(P)$, there is an induced action of $G$ on the leaf spaces $\Lambda_1, \Lambda_2$. The first goal of this section is to show that each $\Lambda_i$ has orderable cataclysms with respect to this action, and then to use this to define a $G$-invariant linear order on the set $\mathrm{Ends}(\Lambda_1) \cup \mathrm{Ends}(\Lambda_2)$. We begin by explicitly defining the order on each cataclysm, and then we show that it is indeed a linear order. \par

We begin by defining a partial order on the leaf spaces $\Lambda_i$. By Proposition \ref{proposition : leaf_spaces}, fixing a continuous orientation on leaves of each of $\F_1$ and $\F_2$ induces an orientation on $\Lambda_2$ and $\Lambda_1$, respectively. Orientability of each $\Lambda_i$ allows us to define a partial order on $\Lambda_i$ as follows:

\begin{definition}
    Let $x,y $ in $\Lambda_i$. We say that $y > x$ if there exists an orientation-preserving embedding $\gamma\colon [0,1] \to \Lambda_i$ such that $\gamma(0) = x, \gamma(1) = y$.
\end{definition}

It is easy to check that this relation is a partial order on $\Lambda_i$. Note that the complement of any leaf of each $\F_i$ has exactly two connected components, we will use this fact to define an order on each cataclysm.

\begin{definition}
    Let $l$ be a leaf of $\F_i$. We define the \emph{front of} $l$ to be the connected component of $P\setminus l$ that consists of all the leaves $l'$ of $\F_i$ such that $q_i(l') > q_i(l)$ in $\Lambda_i$. Likewise, we define the \emph{back of} $l$ to be the connected component of $P\setminus l$ containing some leaf $l''$ of $\F_i$ such that $q_i(l'') < q_i(l)$.
\end{definition}

\begin{remark}
     The front and the back of $l$ are disjoint, and therefore their union is $P\setminus l$. Thus, given two leaves $l, l'$ of $\F_i$, one has that $l'$ is either in the back or in the front of $l$. However, note that (perhaps counterintuitively) it is possible to have $l$ and $l'$ be in the front (or the back) of each other. That is, it is not true that the fact that $l'$ is in front (back) of $l$ implies that $l$ is in the back (front) of $l'$. For an example, see Figure \ref{fig: lemma_nonsep_figure}: the leaves $l_1, l_2$ are in the back of each other. However, the next lemma shows that this does not happen when the leaves intersect a pair of non-separated leaves (a condition satisfied by leaves $s_1, s_2$ in Figure \ref{fig: lemma_nonsep_figure}).
\end{remark}


\begin{lemma}\label{non_sep}
Let $l_1, l_2$ be non-separated leaves of $\F_1$. Then, there exists leaves $s_1, s_2 $ of $ \F_2$ such that $s_i \cap l_{i} \neq \varnothing$ and such that either $s_1$ is in the back of $s_2$ and $s_2$ is in the front of $s_1$, or $s_1$ is in the front of $s_2$ and $s_2$ is in the back of $s_1$. Moreover, if $s_1$ is in the back of $s_2$ (or vice-versa), then this also holds for any other choice of such leaves $s_1', s_2'$.
\end{lemma}

\begin{proof}[Proof of Lemma \ref{non_sep}]

Let $(l_n')_{n\in \mathbb{N}}$ be a sequence of leaves of $\F_1$ that converges to a set of leaves containing $l_1$ and $l_2$, and let $U$ denote the connected region bounded by $l_1$ and $l_2$. First, we show the following:

\begin{lemma*}
The leaves $l_1$ and $l_2$ are oriented so that their orientation coincides with the boundary orientation induced by an orientation of $U$.
\end{lemma*}

\begin{proof}
Consider trivially foliated (and trivially oriented), disjoint rectangular neighborhoods $R_1$ and $R_2$ intersecting $l_1$ and $l_2$ respectively. Taking $N >0$ large enough, we have that $l_n'$ intersects both $R_1$ and $R_2$ for all $n \geq N$, and taking a subsequence we may assume that the $l_n'$ for $n\geq N$ are in the same connected component of $U \setminus l_N'$. Note that since the foliations are assumed to be nonsingular, any leaf can intersect $R_i$ at most once, and all leaves must have their first intersection with $R_i$ on the same side of $R_i$, for $i=1,2$. The notion of ``first intersection'' of a leaf with $R_i$ is well defined since the foliation is oriented.

Now, suppose that $l_1$ and $l_2$ are not oriented in such a way that their orientation coincides with the boundary orientation corresponding to an orientation of $U$ (see Figure \ref{fig: lemma_nonsep_impossible}, on the left). Then, $l_1$ and $l_2$ are in different connected components of $U \setminus l_N'$. This implies that the leaves $l_n'$ for $n>N$ cannot accumulate on one of $l_1$ and $l_2$, which is a contradiction. This proves the lemma.
\end{proof}

\begin{figure}[h]
  \centering
    \includegraphics[width=0.80\textwidth]{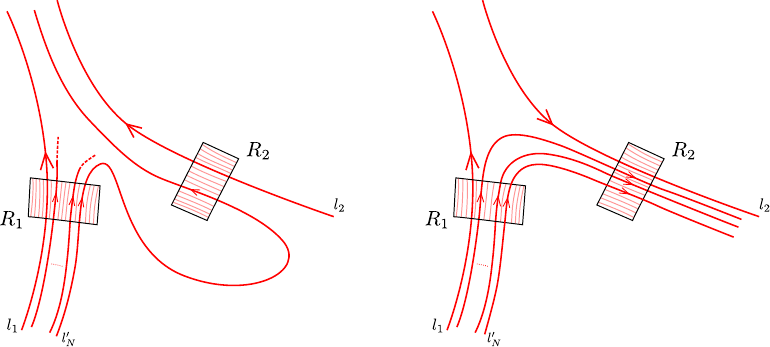}
    \caption{Two ways of orienting $l_1$ and $l_2$. The orientations in the picture on the left are not compatible with $l_1$ and $l_2$ being non-separated.}
    \label{fig: lemma_nonsep_impossible}
\end{figure}

Due to the lemma shown above we may then assume without loss of generality that $l_1, l_2$ are oriented as in Figure \ref{fig: lemma_nonsep_figure}.

\begin{figure}[h]
  \centering
    \includegraphics[width=0.40\textwidth]{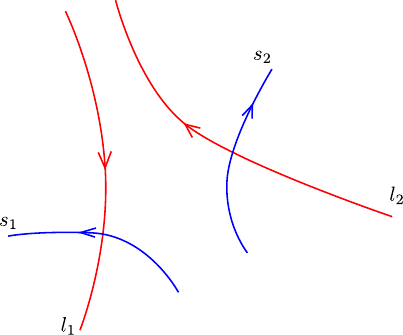}
    \caption{Two non-separated leaves $l_1, l_2$.}
    \label{fig: lemma_nonsep_figure}
\end{figure}

Let $s_1$, $s_2$ be leaves of $\F_2$ such that $s_i \cap l_i \neq \varnothing$. These leaves have orientations coming from the orientation on the foliation $\F_2$. There are two possibilities: either both of the positively oriented rays contained in the $s_i$ and based at $s_i \cap l_i$ are contained in the region between $l_1$ and $l_2$, or the same is true of the analogous negatively oriented rays (no other cases are possible, since the foliations $\F_1, \F_2$ must have the same local product orientation at $s_1 \cap l_1$ and $s_2 \cap l_2$). 

In Figure \ref{fig: lemma_nonsep_figure} we depict the second case. In this case, we can see that $s_2$ is in the back of $s_1$, and $s_1$ is not in the back of $s_2$. The same is true in the other case: recall that it is the orientation of $l_i$ that determines which component of $P \setminus s_1$ is the front of $s_i$. Therefore, reversing the orientations of both of the $s_i$ while leaving those of the $l_i$ unchanged does not change what the front (or back) of each leaf $s_i$ is.


Moreover, any other choice of $s_1', s_2'$ such that $s_i' \cap l_i \neq \varnothing$ can be obtained by continuously varying $s_1$ and $s_2$, so that $s_i$ is in the back of $s_j$ if and only if $s_i'$ is in the back of $s_j'$, for $i,j \in \{ 1,2\}, i\neq j$.
\end{proof}

\begin{definition}\label{order}
    Given two non-separated leaves $l_1, l_2$ of $\F_1$, let $s_1, s_2$ be leaves of $\F_2$ intersecting them. Then, we say that $l_1 < l_2$ (\emph{resp.} $l_1 > l_2$) if $s_1$ is in the back (\emph{resp.} front) of $s_2$.
\end{definition}


We now show that this gives a total order on each cataclysm, and these orders are preserved by the action of $\mathrm{Aut}^+(P)$.



\begin{lemma}\label{orderable_cataclysms}
    Each leaf space $\Lambda_i, i=1,2$ has orderable cataclysms with respect to the action of $\mathrm{Aut}^+(P)$. In particular, this is also true for any subgroup of $\mathrm{Aut}^+(P)$.
\end{lemma}

\begin{proof}
   Let $\mu$ be a cataclysm in $\Lambda_1$. Then $\mu$ is a collection of non-separated $\F_1$-leaves. For any two leaves $l_1, l_2 \in \mu$, because $\mathrm{Aut}^+(P)$ preserves orientations of both foliations, the order on $l_1$ and $l_2$ as described in Definition \ref{order} is well defined by Lemma \ref{non_sep} and is preserved by the action of any element of $\mathrm{Aut}^+(P)$. We need to check that this relation is transitive in order to show that it defines a linear order on $\mu$. Suppose $l_1 < l_2 $ and $l_2 < l_3$, and we want to show that $l_1 < l_3$. 
\begin{figure}[h!]
  \centering
    \includegraphics[width=0.99\textwidth]{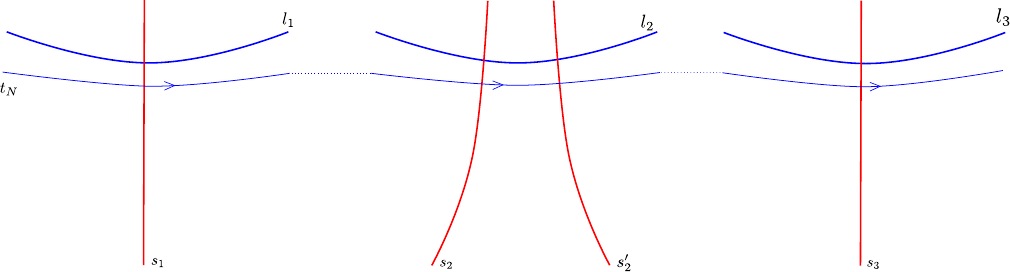}
    \caption{Transitivity of the order on a cataclysm}
    \label{lemma_cataclysms}
\end{figure}
   Since $l_1 < l_2$, there exist $s_1, s_2$ leaves of $\F_2$ intersecting $l_1, l_2$ respectively, such that $s_1$ is in the back of $s_2$. Analogously, there exist $s_2', s_3$ leaves of $\F_2$ intersecting $l_2, l_3$ respectively, such that $s_2'$ is in the back of $s_3$. 
   Let $(t_n)_{n\geq 1}$ be a sequence of $\F_1$-leaves that converges to (possibly among other leaves) $l_1, l_2, l_3$ in $\Lambda_1$. In the plane $P$, for $N$ large enough we must then have that $t_N$ intersects $s_1, s_2, s_2', s_3$. Additionally, the leaves $l_1, l_2, l_3$ must be on the same side of the leaf $t_N$, since otherwise there could not exist a sequence converging to all of them. Therefore, the situation is as shown in Figure \ref{lemma_cataclysms}.
      
   The fact that $s_1, s_2, s_2'$ and $s_3$ intersect $t_N$ together with the fact that $s_1$ is in the back of $s_2$ and $s_2'$ is in the back of $s_3$ implies that (regardless of whether $s_2$ is in the back of $s_2'$ or viceversa) $s_1$ is in the back of $s_3$, meaning that $l_1 < l_3$. This shows that the relation $<$ is transitive, and therefore defines a linear order on $\mu$.\end{proof}

Finally, we can define the desired $\mathrm{Aut}^+(P)$-invariant order:

\begin{proposition}\label{proposition: order_on_ends}
    Let $(P, \F_1, \F_2)$ be a bifoliated plane. Then, there exists an $\mathrm{Aut}^+(P)$-invariant linear order on $\mathrm{Ends}(\Lambda_1) \cup \mathrm{Ends}(\Lambda_2)$.

\end{proposition}

\begin{proof}

First, let $<_1^+$ be an $\mathrm{Aut}^+(P)-$invariant linear order on $\mathrm{Ends}_+(\Lambda_1)$, provided by Proposition \ref{order_zhao} and Lemma \ref{orderable_cataclysms}. Using Proposition \ref{order_zhao} again but with orientations reversed yields an $\mathrm{Aut}^+(P)$-invariant linear order $<_1^-$ on $\mathrm{Ends}_-(\Lambda_1)$. 

Since $\mathrm{Ends}(\Lambda_1) = \mathrm{Ends}_+(\Lambda_1) \sqcup \mathrm{Ends}_-(\Lambda_1)$ and these sets are $\mathrm{Aut}^+(P)$-invariant (the action of $ \mathrm{Aut}^+(P)$ preserves the orientation of both leaf spaces), we can define an order $<_1$ on $\mathrm{Ends}(\Lambda_1)$ by declaring $x <_1 y$ for all $x$ in $\mathrm{Ends}_+(\Lambda_1)$ and all $y $ in $\mathrm{Ends}_-(\Lambda_1)$. When $x,y$ are both in $\mathrm{Ends}_+(\Lambda_1)$ (resp. $\mathrm{Ends}_-(\Lambda_1)$), we can let $x<_1 y$ if and only if $x <_1^+ y$ (resp. $x <_1^- y$).

This is by definition an $\mathrm{Aut}^+(P)$-invariant linear order on $\mathrm{Ends}(\Lambda_1)$. The same procedure shows the existence of an analogous invariant linear order $<_2$ on $\mathrm{Ends}(\Lambda_2)$. Finally, the $\mathrm{Aut}^+(P)$-invariant linear orders on $\mathrm{Ends}(\Lambda_1)$ and $\mathrm{Ends}(\Lambda_2)$ can be combined to yield an $\mathrm{Aut}^+(P)$-invariant order $<$ on $\mathrm{Ends}(\Lambda_1) \cup \mathrm{Ends}(\Lambda_2)$, as we wanted: let  $x < y $ for all $x \in \mathrm{Ends}(\Lambda_1), y \in \mathrm{Ends}(\Lambda_2)$. For $x, y \in \mathrm{Ends}(\Lambda_i)$ and $i = 1,2$, let $x < y$ if and only if $x <_i y$.
\end{proof}

We are now in a position to prove Theorem \ref{theorem:LO_main}.

\begin{proof}[Proof of Theorem \ref{theorem:LO_main}]

We will use an idea of Zhao in section 3.3 of \cite{zhao}, with some modifications. The main step of this idea is to show that the subgroup of $\mathrm{Aut}^+(P)$ acting trivially on the ends of both leaf spaces is left-orderable.

Let $P = (P, \F_1, \F_2)$ be a bifoliated plane, and let $G = \mathrm{Aut}^+(P)$.

By Proposition \ref{proposition: order_on_ends}, there exists a $G$-invariant linear order on $\mathrm{Ends}(\Lambda_1) \cup \mathrm{Ends}(\Lambda_2) $. Let
\begin{align*}
    H&= \{ g\in G : g\cdot x = x \text{ for all } x\in \mathrm{Ends}(\Lambda_1) \cup \mathrm{Ends}(\Lambda_2) \}.
\end{align*}
Then $G/H$ by definition acts faithfully on $\mathrm{Ends}(\Lambda_1) \cup \mathrm{Ends}(\Lambda_2)$ and preserves the linear order on it. This gives a left-order on $G/H$: by fixing a basepoint $x_0 \in \mathrm{Ends}(\Lambda_1) \cup \mathrm{Ends}(\Lambda_2)$, one can define $gH <_{G/H} g'H$ if and only if $g\cdot x_0 < g'\cdot x_0$ (for a complete proof, see Example III of Section 5 in \cite{conrad}).

Having shown the existence of a left-order on $G/H$, the following classical result reduces the proof of Theorem \ref{theorem:LO_main} to showing the existence of a left-order on $H$.

\begin{lemma*}[\cite{conrad}, Section 3.7]
    If $<_{G/H}$ is a left-order on $G/H$ and $<_H$ is a left-order on $H$, then a left-order can be defined on $G$ as follows: given $g,g' $ in $G$, we say $g'<g$ if  one of two conditions holds: either $g'H <_{G/H} gH$, or $g'H = gH$ and $g^{-1}g' <_H e$.
\end{lemma*}

Now, we show that there indeed exists a  left-order on $H$.

By second-countability of the leaf spaces, there exists a countable dense subset $X = \{x_i\}_{i \in \mathcal{I}} \subset \Lambda_1$. For each $i \in \mathcal{I}$, let $r^+_i$ and $r^-_i$ be a positive ray and a negative ray based at $x_i$, respectively. Let $R_i = r^+_i \cup r^-_i$. Then $R_i$ is homeomorphic to $\R$ for all $i \in \mathcal{I}$, so the collection $\{R_i\}_{i \in \mathcal{I}}$ is a countable covering of $\Lambda_1$ by homeomorphic copies of $\R$, and each $R_i$ contains rays representing a positive end of $\Lambda_1$ and a negative end of $\Lambda_1$. Note that a priori, we could have $R_{i_1} = R_{i_2}$ for some $i_1 \neq i_2$, but we may take a subcover with only distinct copies of $\R$. With a slight abuse of notation, we will also use $\{R_i\}_{i \in \mathcal{I}}$ to denote this refined covering by distinct copies of $\R$. Similarly, we can perform an analogous covering by a countable collection of distinct copies of $\R$, $\left\{R_j\right\}_{j \in \mathcal{J}}$, for $\Lambda_2$. \par



Now, let $\mathcal{K} = \mathcal{I} \sqcup \mathcal{J}$. For each $k \in \mathcal{K}$, both the positive end and the negative end represented by rays contained in $R_k$ are fixed by the action of the group $H$. With the choice of our subcovers, simply-connectedness of $\Lambda_1, \Lambda_2$ implies that given any pair consisting of a positive end $x_1$ and a negative end $x_2$ of each $\Lambda_i$, there exists at most one $k_0$ for which $R_{k_0}$ contains rays representing $x_1$ and $x_2$. This means that each $R_k$ must be preserved by the action of $H$. Then for each $k \in \mathcal{K}$, we have a homomorphism $\varphi_k\colon H \to \mathrm{Homeo}_+(R_k)$ associated to the induced action of $H$ on $R_k$. Together, these homomorphisms give rise to a homomorphism $\varphi \colon H \to \prod_{k \in \mathcal{K}} \mathrm{Homeo}_+(R_k)$ by defining the $k$-th factor of $\varphi(h)$ to be $\varphi_k(h)$. The group $\prod_{k \in \mathcal{K}} \mathrm{Homeo}_+(R_k)  \cong  \prod_{k \in \mathcal{K}} \mathrm{Homeo}_+(\R)$ is left-orderable since it is a countable product of left-orderable groups; see \cite{deroin2016groupsordersdynamics}. Thus, to show the left-orderability of $H$ it is enough to show that $\varphi$ is injective.

Let $h \in \mathrm{Ker}(\varphi)$. Then by definition, $h$ acts as the identity on $R_k$ for all $k$. Since $\{R_k\}_{k \in \mathcal{K}}$ covers $\Lambda_1 \cup \Lambda_2$, every point in each $\Lambda_i$ is fixed by $h$. This means that the leaves of the foliations $\F_1, \F_2$ must be preserved by the action of $h$ on $P$. Given any point $x \in P$, since both $\F_1(x)$ and $\F_2(x)$ are preserved by $h$, we must have their unique intersection fixed by $h$ as well, so $h(x) = x$ and we conclude that $h$ is the identity map on $P$. 


Therefore, $\mathrm{Ker}(\varphi) = \left\{\mathrm{id} \right\}$, so $\varphi$ is injective, hence $H$ is left-orderable. From the argument above it then follows that $G$ must be left-orderable.
\end{proof}

\subsection{The case of $G$ acting Anosov-like}\label{subsection:LO_Anosov_like}

Here we prove Corollary \ref{introcor:Anosov_like}, which states that a group $G$ which admits an Anosov-like action on a bifoliated plane $(P, \F_1, \F_2)$ must be realizable as a subgroup of $\mathrm{Homeo}^+(\R)$.

We begin by stating the following lemma. The proof is elementary.
\begin{lemma}\label{lemma:uncountable_stabilizer}
    Let $H$ be an uncountable group acting on a set $C$. Then, if the orbit $H\cdot x$ is countable or finite for some $x\in C$, the stabilizer $\mathrm{Stab}_H(x) < H$ must be uncountable.
\end{lemma}
Now, we state our second lemma before stating and proving the main result in this section. This lemma applies more generally to any foliation of the plane, but we state it for the case of a bifoliated plane.

\begin{lemma}\label{lemma:countable_nonsep}
    Let $(P, \F_1, \F_2)$ be a bifoliated plane with leaf spaces $\Lambda_1, \Lambda_2$. Then, the sets
    \begin{align*}
        \{ x \in \Lambda_i : \text{there exists } x' \in \Lambda_i \text{ s.t. } x, x' \text{ are non-separated}  \}
    \end{align*}
    are at most countable, for $i = 1,2$.
\end{lemma}

\begin{proof}
    Let $\Lambda = \Lambda_1$, and suppose that the statement does not hold, i.e. that there exist uncountably many points in $\Lambda$ that are non-separated from at least one other point of $\Lambda$. In fact, we can assume without loss of generality that there exist uncountably many points in $\Lambda$ that are \emph{positively} non-separated from at least one other point of $\Lambda$. Recall that this means that there exists a positively oriented embedding $\gamma:[0, +\infty) \to \Lambda$ such that as $t\to +\infty$, $\gamma(t)$ converges to (at least) both points. Therefore, ``being positively non-separated'' is an equivalence relation on $\Lambda$.
    
    Let $S\subseteq \Lambda$ be the set of such points.

    Since $\Lambda$ is second-countable and $S$ is uncountable, there exists a point $x_0 \in S$ such that every neighborhood $U \subset \Lambda$ of $x_0$ contains uncountably many points of $S$. Let $U_0$ be an open neighborhood of $x_0$ which is homeomorphic to an open interval and contains an uncountable subset $S_0= \{ x_\alpha \in S : \alpha \in I_0\}$ of $S$. Note that since $U \cong \R$, no two points in $S_0$ can be non-separated from each other (positively or otherwise).
    
    For each $x_\alpha \in S_0$, there exists by definition of $S$ at least one element $x_\alpha' \in S\setminus S_0$ such that $x_\alpha, x_\alpha'$ are positively non-separated. 

    Since being positively non-separated is an equivalence relation, the previous two paragraphs tell us that for all $\alpha, \beta \in I_0$ with $\alpha \neq \beta$, we have that $x_\alpha'$ and $x_{\beta}'$ are not positively non-separated. In particular, for all $\alpha \neq \beta$ we have that $x_\alpha '\neq x_{\beta}'$, implying that the set $\{ x_\alpha': \alpha \in I_0 \}$ is uncountable.

    Since $\Lambda$ is simply connected, for each $\alpha \in I_0$ there exists $R_\alpha \cong [0, +\infty) $ a properly embedded ray in $\Lambda$ based at $x_{\alpha}'$ and not intersecting $U_0$. Again using that $\Lambda$ is simply connected, we can see that for $\alpha \neq \beta$ we must have $R_\alpha \cap R_\beta = \emptyset$. But then $\{\mathrm{int}(R_\alpha) : \alpha \in I_0\} $ is an uncountable family of pairwise disjoint open subsets of $\Lambda$, which is impossible by second countability of $\Lambda$.

\end{proof}

\begin{proposition}\label{prop:countable_or_non_H}
    Let $G$ be a group acting faithfully on a bifoliated plane $(P, \F_1, \F_2)$, such that stabilizers of leaves are countable. Then, at least one of the following must be true.
    \begin{enumerate}
        \item $G$ is countable.
        \item The leaf spaces $\Lambda_1, \Lambda_2$ are Hausdorff.
    \end{enumerate}

    In particular, an uncountable group can never act faithfully with countable leaf stabilizers on a bifoliated plane with non-Hausdorff leaf spaces.
        
\end{proposition}

\begin{proof}
    Suppose that $G$ is uncountable, and we show that $\Lambda_1, \Lambda_2$ are non-Hausdorff. 
    
    Let $\Lambda = \Lambda_1$ (the proof for $\Lambda_2$ is analogous), and let $\mathcal{C}$ be the set of all cataclysms on $\Lambda$ that consist of at least two non-separated points. We will show that $\mathcal{C}$ is empty. By Lemma \ref{lemma:countable_nonsep}, this is a countable (or finite) set. 
    
    Since $G$ acts on $\mathcal{C}$, if this set is nonempty then Lemma \ref{lemma:uncountable_stabilizer} tells us that there exists a cataclysm $C_0 \in \mathcal{C}$ with uncountable stabilizer $H_0 < G$ (in fact, the Lemma tells us that this holds for all cataclysms in $\mathcal{C}$).

    Now, again by Lemma \ref{lemma:countable_nonsep}, we know that $C_0$ must be countable (or finite). But then, since $H_0 = \mathrm{Stab}_H(C_0)$ acts on $C_0$ and is uncountable, Lemma \ref{lemma:uncountable_stabilizer} again tells us that there exists a point $x \in C_0$ with uncountable stabilizer $H_x < H_0$. That is, there exists a leaf of $\F_1$ with uncountable stabilizer, contradicting our hypothesis. We conclude that $\mathcal{C}$ must be empty, i.e. $\Lambda$ is Hausdorff.
    
\end{proof}

\begin{corollary}\label{corollary:LO_countable_stab_realized}
    Let $G$ be a group acting faithfully on a bifoliated plane, with countable leaf stabilizers. Then, $G$ can be realized as a subgroup of $\mathrm{Homeo}^+(\R)$. 
\end{corollary}
\begin{proof}
    By Proposition \ref{prop:countable_or_non_H}, we know that either $G$ is countable or the leaf spaces $\Lambda_1, \Lambda_2$ are Hausdorff. If $G$ is countable, then as we have mentioned earlier, left orderability (implied by Theorem \ref{theorem:LO_main}) together with an application of Theorem \ref{theorem:LO_countable_ord_implies_homeo} imply that $G$ can be realized as a subgroup of $\mathrm{Homeo}^+(\R)$. 
    
    If the leaf spaces $\Lambda_1, \Lambda_2$ are Hausdorff, then they are homeomorphic to $\R$. The action of $G$ on $\Lambda_1 \cup \Lambda_2$ must be faithful since $G$ acts faithfully on $P$. We can embed $\Lambda_1$ and $\Lambda_2$ as disjoint open intervals in the real line, yielding a faithful action of $G$ on $\mathbb{R}$ by orientation preserving homeomorphisms (the action leaves these intervals invariant). This realizes $G$ as a subgroup of $\mathrm{Homeo}^+(\R)$.
\end{proof}

\begin{corollary}\label{corollary:LO_Anosov_like_uncount_skew_or_triv}
    Let $(P, \F^+, \F^-)$ be a bifoliated plane equipped with an Anosov-like action by an uncountable group $G$. Then, $(P, \F^+, \F^-)$ is either skew or trivial.
\end{corollary}

\begin{proof}
    The proof uses Proposition \ref{prop:countable_or_non_H} above combined with the following two results: by Lemma 2.4 in \cite{barthelme2022orbit}, the stabilizer of any leaf must act with a common fixed point, and by Theorem 5.1 in \cite{barthelme2024nontransitivepseudoanosovflows}, point stabilizers are virtually cyclic in bifoliated planes with non-Hausdorff leaf spaces.
    
    This means that if one of the leaf spaces is non-Hausdorff (in the Anosov-like case it is known that this implies that both must be, but we do not need this) then stabilizers of leaves must be countable.

    Therefore, we can apply Proposition \ref{prop:countable_or_non_H} above to conclude that $\Lambda^+$ and $\Lambda^-$ must be Hausdorff. Theorem 2.16 in \cite{barthelme2022orbit} shows that the only bifoliated planes (with non-singular foliations) with Hausdorff leaf spaces that support an Anosov-like action are the trivial and the skew planes. 
\end{proof}

\begin{proof}[Proof of Corollary \ref{introcor:Anosov_like}]
    If $G$ is countable, then again the left orderability of $G$ shown in Theorem \ref{theorem:LO_main} and an application of Theorem \ref{theorem:LO_countable_ord_implies_homeo} tell us that $G$ can be realized as a subgroup of $\mathrm{Homeo}^+(\R)$.

    If $G$ is uncountable, Corollary \ref{corollary:LO_Anosov_like_uncount_skew_or_triv} tells us that $(P, \F^+, \F^-)$ is skew or trivial. Therefore, the same argument as in the proof of Corollary \ref{corollary:LO_countable_stab_realized} shows that $G$ can be realized as a subgroup of $\Homeo^+(\R)$.
\end{proof}

\section{Realizing ends on the circle at infinity}\label{section:LO_realizing_ends}

In this section, we study the correspondence between ends of the leaf spaces $\Lambda_1$, $\Lambda_2$ and points in the circle at infinity $\partial P$. Before doing this, we briefly discuss the sets $\mathrm{Ends}(\F_i) \subset \partial P$, the points of the circle at infinity which are endpoints of some leaf of $\F_i$.

There is a natural map $E \to \partial P$, assigning an ideal point to each end of a leaf. In general, the image of this map will be dense in, but not equal to, $\partial P$. However, we can say more: except in the case of certain bifoliated planes containing \emph{infinite product regions}, the image of each of the sets $\mathrm{Ends}(\F_1), \mathrm{Ends}(\F_2) \subset E$ is dense in $\partial P$.

\begin{definition}
An \emph{infinite product region} in $(P, \F_1, \F_2)$ is an open subset $U \subset P$ such that for $i, j = 1, 2$ and $i \neq j$,
\begin{enumerate}
    \item the boundary of $U$ consists of a compact segment of an $\F_i$-leaf $l$ and two $\F_j$-leaves, $j = 3 - i$, $f_1$, $f_2$ that intersect $l$;
    \item for each $x \in U$, $\F_i(x)$ intersects both $f_1$ and $f_2$.
\end{enumerate}If $U$ is an infinite product region, we say that \emph{$U$ is based in $\F_i$}, $i = 1, 2$, if the compact segment of $\partial U$ inside $P$ is on an $\F_i$-leaf.
\end{definition}

\begin{proposition}
    Let $P = (P, \F_1, \F_2)$ be a bifoliated plane. For $i = 1, 2$, $P$ contains no infinite product region based in $\F_i$ if and only if $\mathrm{Ends}(\F_i)$ is dense in $\partial P$.
\end{proposition}

\begin{proof}
    We start with the \emph{if} direction. Suppose that $P$ contains an infinite product region $U$ based in $\F_i$. Let $I = \partial \overbar{U} \cap \partial P$, then by definition, $I \cap \Ends(\F_i) = \varnothing$. Thus $\Ends(\F_i)$ is not dense in $\partial P$. \par
    
    For the \emph{only if} direction, suppose that $P$ does not contain any infinite product region based in $\F_i$. We show that given any interval $I \subset \partial P$, there exists $x \in P$ such that $\partial \F_i(x) \cap I \neq \varnothing$. Suppose such $x$ does not exist. Pick two distinct $\F_j$-leaves $l_1, l_2$ such that each has an ideal point in $I$. Then there exists $x^\prime \in P$ such that $\F_i(x^\prime)$ either bounds an infinite product region together with $l_1$ and $l_2$ or intersects one of them twice. Neither is allowed here, so by contradiction we get the density of $\Ends (\F_i)$.
\end{proof}

So far in this section, we have made no mention of the orientations on the foliations $\F_1$ and $\F_2$ of a bifoliated plane. In the presence of such orientations, it will be useful to further divide elements of each $\mathrm{Ends}(\F_i)$ into positive and negative ideal points.
\begin{definition}
    Given a leaf $l$ of $\F_i$ equipped with the orientation induced from the foliation, we say an ideal point $x \in \mathrm{Ends}(\F_i)$ of $l$ is a \emph{positive ideal point of} $l$ if $x = \lim_{t \to +\infty} \gamma(t)$, where $\gamma \colon [0, +\infty) \to l$ is an orientation preserving and proper continuous map. 
    
    Otherwise (i.e.~if the same holds with $\gamma$ orientation reversing), we say $x$ is a \emph{negative ideal point} of $l$.

\end{definition}

We will also use the notion of \emph{quadrant} to describe relative positions of points.

\begin{definition}
    Let $x \in P = (P, \F_1, \F_2)$. We call each connected component of $P \setminus (\F_1(x) \cup \F_2(x))$ a \emph{quadrant} given by $x$. Let $\psi: P \to \R^2$ be an orientation-preserving local homeomorphism that sends $x$ to the origin, and sends $\F_1(x)$, $\F_2(x)$ to the oriented $x$-axis and $y$-axis, respectively. The pre-images of the first, second, third, and fourth quadrants in $\R^2$ under $\psi$ are said to be the \emph{top-left}, \emph{top-right}, \emph{bottom-left}, and \emph{bottom-right} quadrants of $x$, respectively.
\end{definition}

Now, we build a natural correspondence between the set of ends of the leaf spaces of the bifoliation and certain subset of the boundary circle at infinity. We present the result here in the form of Proposition \ref{proposition: mapping to ends}. \par

Recall that the leaf space of a foliation $\F_i$ is defined to be $\Lambda_i = P / \F_i$, equipped with the quotient topology. 

We use $q_i \colon P \to \Lambda_i$ to denote the projection map from $P$ to the leaf space $\Lambda_i$ of the foliation $\F_i$.

To provide some intuition, recall that each point in the leaf space $\Lambda_i$ represents a leaf in the plane under the projection map $q_i$. Then the pre-image of each ray $r_i \in \Lambda_i$ under $q_i$ is an unbounded monotone sequence of $\F_i$-leaves $\{l_\alpha\}_{\alpha \in \mathcal{A}}$. We can then represent $\{l_\alpha\}_{\alpha \in \mathcal{A}}$ with a transverse ray $\tau \subset P$ that intersects every leaf in this sequence. Note that $\tau$ is not necessarily a half-leaf of $\F_j$, $j = 3 - i$. Moreover, $\tau$ cannot be contained in any compact subset of the plane, and we will show that its unbounded side must accumulate to a single point on the boundary circle at infinity. \par

We start the construction by a lemma that characterizes the limit set of a ray which has no accumulation inside a topological plane .
     
\begin{lemma}\label{lemma: curve connected limit set}
    Let $P$ be a topological plane. If $\tau \subset P$ is a ray with no accumulation inside $P$, then $ \partial^+ \tau = \lim_{t \to +\infty}\tau[t, +\infty)$ is a connected subset of $\partial P \cong \mathbb{S}^1$.
\end{lemma}

\begin{proof}
    Suppose that $\partial^+ \tau$ is disconnected in $\partial P$. Since $\tau$ is a connected ray, $\partial^+ \tau$ is connected in $P \cup \partial P$. Then $\partial^+ \tau \cap P \neq \varnothing$. But this contradicts that $\tau$ has no accumulation inside $P$.
\end{proof}

For each foliation $\F_i$, $i = 1, 2$, any interval $I$ in the leaf spaces $\Lambda_i$ is the image under the projection $q_i$ of some curve $\tau$ transverse to $\F_i$: given any point $x$ on a transverse curve $\tau$ to $\F_i$, $q_i(\F_i(x))$ a point in $\Lambda_i$, so $q_i(\tau)$ is connected in $\Lambda_i$ and contains more than one point, hence an interval. Since ends in each leaf space are equivalent classes of rays, it is only natural to study \emph{transverse rays} to a given foliation. \par

Unless otherwise mentioned, for the rest of this chapter we assume transverse rays to be parameterized by $[t, +\infty)$, i.e.~embeddings of intervals of the form $[t, +\infty)$. Since we will be dealing with limits and accumulation sets frequently, to avoid excessive repeating of the term, we make the following convention by slightly abusing the notation $\partial$:

\begin{convention}
If $\tau \subset P$ is a transverse ray to $\F_i$, then $\partial^+ \tau$ denotes the accumulation of $\tau$ in its unbounded direction.
\end{convention}


We describe some rules of how transverse curves in $P$ must behave. 

\begin{lemma}\label{lemma: transversal no double intersection}
    If $\tau$ is a curve in $P$ transverse to $\F_i$, then $\tau$ can only intersect an $\F_i$-leaf at most once. Moreover, if $\tau$ has an ideal point on $\partial P$ and intersects an $\F_i$-leaf $l$ in $P$, then $\tau$ and $l$ cannot share a common ideal point on $\partial P$.
\end{lemma}

\begin{proof}
    Suppose that $\tau$ intersects some $\F_i$-leaf $l$ at two distinct points $x_1$ and $x_2$. Then the transverse orientations on any sufficiently small neighbourhoods of $x_1$ and $x_2$ must be opposite. But this implies that there exists $x_3 \in \tau$ between $x_1$ and $x_2$ at which $\tau$ is tangent to $\F_i$, contradicting that $\tau$ is transverse to $\F_i$. \par

    Now suppose that $\tau \cap l = \tau(t_0)$ and $\partial^+ \tau \cap \partial l = \xi$. Since any $\F_i$-leaf can only intersect $\tau$ at most once and two distinct $\F_i$-leaves cannot intersect each other, for all $x \in \tau(t_0, +\infty)$, we have $\xi \in \partial \F_i(x)$. But $\{F_i(x)\}_{x \in \tau(t_0, +\infty)}$ is an uncountable collection of leaves, this is impossible by Lemma 3.2 in \cite{bonatti}.
\end{proof}

The next lemma shows that any transverse ray must land at some point on the circle at infinity.

\begin{lemma}\label{lemma: transverse ray lands on the boundary}
    If $\tau \subset P$ is a transverse ray to $\F_i$, then $\tau$ cannot accumulate inside $P$, meaning that there does not exist any $x \in P \setminus \tau$ such that $\tau$ intersects every neighbourhood of $x$. Furthermore, $\partial^+ \tau$ must be a point on $\partial P$.
\end{lemma}

\begin{proof}
    Suppose that there exists $x \in P \setminus \tau$ that is accumulated by $\tau$. Take a small neighborhood $U$ of $x$, which is foliated by leaves of $\F_i$. Then $\tau$ necessarily intersects $\F_i(x)$ infinity many times as it is a transverse ray, which contradicts Lemma \ref{lemma: transversal no double intersection}. \par
    By Lemma \ref{lemma: curve connected limit set}, $\partial^+ \tau$ is a connected subset of $\partial P \cong \mathbb{S}^1$ --- so it is either the entire $\mathbb{S}^1$, an interval or a single point. As a first case, suppose that $\partial^+ \tau = \mathbb{S}^1$. Pick a leaf $l \in \F_i$, which separates $P$ into two half-planes $H_1$ and $H_2$, and assume that $\tau(0) \in H_1$. Then there exists $t_1, t_2, t_3 \in \R$ such that $t_1 < t_2 < t_3$ and $\tau[0, t_1) \subset H_1, \tau(t_1, t_2) \subset H_2, \tau(t_2, t_3) \subset H_1$. But this implies that $\tau(t_1) \in l$ and $\tau(t_2) \in l$, contradicting Lemma \ref{lemma: transversal no double intersection}. For the second case, suppose that $\partial^+ \tau$ is some interval $I = [\gamma_1, \gamma_2] \subset \partial P$. Pick a leaf $l^\prime \in \F_i$ with at least one ideal point in $\mathring{I}$ and let $\xi$ be such ideal point. Then $\tau \cap l^\prime \neq \varnothing$. Let $z = \tau(t_0) \in \tau \cap l$. By Lemma \ref{lemma: transversal no double intersection}, $\tau(t_0, +\infty)$ is contained in exactly one of the half-planes separated by $l^\prime$, so we have either $\partial^+ \tau \subset [\gamma_1, \xi]$ or $\partial^+ \tau \subset [\xi, \gamma_2]$. But neither satisfies that $\partial^+ \tau = I$, so we have a contradiction. Therefore, we can only be in the third case, that is, $\partial^+ \tau$ must be a point on $\partial P$.
\end{proof}

It follows that every transverse ray has an ideal point. For this, we say that a transverse ray \emph{lands at} some point on the boundary at infinity. \par

Before stating and proving our main proposition, we build the subset on the boundary circle at infinity whose elements represent ends of leaf spaces. Intuitively, given a transverse ray $\tau$ to $\F_i$, $\tau$ landing at some point on $\partial P$ is a necessary condition for $q_i(\tau)$ to be a ray in $\Lambda_i$. But it is not sufficient. We first describe configurations on $\partial P$ that prevents us from getting rays in $\Lambda_i$.

\begin{lemma}\label{lemma: R_i cannot be rays}
    Let $\xi \in \partial P$ be an ideal point of some $\F_i$-leaf or an accumulation point of ideal points of non-separated $\F_i$-leaves, and let $\tau \subset P$ be a transverse ray to $\F_i$ landing at $\xi$. Then $q_i(\tau)$ is bounded in $\Lambda_i$. In particular, since $q_i(\tau)$ is connected, it must be an interval and thus cannot be a ray. 
\end{lemma}

\begin{proof}
    Let $x \in \tau$. If there exists some $l \in \F_i$ such that $l$ has an ideal point at $\xi$, then $q_i(\F_i(x))$ is bounded from above by $q_i(l)$. If there exists a collection of non-separated leaves $L \subset \F_i$ such that its elements have ideal points accumulating to $\xi$, then $q_i(\F_i(x))$ is bounded from above by any point contained in the cataclysm $q_i(L)$. Since this is true for any $x$, we conclude that $q_i(\tau)$ is bounded.
\end{proof}


\begin{definition}
   A \emph{gap} in $\F_i$ is a connected component of $\lim\limits_{n \to \infty} l_n \cap \partial P$ that has non-empty interior in $\partial P$, where $\{l_n\}_{n \in \N} \subset \F_i$ is a sequence of leaves such that $\lim\limits_{n \to \infty} l_n$ contains at least one $\F_i$-leaf. Here the limit of $\{ l_n \}_{n\in \N}$ is taken in $P \cup \partial P$.
\end{definition}

\begin{remark}
    If $\mathcal{G}$ is a gap in $\F_i$, then by definition, $\mathring{\mathcal{G}}$ does not contain any ideal point of $\F_i$-leaves. Moreover, at least one endpoint of $\partial \mathcal{G}$ is an ideal point of some $\F_i$-leaf or an accumulation point of ideal points of non-separated $\F_i$-leaves, otherwise $\lim_{n \to \infty} l_n$ cannot contain any $\F_i$-leaf.
\end{remark}

Note that with this definition, for each $\F_i$, there are two types of infinite product regions: ones whose boundary component in $\partial P$ is contained in a gap in $\F_i$, and ones whose boundary component in $\partial P$ is not part of any gap. Since these types will be crucial to our construction, we make the definition formal here:

\begin{definition}
     Let $U$ be an infinite product region based in $\F_i$, and let $I = \partial \overbar{U} \cap \partial P$. If $I$ is contained in some gap in $\F_i$, then $U$ is said to be a \emph{gap product region based in $\F_i$}. Otherwise, $U$ is said to be a \emph{standard product region based in $\F_i$}.
\end{definition}

\begin{figure}[h]
    \centering
    \includegraphics[scale = 0.85]{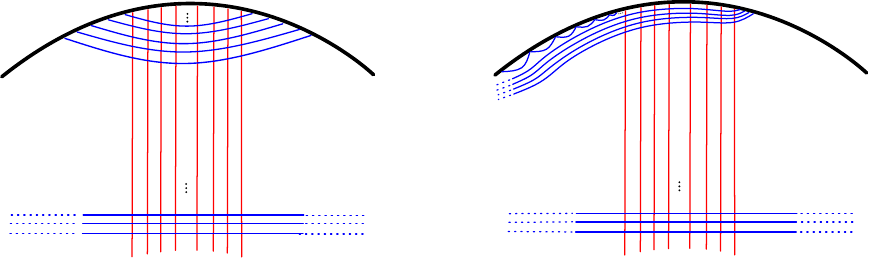}
    \caption{On the left is a standard product region based in the \emph{horizontal} foliation, and on the right is an example of a gap product region based in the \emph{horizontal} foliation with accumulating non-separated \emph{horizontal} leaves on one side.}
    \label{fig: types of product regions}
\end{figure}

\begin{lemma}\label{lemma: ray cannot land in a gap}
    The projection of any transverse ray to $\F_i$ landing in a gap in $\F_i$ is bounded in $\Lambda_i$.
\end{lemma}

\begin{proof}
    Let $\mathcal{G}$ be a gap in $\F_i$. For any two transverse rays $\tau_1, \tau_2$ to $\F_i$ that land in $\mathcal{G}$, there exists $x_1 \in \tau_1$ and $x_2 \in \tau_2$ such that $\F_i(x_1) = \F_i(x_2)$. Then with reparametrization, there exists $t_0 \in (0, +\infty)$ such that $\tau_1[t_0, +\infty) = \tau_2[t_0, +\infty)$, so $q_i(\tau_1)$ and $q_i(\tau_2)$ are either both unbounded or both bounded. \par
    
    To prove the lemma, first, let $\tau \subset P$ be a transverse ray to $\F_i$ landing on some point in $\partial \mathcal{G}$ that is an ideal point of some $\F_i$-leaf or an accumulation point of ideal points of non-separated $\F_i$-leaves. Then $q_i(\tau)$ is bounded by Lemma \ref{lemma: R_i cannot be rays}. Now, if $\tau^\prime \subset P$ is a transverse ray to $\F_i$ that lands at any point in $\mathcal{G}$, then by the first paragraph of this proof, $q_i(\tau^\prime)$ must also be bounded.
\end{proof}

For $i = 1, 2$, let
\begin{align*}
    E_i = \{\xi \in \partial P \mid \xi \textrm{ is an ideal point of some transverse ray to } \F_i\}, 
\end{align*}
let
\begin{align*}
    R_i = \{\eta \in \partial P \mid \eta &\textrm{ is neither an ideal point of any } \F_i \textrm{ leaf} \\ &\textrm{ nor accumulation of ideal points of non-separated } \F_i \textrm{-leaves}\},
\end{align*}
and let
\begin{align*}
    N_i = \partial P \backslash \{\textrm{gaps in $\F_i$}\}.
\end{align*}
 
Let $S_i = E_i \cap R_i \cap N_i$, then we have the following results:

\begin{lemma}\label{lemma: transversals are rays}
    The projection of any transverse ray landing at some point in $S_i$ is a ray in $\Lambda_i$.
\end{lemma}

\begin{proof}
    Let $\xi \in S_i$ and let $\tau \subset P$ be a transverse ray to $\F_i$ that lands at $\xi$ with an initial point $x_0 \in P$. We show that $q_i(\tau)$ is unbounded in $\Lambda_i$, then the statement follows since $q_i(\tau)$ is connected. Suppose otherwise, and let $a = q_i(x_0) \in \Lambda_i$. Then there exists $b \in \Lambda_i$ such that $q_i(\tau) = [a, b)$. Then the pre-image of $b$ under $q_i$ is an $\F_i$-leaf that either has an ideal point at $\xi$ or is contained in some non-separated collection of $F_i$-leaves whose ideal points accumulate to $\xi$. But either of these means that $\xi \notin R_i$, and we have a contradiction.
\end{proof}

\begin{figure}[h]
    \labellist
    \small \hair 2pt
    \pinlabel $U$ at 122 56
    \pinlabel $\partial P$ at -13 81
    \endlabellist
    \centering
    \includegraphics[scale = 0.75]{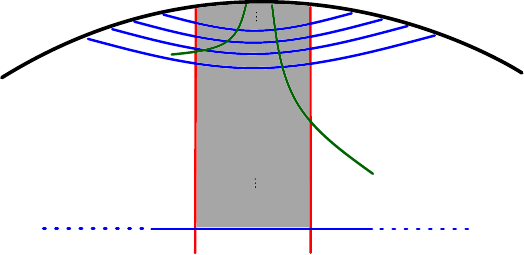}
    \caption{An end realized by the boundary of an infinite product region on $\partial P$}
    \label{fig: ends on boundary}
\end{figure}

\begin{lemma}\label{lemma: transversals landing at the same point represent an end}     
    If $\tau_1$ and $\tau_2$ are two distinct transverse rays to $\F_i$ that represent ends in $\Ends(\Lambda_i)$ and land at the same ideal point $\xi \in S_i$, then their projections are rays in the same equivalence class in $\Lambda_i$, i.e.~they represent the same end in $\Lambda_i$. Moreover, if $P$ contains a standard product region $U$ based in $\F_i$, then all transverse rays landing in $I = \partial \overbar{U} \cap \partial P$ represent the same end in $\Lambda_i$.
\end{lemma} 

\begin{proof}
    Let $\tau_1$ and $\tau_2$ be two transverse rays to $\F_i$, both landing at $\xi \in S_i$, with different initial points $x_1$ and $x_2$, and let $r_1 = q_i(\tau_1)$ and $r_2 = q_i(\tau_2)$ be their corresponding rays in $\Lambda_i$. We breakdown the proof of the first statement into 2 cases. \par
    
    \noindent\textbf{Case 1:} If either $\F_i(x_1) \cap \tau_2 \neq \varnothing$ or $\F_i(x_2) \cap \tau_1 \neq \varnothing$, let $s$ be a point of such intersection, then $r_1$ and $r_2$ agree after passing $q_i(s)$, so $r_1$ and $r_2$ represent the same end. \par

    \noindent\textbf{Case 2:} Suppose that $\F_i(x_1) \cap \tau_2 = \varnothing$ and $\F_i(x_2) \cap \tau_1 = \varnothing$. Without loss of generality, let $V$ be the region bounded by $\tau_1, \tau_2$, the positive half-leaf of $\F_i(x_1)$, the negative half-leaf of $\F_i(x_2)$ and $\partial P$ (Figure \ref{fig: transversals landing at the same point}). Then either there exists $x_0 \in \tau_1$ such that the positive half-leaf of $\F_i(x_0)$ intersects $\tau_2$, or for all $x \in \tau_1$ the positive half-leaf of $\F_i(x)$ has an ideal point in $\partial \overbar{V} \cap \partial P$. Suppose we have the first case. With a slight abuse of notation, by considering $\tau_1$ with the initial point $x_0$ instead of $x_1$, we are back in \textbf{Case 1}. Now suppose that we have the second case. Then for all $y \in \tau_2$, the negative half-leaf of $\F_i(y)$ cannot intersect $\tau_1$ and thus must have an ideal point in $\partial \overbar{V} \cap \partial P$. Then as $x$ and $y$ converge to $\xi$, $\F_i(x)$ and $\F_i(y)$ must converge to a single $\F_i$-leaf that has an ideal point at $\xi$, contradicting $\xi \in S_i$. Since the second case cannot happen, we have that $[r_1] = [r_2]$. \par

    Now, suppose that $P$ contains a standard product region $U$ based in $\F_i$. Let $I = \partial \overbar{U} \cap \partial P$. By definition, $I$ is not a gap in $\F_i$, and it does not contain any point that is an ideal point of a $\F_i$-leaf. Moreover, for any $\xi \in I$, there exists a unique $\F_j$-leaf, $j = 3 - i$, $f$ such that $f \cap U \neq \varnothing$ and $\xi \in \partial f$. Hence, $I \subset S_i$, and every point in $I$ represents an end in $\Lambda_i$ by Lemma \ref{lemma: transversals are rays}. Since any two $\F_j$-leaves landing in $I$ intersect the same set of $\F_i$-leaves after entering $U$, all $\F_j$-leaves landing in $I$ represent the same equivalent class of rays in $\Lambda_i$. Since any transverse ray landing in $I$ shares an ideal point with some $\F_j$-leaf, by the first statement of this lemma, which we have proved, all transverse rays landing in $I$ represent the same end in $\Lambda_i$.
\end{proof}

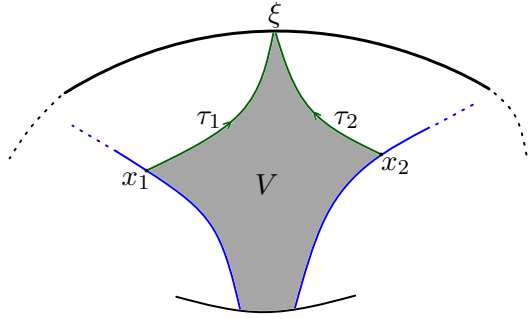
\begin{figure}[h]
    \labellist
    \small \hair 2pt
    \pinlabel $V$ at 152 76
    \pinlabel $x_1$ at 75 78
    \pinlabel $x_2$ at 228 87
    \pinlabel $\tau_1$ at 118 114
    \pinlabel $\tau_2$ at 199 114
    \pinlabel $\xi$ at 157 176
    \endlabellist
    \centering
    \vspace{3mm}\includegraphics[scale = 0.65]{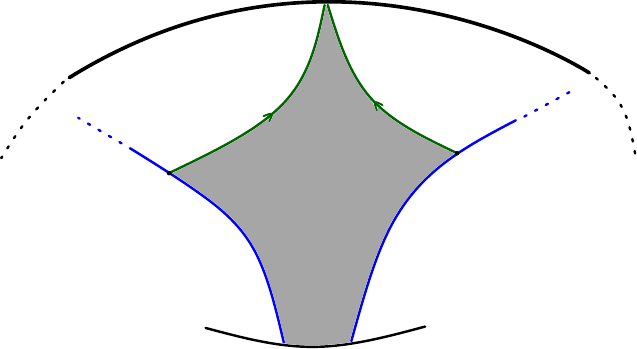}
    \caption{Two transversals landing at the same point}
    \label{fig: transversals landing at the same point}
\end{figure}

\begin{corollary}\label{corollary: increasing product regions}
    If $\{I_\alpha\}_{\alpha \in \mathcal{A}} \subset \partial P$ is an increasing (for inclusion) sequence of intervals such that each $I_\alpha$ is the boundary component of a standard product region $U_\alpha$ based in $\F_i$, then $I = \bigcup_{\alpha \in \mathcal{A}} I_\alpha$ represents the same end in $\Lambda_i$ as any $I_\alpha$.
\end{corollary}

\begin{proof}
    Let $\xi, \xi^\prime \in \partial P$ so that $I = [\xi, \xi^\prime]$. If $I_\alpha \supset I_\beta$, then $I_\alpha$ and $I_\beta$ represent the same end by Lemma \ref{lemma: transversals landing at the same point represent an end}. So if $I$ is a boundary component of a standard product region, then $\xi$ and $\xi^\prime$ are ideal points of $\F_j$-leaves, $j = 3 - i$, and the statement follows since $I$ contains all $I_\alpha$, $\alpha \in \mathcal{A}$. \par
    
    Now suppose that $I$ does not bound any standard product region. Then $\xi$ and $\xi^\prime$ are accumulated by ideal points of $\F_j$-leaves that bound standard product regions in $\{U_\alpha\}_{\alpha \in \mathcal{A}}$. We need to show that both $\xi$ and $\xi^\prime$ are elements of $S_i$. Then we get that $I \subset S_i$, and the proof will be completed by continuity and Lemma \ref{lemma: transversals landing at the same point represent an end}. \par
    
    In fact, by symmetry, we will only need to show that $\xi$ is a point in $S_i$. Let $U = \bigcup_{\alpha \in \mathcal{A}} U_\alpha$. Let $x \in U$. Moreover, without loss of generality, fix a transverse orientation so that $\xi$ is in the top-left quadrant of $x$. Pick a sequence of points $\{x_n\}_{n \in \N} \subset U$ that converge to $\xi$ such that $x_1 = x$ and $x_{n+1}$ is in the top-left quadrant of $x_n$ for all $n \in \N$. Then we get a transverse ray to $\F_i$ that lands at $\xi$ by connecting $x_n$ and $x_{n+1}$ with a straight line for all $n \in \N$. Hence, $\xi \in E_i$. By construction, $\xi$ must be in $R_i$ because otherwise it implies that $I$ is a gap in $\F_i$, but $I$ is the union of boundary components of standard product regions. For the same reason, $\xi$ must be in $N_i$. Therefore, $\xi \in S_i$, and we are done.
\end{proof}

We are now ready to complete the construction. Since a group action on $P$, preserving foliations and their orientations, induces an action on $\partial P$, with the following proposition we will have the tool to associate group actions on $P$ with group actions on the set of ends of the leaf spaces --- this is the crucial step in proving Theorem \ref{theorem: faithful}.

\begin{proposition}\label{proposition: mapping to ends}
    Let $\xi \in S_i$. For $i = 1, 2$, let $\tau_\xi$ be any transverse ray to $\F_i$ that has an ideal point at $\xi$, and let $r_\xi = q_i(\tau_\xi) \in \Lambda_i$. Then the map $\Phi_i \colon S_i \to \Ends(\Lambda_i)$ such that $\Phi_i(\xi) = [r_\xi]$ is well-defined and surjective, and the pre-image of any end in $\Ends(\Lambda_i)$ is either a single point, or a closed interval $I \in \partial P$ such that any two $\F_j$-leaves, $j = 3 - i$, ending in $\mathring{I}$ bound a standard product region based in $\F_i$.
    
\end{proposition}

\begin{proof}
    The map $\Phi_i$ is well-defined by Lemma \ref{lemma: transversals are rays} and Lemma \ref{lemma: transversals landing at the same point represent an end}. \par
    
    We first check the surjectivity of $\Phi_i$. Let $[r] \in \Ends(\Lambda_i)$. Then in $\Lambda_i$, $r$ is a ray with some initial point $x_0$. The pre-image of $r$ under $q_i$ in $P$ is saturated by a collection of $\F_i$-leaves, $\{l_\alpha\}_{\alpha \in \mathcal{A}}$, with an ``initial'' leaf $l_{\alpha_0}$ whose projection in $\Lambda_i$ is $x_0$. We can find some transverse ray $\tau \subset P$ to $\F_i$ that starts at some point on $l_{\alpha_0}$ and intersects every leaf of $\{l_\alpha\}_{\alpha \in \mathcal{A}}$ so that $q_i(\tau) = r$. By Lemma \ref{lemma: transverse ray lands on the boundary}, it lands at some point $\eta \in \partial P$. By Lemma \ref{lemma: ray cannot land in a gap}, we know that $\eta$ is not contained in any gap. We are left to check that $\eta \in R_i$, and this is true by Lemma \ref{lemma: R_i cannot be rays}. \par

    To prove the last part of the proposition, we show that if $[r]$ is an end of $\Lambda_i$ such that $\Phi_i^{-1}([r])$ contains distinct points $\xi, \xi^\prime$, then there exists an interval $I \in \partial P$ such that $\xi, \xi^\prime \in I$ and $\Phi_i^{-1}([r]) = I$. Let $\tau$ and $\tau^\prime$ be transverse rays to $\F_i$ that lands at $\xi$ and $\xi^\prime$, respectively. Then by Lemma \ref{lemma: ray cannot land in a gap}, $\xi$ and $\xi^\prime$ are not in the boundary of any gap product region based in $\F_i$. Let $r_\xi = q_i(\tau_\xi)$ and $r_{\xi^\prime} = q_i(\tau_{\xi^\prime})$. Since $r_\xi \sim r_{\xi^\prime}$, then up to switching $r_\xi$ and $r_{\xi^\prime}$ there exists $t_0 \in \mathbb{R}$ such that $r_\xi[t_0, +\infty) = r_{\xi^\prime}[0, +\infty)$. Let $l_0 \subset P$ be the pre-image of $r_\xi(t_0)$ under $q_i$. Then $l_0$, as well as every $\F_i$-leaf whose projection under $q_i$ is contained in $r_\xi(t_0, +\infty)$, intersect both $\tau$ and $\tau^\prime$. If both $\tau$ and $\tau^\prime$ are $\F_j$-leaves, then $l_0, \tau$, and $\tau^\prime$ form a standard product region whose boundary component in $\partial P$ is $J = [\xi, \xi^\prime]$, and $J \subset \Phi_i^{-1}([r])$ by Lemma \ref{lemma: transversals landing at the same point represent an end}. Otherwise, since ideal points of leaves are dense in $\partial P$, we can find sequences of $\F_j$-leaves, $\{f_n\}_{n\in \N}$ and $\{f^\prime_n\}_{n\in \N}$, landing in $J$ whose ideal points converge to $\xi$ and $\xi^\prime$, respectively. Since $r_\xi[t_0, +\infty) = r_{\xi^\prime}[0, +\infty)$, for each $n$, we can find some $l_n \in \F_i$ with $q_i(l_n) \in r_\xi(t_0, +\infty)$ such that $f_n$, $f^\prime_n$, and $l_n$ form a standard product region whose boundary component in $\partial P$ is the closed interval $J_n$. By construction, $\{J_n\}_{n \in \N}$ is an increasing (for inclusion) sequence of intervals, thus $J = [\xi, \xi^\prime] = \bigcup_{n \in \N} J_n \subset \Phi^{-1}_i([r])$ by Corollary \ref{corollary: increasing product regions}. Following from this, if there exists $\xi^\star$ outside of $J$ such that $\xi^\star \in \Phi_i^{-1}([r])$, then up to switching $\xi$ and $\xi^\prime$, we have $[\xi^\star, \xi^\prime] \subset \Phi_i^{-1}([r])$. In particular, if $\Phi_i^{-1}([r])$ contains any interval $J$, then $\Phi_i^{-1}([r])$ is equal to some maximal interval $I$, in the sense of inclusion, that contains $J$. 
\end{proof}

Thus, we make the following definition:

\begin{definition}
    Let $S = S_1 \cup S_2$. We call elements of $S$ the \emph{realizations} of ends of leaf spaces on $\partial P$.
\end{definition}

It follows from Proposition \ref{proposition: mapping to ends} that there are two \emph{types} of realizations of ends: a single point or an interval. For example, the leaf spaces of $\R^2$ foliated by horizontal and vertical lines --- both homeomorphic to $\R$ --- have a total of four ends, which are realized by four intervals each consisting of endpoints of leaves along with their accumulation points. \par

One application of this correspondence is that under certain conditions, we may obtain a faithful action on the bifoliated plane from a faithful action on the set of ends, which a priori asks for less information:

\begin{theorem}\label{theorem: faithful}
    Let $P = (P, \F_1, \F_2)$ be a bifoliated plane and $\partial P$ be the boundary circle at infinity of $P$. Let $G$ be a group that acts on $P$ by homeomorphisms, preserving foliations and their orientations. The action by $G$ on $P$ is faithful if and only if the induced action on $\partial P$ is faithful. Moreover, if the set of point-type realizations of ends is dense in $\partial P$, then the action on $P$ is faithful if and only if the induced action on ends in leaf spaces is faithful.
\end{theorem}

\begin{proof}
    If $G$ acts faithfully on $P$, then for every non-identity element $g \in G$, there exists $x \in P$ such that $g(x) \neq x$. Then at least one of $\F_1(x)$ and $\F_2(x)$ is not fixed by $g$. The leaf that is not fixed by $g$ gives at least one ideal point on $\partial P$ that is not fixed by $g$. So $g$ does not act like the identity on $\partial P$, hence $G$ acts faithfully on $\partial P$. Conversely, if $G$ acts faithfully on $\partial P$, then for every non-identity element $h \in G$, there exists $\xi \in \partial P$ such that $h(\xi) \neq \xi$. Since any point on $\partial P$ is either an ideal point of some leaf or accumulated by ideal points of leaves, by continuity of the action, there exists some leaf that is not preserved by $h$, hence the action is faithful in $P$. \par

    The second \emph{if and only if} statement follows from the first plus the fact that point-type realization is dense in $\partial P$ implies that there is no infinite product region and hence there is a bijection between a dense subset of $\partial P$ and $\Ends(\Lambda_1) \cup \Ends(\Lambda_2)$ by Proposition \ref{proposition: mapping to ends}.
\end{proof}

Note that the induced action on $\partial P$ by $G$ preserves the induced orientation on $\partial P$.

\begin{corollary}
    Let $G$ be a group acting faithfully on $P$, preserving foliations and their orientations. If $G$ acts minimally on $\partial P$, then $G$ is left-orderable.
\end{corollary}

\begin{proof}
    Since $G$ acts minimally on $\partial P$, the orbit of any point-type realization of an end is dense in $\partial P$, implying that the set of point-type realizations of ends is dense in $\partial P$. Then by Theorem \ref{theorem: faithful}, we get a faithful action on ends and thus obtain left-orderability of $G$. 
\end{proof}

An example of such minimal action is the Anosov-like action defined in \cite{barthelme2022orbit} on any non-trivial or skewed bifoliated plane. \par

The last result of this chapter provides a method of realizing a group acting on $(P, \F_1, \F_2)$ as a subgroup of $\Homeo^+(\R)$ when the structure of leaf spaces is simpler, namely when there are only finitely many ends in some leaf space. In this case, it also gives a more straightforward way of obtaining left-orderability of the group.

This is a restatement of Corollary \ref{introcor_finite_ends}, stated in the introduction to this thesis.
\begin{corollary}\label{corollary: lift to left-orderability}
If $G$ acts faithfully on $P$ preserving foliations and their orientations, and if $\Ends_+(\Lambda_i)$ or $\Ends_-(\Lambda_i)$ is finite for some $i$, then $G$ has global fixed point(s) on $\partial P$. In particular, $G$ is isomorphic to a subgroup of $\Homeo^+(\R)$.
\end{corollary}

\begin{proof}
Recall that the fact that $G$ preserves foliations and their orientations implies that the induced action of $G$ on $\Ends_+(\Lambda_i)$ is order preserving, for $i=1,2$. Suppose without loss of generality that $\Ends_+(\Lambda_1)$ is finite. Then, there exists a minimal element $x\in \Ends_+(\Lambda_1)$ with respect to the order on $\Ends_+(\Lambda_1)$. Therefore, this must be a global fixed point of the action of $G$ on $\Ends_+(\Lambda_1)$.

Now, let $J = \Phi^{-1}_1(x) \subset \partial P$ be the realization of $x$ in $\partial P$. Then, $J$ is either a point or a closed interval. If it is a point, then this point must be a global fixed point for the action of $G$ on $\partial P$. 

If $J$ is an interval, then this interval must be preserved by the action of $G$. Suppose that some element $g\in G$ permutes the boundary points of $J$. Then, the action of $g$ on $J$ has some fixed point $\xi$ in the interior of $J$, which must be the endpoint of a unique leaf $l\in \F_2$ since it is in the closure of a standard product region. Therefore, $l$ is preserved by $g$, and $g$ must permute the connected components of $(P\cup\partial P) \setminus \overline{l}$ since it permutes the endpoints of $J$ and preserves $J$. This contradicts that $g$ preserves orientations of both foliations. We conclude that if $J$ is an interval, both boundary points of $J$ must be preserved by the action of $G$. 

We have thus shown that there exists a global fixed point $\xi\in \partial P$ for the action of $G$.  We have $\partial P \setminus \{\xi\} \cong \R$, so the faithful action by $G$ on $\partial P$ induces a faithful action on $\R$ --- this in particular implies that $G$ is left-orderable; see Theorem 6.8 in \cite{ghys_1984}.

\end{proof}   

\chapter{Bifoliated planes of some totally periodic Anosov flows}\label{chapter:BF}

\section{Introduction}\label{section:BF_intro}

In this chapter, we describe the bifoliated planes $(P_\varphi, \F^+, \F^-)$ associated with totally periodic flows $\varphi:\R\times M \to M$ obtained by gluing building blocks as in \cite{barbot2013pseudo}. Then, we prove that these bifoliated planes are in fact all isomorphic to each other. We also describe the action of the fundamental group $\pi_1(M)$ on $(P_\varphi, \F^+, \F^-)$ for a particular example, the Bonatti-Langevin flow first defined in \cite{bonatti1994exemple}. The ideas used to understand this particular case generalize to the general case. 

In Section \ref{section:BF_background} we briefly recall the construction of totally periodic Anosov flows by Barbot and Fenley in \cite{barbot2013pseudo} by gluing circle bundles equipped with semiflows along their boundary components. In said article, the authors mention the possibility of performing certain Dehn surgeries on these circle bundles prior to the gluing step, thus obtaining a larger family of examples which is shown in \cite{barbot2015classification} to include all totally periodic Anosov flows on graph manifolds. Here, we will not consider this broader class of examples, and will restrict ourselves to the case where no such surgeries are performed.

In Section \ref{section:BF_description_planes}, we give our description of the bifoliated planes corresponding to these flows. We begin by discussing \emph{trees of scalloped regions} in \ref{subsection:BF_trees_of_scalloped}. These are a type of chain of lozenges, which in these case cover the whole plane. Therefore, understanding both individual scalloped regions and how different scalloped regions fit together is fundamental in order to understand these bifoliated planes. In Proposition \ref{proposition:BF_periodic_projection_trees} and Corollary \ref{corollary:BF_all_trees_all_scalloped}, we identify the regions of the universal cover $\widetilde{M}$ that project to trees of scalloped regions on the plane.

In \ref{subsection:BF_fitting_trees}, we describe how different trees of scalloped regions intersect. Specifically, Proposition \ref{proposition:BF_tree_intersection} shows that given any scalloped region $S\subset T$ contained in a tree of scalloped region, we can find a different tree of scalloped regions $T'$ such that $T\cap T'= S$. Then, we define the notion of \emph{levels} of trees of scalloped regions, and show that any two points in the plane can be connected by a finite sequence of intersecting trees of scalloped regions, where two adjacent trees in the sequence have levels differing by at most one. This will be an important property in the proof of Theorem \ref{theorem:BF_equivalent_planes} in Section \ref{section:BF_equivalence_totally_periodic_planes}.

In Section \ref{section:BF_equivalence_totally_periodic_planes}, we prove Theorem \ref{theorem:BF_equivalent_planes}, defining via an iterative procedure an isomorphism between the bifoliated planes associated with different totally periodic flows. First, in Proposition \ref{proposition:BF_X_infty_properties} we make explicit some facts about how non-separated leaves corresponding to scalloped regions of different levels intersect lozenges in the plane. These are used in the proof, together with results from the previous section.

In Section \ref{section:BF_action_pi1}, we discuss the action of the fundamental group of a manifold supporting a totally periodic Anosov flow on the bifoliated plane associated with the flow. In \ref{subsection:BF_BL_example} we explicitly describe this action in the case of the Bonatti-Langevin flow, by studying how the elements in a natural generating set of the group act on the bifoliated plane.

\section{Background: the Barbot-Fenley construction of totally periodic Anosov flows}\label{section:BF_background}

Here we outline the construction of some totally periodic flows given in \cite{barbot2013pseudo}.

\begin{definition}
    A graph manifold is a $3$-manifold such that every piece of the JSJ decomposition is Seifert fibered. 
\end{definition}

\begin{definition}
    An Anosov flow $\varphi$ on a graph manifold $M$ is said to be totally periodic if for every Seifert piece in the JSJ decomposition of $M$, a regular fiber has a power which is freely homotopic to a periodic orbit of the flow.
\end{definition}

\begin{definition}
    Let $I = [-\frac{\pi}{2}, \frac{\pi}{2}]$, and let $N = I \times S^1 \times I$, where $S^1 = \frac{[0,1]}{0 \sim 1}$.
    Given $\lambda > 0$, define a vector field $X_\lambda$ on $N$ as

    \[
    X_\lambda(x,y,z) = (0, \, \lambda \sin(x) \cos^2(z), \, \cos^2(x) + \sin^2(z) \sin^2(x))
    \]
and define $\psi_\lambda: U \subset \R \times N \to N$ to be the (local) flow on $N$ generated by $X_\lambda$. Note that $\psi_\lambda$ is not defined globally, that is, $U \subsetneq \R \times N$.

\end{definition}
\begin{figure}[h]
  \centering
  \includegraphics[width=0.5\linewidth]{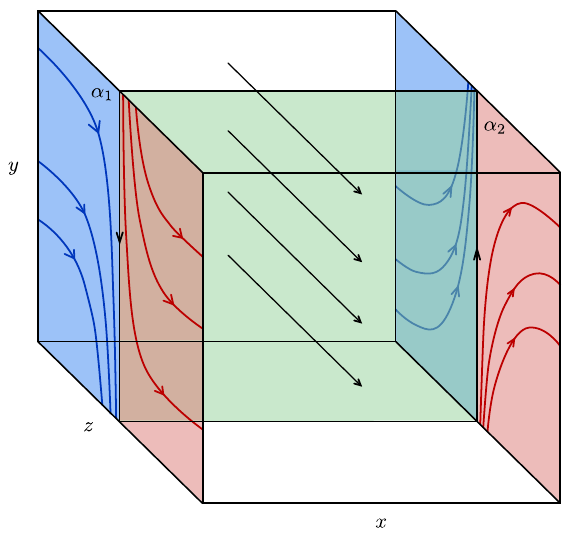}
  \caption{A building block $I\times S^1\times I$. Birkhoff annulus shown in green. Colorized version of Figure 4 in \cite{barbot2013pseudo}.}
  \label{fig:BF_block}
\end{figure}

\begin{proposition}[\cite{barbot2013pseudo}]\label{proposition:BF_background_block}   
For the flow $\psi_\lambda$ on $N$, the following hold:
\begin{enumerate}

    \item There are exactly two closed orbits of $\psi_\lambda$, given by $\alpha_1 = \{ -\pi/2\} \times S^1 \times \{ 0\}, \, \alpha_2 = \{ \pi/2\} \times S^1 \times \{ 0\} $.

    Notice that in orbit $\alpha_1$, the $y$ coordinate decreases in the future, while the opposite happens for $\alpha_2$. Therefore, as oriented orbits, $\alpha_1$ and $(\alpha_2)^{-1}$ are freely homotopic.

    \item 

    There exists, up to isotopy, a unique embedded elementary Birkhoff annulus for the flow $\psi_\lambda$, given by $A = I \times S^1 \times \{ 0\}$. Its boundary components are the orbits $\alpha_1$ and $\alpha_2$.

    \item 

The flow is tangent to the boundary components given by $x = \pm \frac{\pi}{2}$, which contain the periodic orbits $\alpha_1$ and $\alpha_2$.

It is transverse to the boundary components given by $z = \pm \frac{\pi}{2}$. It is incoming for $z = - \frac{\pi}{2} $ and outgoing for $z = \frac{\pi}{2}$.

\item The tangential boundary component given by $x = - \frac{\pi}{2}$ is divided by the orbit $\alpha_1$ into two connected components. These are the stable and unstable leaves of $\alpha_1$: orbits in $\{ -\frac{\pi}{2}\} \times S^1 \times [-\frac{\pi}{2}, 0) $ approach $\alpha_1$ in the future, hence $W^s(\alpha_1) = \{ -\frac{\pi}{2}\} \times S^1 \times [-\frac{\pi}{2}, 0) $. Similarly, $W^u(\alpha_1) = \{ -\frac{\pi}{2}\} \times S^1 \times (0, \frac{\pi}{2}]$. See Figure \ref{fig:BF_block}.

An analogous situation occurs for the tangential boundary component given by $x = \frac{\pi}{2} $ which contains $\alpha_2$. Here we have $W^s(\alpha_2) =  \{ \frac{\pi}{2}\} \times S^1 \times [-\frac{\pi}{2}, 0) $ and $W^u(\alpha_2) =  \{ \frac{\pi}{2}\} \times S^1 \times (0, \frac{\pi}{2}] $.

\item Every orbit in the interior of $N$ enters $N$ through the incoming boundary component and exits through the incoming boundary component.

\end{enumerate}

\end{proposition}

The manifold $N$ together with the flow $\psi_\lambda$ is referred to as a \emph{building block}.

The next step in the construction is to glue together a number of building blocks along their tangential boundary components. In order to do this, we need the combinatorial information from a fat graph.

\begin{definition}\label{definition:fatgraph}
    A fat graph $X$ is a graph embedded in a surface with boundary $\Sigma$, in such a way that $X$ is a retract of $\Sigma$.
    
    Moreover, we will say a fat graph $X \subset \Sigma$ is admissible if it satisfies the following:
    \begin{enumerate}
        \item The valence of every vertex of $X$ is four.
        \item The set $\mathcal{B}$ of boundary components of $\Sigma$ can be partitioned into two subsets $\mathcal{B}_+$ and $\mathcal{B}_-$, such that for any retraction $r: \Sigma \to X$ and any edge $e$ of $X$, the image of exactly one boundary component from each subset under $r$ intersects $e$.

        We call boundary components in $\mathcal{B}_+$ \emph{outgoing}, and those in $\mathcal{B}_-$ \emph{incoming}.
    \end{enumerate}
    
\end{definition}

Now, let $X$ be an admissible fat graph as above. For every edge $e$ of $X$, let $N_e$ be a building block. We construct a $3$-manifold with boundary $N(X)$ from the building blocks $N_e$.

Every incoming boundary component $B\in \mathcal{B}_-$ of $ \Sigma$ corresponds to a cyclic sequence of edges $(e_1, e_2, \dots, e_k)$. These are the edges $e_i$ such that for a retraction $r:\Sigma \to X$, $r(B) \cap e_i \neq \emptyset$. For each $i=1,\dots, k$, we glue $N_{e_i} $ to $N_{e_{i+1}}$ along the stable manifolds $ \{ \frac{\pi}{2}\} \times S^1\times [-\frac{\pi}{2}, 0) \subset N_{e_i}$ and $ \{- \frac{\pi}{2} \}\times S^1 \times [-\frac{\pi}{2}, 0) \subset N_{e_{i+1}}$ via the map $(\frac{\pi}{2}, y, z) \mapsto (- \frac{\pi}{2}, -y, z)$. Here we consider $e_{k+1} = e_1$.

We perform the gluings described above for each incoming boundary component $B\in \mathcal{B}_-$. Then, the procedure is repeated for the outgoing boundary components $B \in \mathcal{B}_+$. In this way, we obtain a $3$-manifold with boundary $N(X)$ which is a circle bundle over the surface with boundary $\Sigma$, with the fibers in each building block being the sets of the form $
\{ x\}\times S^1\times \{z\}$. The flow is transverse to the boundary of $N(X)$, and there are incoming and outgoing boundary components. Each boundary component is a torus or a Klein bottle. Note that the only periodic orbits in this flow are those coming from the different copies of the orbits $\alpha_1$ and $\alpha_2$ in the building blocks $N_e$.

Now, we choose a finite collection $X_1, \dots ,X_k$ of fat graphs. For each of these, we construct a manifold with boundary $N(X_i)$ as above, equipped with a flow transverse to the boundary. These are chosen so that two conditions are satisfied. First, all the boundary components must be tori. Second, the total number of outgoing boundary components in the $N(X_i)$ must equal the total number of incoming boundary components.

Finally, one chooses a corresponding incoming boundary component $T'$ for every outgoing boundary component $T$, as well as a gluing map $A: T\to T'$. These gluing maps must satisfy the following condition $(*)$:

\begin{align*}
    (*)& \, \, \,\, \text{for periodic orbits } \alpha, \alpha' \subset N(X), \text{ the curve (if non-empty)} \\  & \, \,\,\,A(W^u(\alpha)\cap T) \text{ is not isotopic to } W^s(\alpha')\cap T'
\end{align*}

In this way, we obtain a manifold $M = M(\mathcal{X}, \mathcal{A})$, where $\mathcal{X}= \{ X_1, \dots, X_k \}$ is the collection of fat graphs used for the construction of the $N(X_i)$ above, and $\mathcal{A} = \{ A_i : T_i\to T_j, \,i\neq j\}$ is the collection of gluing maps. Observe that the JSJ pieces of $M$ are the $N(X_i)$. We also have a flow $\varphi_\lambda$ on $M$.

\begin{theorem}[\cite{barbot2013pseudo}]

For $\lambda > 0$ large enough, the flow $\varphi_\lambda$ on $M(\mathcal{X}, \mathcal{A})$ is Anosov.

\end{theorem}

In \cite{barbot2015classification} it was shown that, as long as the constant $\lambda>0$ is large enough so that the flow $\varphi_\lambda$ on $M(\mathcal{X}, \mathcal{A})$ is Anosov, this flow does not depend on $\lambda$ (up to orbit equivalence). Moreover, it only depends on the isotopy classes of gluing maps in $\mathcal{A}$, and not in the specific representatives chosen. 

\begin{definition}
    We refer to a flow $\varphi_\lambda$ constructed via the procedure described in this section as a \emph{totally periodic Anosov flow with no surgeries}. 
    
    Since up to orbit equivalence there is no dependence on $\lambda$, we write $\varphi = \varphi_\lambda$.
\end{definition}

\section{Description of the bifoliated planes}\label{section:BF_description_planes}
In this section, we describe the bifoliated plane $(P_\varphi, \F^+, \F^-)$ corresponding to a totally periodic Anosov flow with no surgeries $\varphi$ on a manifold $M = M(\mathcal{X}, \mathcal{A})$, as defined in Section \ref{section:BF_background}.

\subsection{Trees of scalloped regions}\label{subsection:BF_trees_of_scalloped}

Let $\varphi$ be a totally periodic Anosov flow with no surgeries, defined on the manifold $M = M(\mathcal{X}, \mathcal{A})$.

Let $N(X_1), N(X_2), \dots, N(X_k)$, where $X_i \in \mathcal{X}, \, X_i \subset \Sigma_i$ are fat graphs embedded in the surfaces $\Sigma_i$. Here the $3$-manifolds with boundary $N(X_i)$ are using the combinatorial data from $X_i$, via the construction explained in Section \ref{section:BF_background}. 

We know that each $N(X_i)$ is a circle bundle over the surface-with-boundary $\Sigma_i$. Therefore, its universal cover is $\widetilde{N(X_i) } \cong S_i \times \R$, where $S_i$ is the universal cover of $\Sigma_i$. For each $i$, $S_i$ is a surface with boundary which is topologically a plane with infinitely many open half planes removed, and where the boundary components do not accumulate on each other. Explicitly, we can see that $S_i$ is a thickening of an infinite tree $\widetilde{X}$ where each node has valence $4$. That is, the universal cover $S_i$ of all of the surfaces $\Sigma_i$ is the same (of course, the deck group is not in general the same). We denote this space by $S$.

The infinite tree $\widetilde{X}$ is the universal cover of each of the fat graphs $X_i$. The surface with boundary $S$ retracts onto $\widetilde{X}$.

We will now discuss what the connected components of a lift of each $N(X_i)$ to the universal cover $\widetilde{M}$ are. Since these lifts are the same up to homeomorphism for all $i$ (the map $\pi_1(N(X_i)) \to \pi_1(M)$ induced by inclusion is injective for all $i$, and the $N(X_i)$ have the same universal cover), it is enough to study a connected component of one of these lifts. We denote it by $\widetilde{N(X)}$.

The boundary $\partial N(X) = \bigcup_i T_i$ is a union of tori, which lift in the universal cover to vertical planes $P_1, P_2, \dots, $ of the form $P_i = C_i\times \R \subset S \times \R $, where $C_i \subset \partial S$ is a boundary component of $S$. The planes $P_i$ bound $\widetilde{N(X)}$. 

The universal cover $\widetilde{M}$ of $M$ can then be obtained by taking infinitely many copies of $\widetilde{N(X)}$ and gluing them to each other along the vertical planes that bound them.

We note the following: in the same way that $N(X)$ is obtained by gluing some number of building blocks $I\times S^1 \times I$ together, the same is true of each connected component of $\widetilde{N(X)}$ (except with infinitely many building blocks). Here the blocks that make up a connected component of $\widetilde{N(X)}$ are simply the universal cover $I \times \R \times I$ of $N = I \times S^1 \times I$, where each block shares each stable or unstable half-leaf with another block. 

As stated in item $2$ of Proposition \ref{proposition:BF_background_block}, each building block $N$ of $N(X)$ is a neighborhood of a Birkhoff annulus $A$ for the flow $\varphi$, with the boundary components of $A$ being the periodic orbits contained in the building block $N$. Therefore, the lift of such a building block $N$ to the universal cover $\widetilde{M}$ is a neighborhood of a lift $D\cong I\times \R $ of the elementary Birkhoff annulus $A$. The boundary components of $D$ are lifts of periodic orbits of $\varphi$.

We will focus now on understanding the projection of a single copy of $\widetilde{N(X)}$ to the bifoliated plane $P_\varphi$.

We can understand this projection by translating the structure of $\widetilde{N(X)}$ as a union of blocks which are neighborhoods of lifts of Birkhoff annuli to the bifoliated plane. First, we give a definition. Recall that (see Definition  \ref{definition:chain_adjacent_lozenges}) a chain of lozenges is said to be a chain of adjacent lozenges if any two lozenges in the chain can be joined by a finite number of lozenges in the chain, so that two consecutive lozenges share a side (i.e. are adjacent). 

\begin{definition}
    A \emph{tree of scalloped regions} is an infinite chain of adjacent lozenges such that each lozenge is adjacent to four others in the chain, one on each of its sides.
\end{definition}

\begin{remark}
    It follows from the definition that a tree of scalloped regions is a maximal chain of adjacent lozenges.
\end{remark}

\begin{proposition}\label{proposition:BF_periodic_projection_trees} If $\pi : \widetilde{M} \to M$ is the universal covering of $M$ and $p: \widetilde{M} \to P_\varphi$ is the projection to the orbit space of $\varphi$, then for any connected component $U$ of $\pi^{-1}(N(X)))$, the subset $p(U) \subset P_\varphi$ is a tree of scalloped regions in the bifoliated plane of $\varphi$.
\end{proposition}

\begin{figure}[h!]
  \centering
  \includegraphics[width=.8\linewidth]{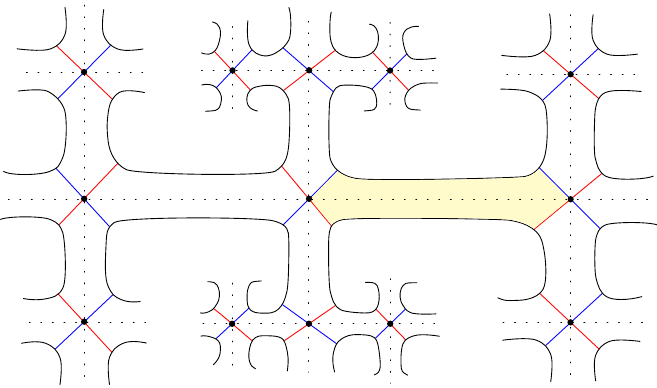}
  \caption{Horizontal cross section of $\widetilde{N(X_i))}$}
  \label{fig:your_label}
\end{figure}

\begin{proof}

Let $U$ be a connected component of $\widetilde{N(X)}$. We want to show that $p(U) \subset P_\varphi$ is a tree of scalloped regions. 

Let $\{ A_i : i =1,\dots, l\}$ be a collection of Birkhoff annuli such that each building block $N_i\cong I\times S^1\times I$ of $N(X)$ is a neighborhood of exactly one Birkhoff annulus $A_i$ in the collection. Recall from Proposition \ref{proposition:BF_background_block} that each orbit in the block $N_i$ which is not contained in a transverse boundary component of $N_i$ intersects the Birkhoff annulus $A_i$.

Let $\{ D_j: j\in \N \}$ denote the collection of all lifts of the Birkhoff annuli $A_i$, such that the $D_j$ are contained in the connected component $U \subset \widetilde{M}$ of $\pi^{-1}(N(X))$. Then, each $D_j$ is bounded by two orbits of $\widetilde{\varphi}$, which are lifts of the periodic orbits of $\varphi$ bounding the Birkhoff annulus $\pi(D_j) = A_{i_j}$. Note that these periodic orbits do not intersect the boundary tori of $N(X)$, by construction of the flow. 

Each $D_j$ is contained in some lift $W_j \cong I\times \R \times I$ of the building block $N_{i_j}$, and therefore we know that each orbit of $\widetilde\varphi$ in $W_j$ which is not contained in a transverse boundary component must intersect $D_j$.

\begin{figure}[h!]
  \centering
  \includegraphics[width=.7\linewidth]{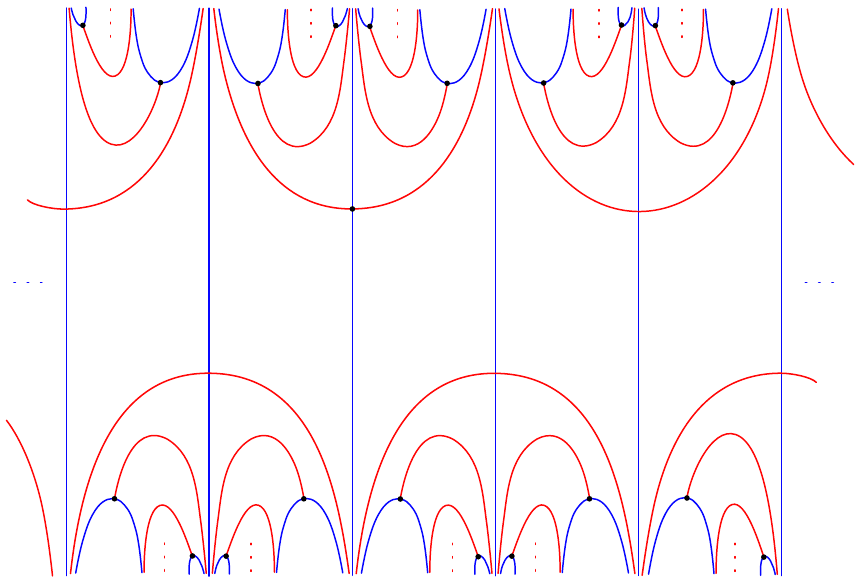}
  \caption{(part of) a tree of scalloped regions.}
  \label{fig:part_tree_scalloped}
\end{figure}

By Proposition \ref{proposition:birkhoff_projects_lozenges}, each $D_j \subset W_j$ projects to a lozenge $L_j \subset p(U)$, and by the paragraph above, all the orbits in the interior of $W_j$ intersect $D_j$. Therefore, the orbits in $W_j$ all project to $L_j$, with the orbits contained in the tranverse boundary components of $W_j$ projecting to the stable or unstable leaves of the corners of $L_j$, which are the projections of the orbits bounding $D_j$.

Moreover, we can see that two adjacent (glued along a stable or unstable half-leaf) blocks $W_j, W_{j'}$ in $U$ project to lozenges that share a side in the bifoliated plane. If the blocks share a stable (resp. unstable) half-leaf, then their projections will be adjacent along a stable (resp.) unstable boundary half-leaf. Since each periodic orbit of the flow $\varphi$ which is in the boundary of a building block must be contained in four such building blocks (recall that the valence of the fat graph $X$ is $4$), this means that each lozenge $L_k$ is adjacent to four lozenges in $p(U)$. Therefore, $p(U)$ is a tree of scalloped regions.

\end{proof}


Two corollaries of the above:

\begin{corollary}
    The plane $P_\varphi$ is a union of trees of scalloped regions.
\end{corollary}
This follows immediately from Proposition \ref{proposition:BF_periodic_projection_trees} above, since any orbit of the flow on $M$ must intersect at least one of the $N(X_i)$.

\begin{corollary}\label{corollary:BF_lifted_torus_projects_scalloped}
Each lift $\widetilde T$ of a torus bounding $N(X)$ projects to a scalloped region in $P_\varphi$. 
\end{corollary}

\begin{proof}
Let $T$ be a torus bounding $N(X)$, and let $\widetilde{T}$ be a lift of $T$ to $\widetilde{M}$. Let $\widetilde{T} \subset \partial U$, where $U$ is a connected component of $\pi^{-1}(N(X)) \subset \widetilde{M}$. We know by Proposition \ref{proposition:BF_periodic_projection_trees} that the projection $p(U)$ to $P_\varphi$ is a tree of scalloped regions. 

The proof of this proposition shows in fact that the interior of each lozenge $L$ in this tree is the projection to $P_\varphi$ of the interior of a block $\widetilde{N} \cong I\times \R\times I$, such that $\pi(\widetilde{N}) = N$ for a building block $N\subset N(X)$. Moreover, it shows that two such lozenges $L, L'$ are adjacent along a stable (or unstable) half-leaf if and only if the blocks $N, N'$ are adjacent along a stable (resp. unstable) half-leaf which is contained in a transverse boundary component of both blocks. 

Since the torus $T$ is either incoming or outgoing for the flow restricted to $N(X)$, the transverse boundary components that intersect $T$ do so only along stable leaves (if incoming) or only along unstable leaves (if outgoing). Therefore, the blocks $\widetilde{N}$ intersecting $\widetilde{T}$ are adjacent along half-leaves that are either always unstable or always stable. This means that the same is true for the lozenges to which these blocks project.

We have shown that $p(\widetilde{T})$ is a bi-infinite chain of lozenges such that two adjacent lozenges are always adjacent along leaves of a fixed foliation. By Proposition \ref{proposition:infinite_nonsep_implies_scalloped_boundary}, $p(\widetilde{T})$ is a scalloped region.
    
\end{proof}

Recall that in the proof of Proposition \ref{proposition:BF_periodic_projection_trees} we denoted by $D_k$ the lifts of Birkhoff annuli that join vertical periodic orbits in a connected component of $\widetilde{N(X_i)}$. Abusing notation and relabeling, let $\{ D_k : k\in \N\}$ denote the set of all such lifts of Birkhoff annuli, for all connected components of the different $\widetilde{N(X_i)}$. 

\begin{definition}
    We denote by $\mathcal{L}$ the set $\mathcal{L} = \{ p(D_k) : k\in \N\}$, and $\mathcal{T} = \{ p(U) : U \text{ connected component of } \widetilde{N(X_i)} \text{ for some } i\}$.
\end{definition}

\begin{proposition}\label{proposition:BF_all_lozenges_mathcal}
    All lozenges in $P_\varphi$ are elements of $\mathcal{L}$.
\end{proposition}

We will use the following result, which is Lemma 2.29 in \cite{barthelme2022orbit}:

\begin{lemma}[\cite{barthelme2022orbit}]\label{lemma:BF_interior_scalloped_no_corner}
    Let $S$ be a scalloped region, and let $L_0 \subset S$ be a lozenge. Then, no point in the interior of $L_0$ can be the corner of a lozenge. 
\end{lemma}

\begin{proof}[Proof of Proposition \ref{proposition:BF_all_lozenges_mathcal}]

Suppose that $L \subset P_\varphi$ is a lozenge. Let $z$ be a corner of $L$. If $z$ is in a boundary leaf $l$ of some lozenge $L'$ in $\mathcal{L}$, we can see from the fact that $L'$ is part of a tree of scalloped regions that it is impossible for the other boundary leaf $l'$ of $L$ containing $z$ to make a perfect fit unless $z$ coincides with the corner of the lozenge $L'$.

One can see that, by the definition of lozenge, given a point $x\in P_\varphi$ and two half-leaves based at $x$, there can exist at most one lozenge having the half-leaves as sides. Therefore, in this case we must have $L = L'$, so $L\in \mathcal{L}$.

On the other hand, if $z$ is not in any boundary leaf of a lozenge in $\mathcal{L}$, then the fact that $P_\varphi = \bigcup_{L\in \mathcal{L}} L $ implies that there exists some $L_0 \in \mathcal{L}$ such that $z$ is in the interior of $L_0$. By the Lemma above, this cannot happen, so we have shown what we wanted.

\end{proof}

\begin{corollary}\label{corollary:BF_all_trees_all_scalloped}
    All trees of scalloped regions in $P_\varphi$ are in $\mathcal{T}$, and all scalloped regions in $P_\varphi$ are projections of lifts of the transverse tori that bound the pieces $N(X_i)$.
\end{corollary}

\subsection{Fitting together different trees of scalloped regions}\label{subsection:BF_fitting_trees}

In this section, we describe how the trees of scalloped regions described in the previous section fit together to form $P_\varphi$.

First, we prove some useful results.

\begin{definition}
    Let $T \in \mathcal{T}$ be a tree of scalloped regions. Let $L_0 \subset T$ be a lozenge such that $T$ coincides with the maximal chain of adjacent lozenges containing $L_0$. We define $\mathcal{L}(T) $ to be the set of lozenges contained in said maximal chain.
\end{definition}

It's easy to see that the definition of $\mathcal{L}(T)$ does not depend on the choice of $L_0$. The reason this definition is necessary is that in all trees of scalloped regions $T$ there exist lozenges which are \emph{contained} in $T$, but such that the maximal chain of adjacent lozenges containing them only intersects $T$ in a scalloped region (recall that a scalloped region is a chain of adjacent lozenges in two different ways). It will be necessary to differentiate between these and the lozenges in $\mathcal{L}(T)$. We will discuss this further below.

\begin{lemma}\label{lemma:BF_no_tree_loop}

Let $U_0, U_1, \dots, U_n \subset \widetilde{M}$ such that for all $i$, $U_i $ is a connected component of $\pi^{-1}(N(X_{j_i}))$, for $X_{j_i} \in \mathcal{X}$ (note that we allow $X_{j_i} = X_{j_k}$) for $i\neq k$), such that $\emptyset \neq U_i \cap U_{i+1}  \subset \partial U_i \cap \partial U_{i+1}$ for $i=1,2,\dots, n-1$. 

Then, $U_0 \neq U_n$.

\end{lemma}

\begin{proof}
    If $U_0 = U_n$, then we can construct a loop based at a point $x\in U_0$ which intersects each set $U_i$ in a segment. This loop intersects a connected component of $\partial U_0$ exactly once. This connected component is a properly embedded plane contained in $\widetilde{M}$, so the loop cannot be nullhomotopic, contradicting simple-connectedness of $\widetilde{M}$.
\end{proof}

\begin{proposition}\label{proposition:BF_tree_intersection}
    For any tree of scalloped regions $T$ and scalloped region $S \subset T$, there exists a tree of scalloped regions $T'$ such that $T' \cap T = S$. 
    
    Moreover, for any tree of scalloped regions $T \neq T_0$ such that $T \cap T_0 \neq \emptyset$, there exists a scalloped region $S_0 \subset T_0 $ such that $T \cap T_0 \subseteq S_0$.
\end{proposition}

\begin{proof}

Let $T$ be a tree of scalloped regions, and let $S\subset T$ be a scalloped region. By Proposition \ref{corollary:BF_all_trees_all_scalloped}, $T$ is the projection to $P_\varphi$ of a lift $\widetilde{N(X_i)}$ of a totally periodic piece $N(X_i)$ of the flow $\varphi$ to the universal cover $\tilde{M}$, and $S$ is the projection of the lift $\tilde{T_j} \subset \tilde{M}$ of a transverse torus $T_j \subset \partial N(X_i)$. For convenience, denote $N(X_i)$ by $N_i$ from now on.

Suppose without loss of generality that $S$ consists of lozenges $L_k \subset T$ that are adjacent on their $s$-sides, that is, the flow is incoming on $T_j$. We know that there exists a torus $T_j' \subset \partial N_i'$ where $N_i'$ is a totally periodic piece of the flow (here possibly $N_i' = N_i$, but in any case $T_j'\neq T_j$) such that $T_j$ and $T_j'$ are identified in $M$ by a map $A: T_j \to T_j'$. Moreover, the flow is outgoing on $T_j'$. The projection of $T_j'$ to $P_\varphi$ is then a scalloped region $S'$ contained in the tree of scalloped regions $T'$ which is the projection of $\tilde{N_i'}$ (note that, even if $N_i = N_i'$, we must have $\tilde{N_i} \neq \tilde{N_i'}$). The scalloped region $S'$ is a union of lozenges $L_k' \subset T'$ which are adjacent on their $u$-sides, since the flow is outgoing on $T_j'$.

Since $\tilde{T_j'} $ and $\tilde{T_j}$ are identified in $\tilde{M}$, we must have $S = S'$. Now, we show that $T\cap T' = S$. Clearly $S \subseteq T\cap T'$, so we only need to show the other inclusion. Fix a point $x_0 \in S$. We must have $x_0 \in L_0\cap L_0'$, where $L_0 \subset T$, $L_0' \subset T'$ are lozenges. Given another point $x \in T\cap T'$, we have $x\in L\cap L'$ for lozenges $L\subset T, L'\subset T'$ and since $T, T'$ are trees of scalloped regions there must exist a finite sequence of adjacent lozenges $L_0, L_1, L_2, \dots, L_n = L $ in $T$ joining $L_0, L$ and a sequence $L_0', \dots, L_m' = L'$ of lozenges in $T'$ joining $L_0', L'$. We may assume that these sequences are of minimal length among all such sequences. 

By definition of an $s$-scalloped region, any lozenge adjacent to $L_0$ on its $s$-side must be contained in $S$. Suppose that there exist $L_k, L_{k+1}$ in the sequence joining $L_0$ to $L$ that are adjacent on their $u$-sides. Let $l$ be the $u$-leaf separating $L_k$ from $L_{k+1}$. Since the sequence is of minimal length, we must have that $x$ and $x_0$ are in different connected components of $P_\varphi - l$. But this contradicts the existence of the sequence $L_0', \dots, L_n' = L'$ joining $L_0'$ to $L'$, since any lozenge in such a sequence must be contained in the same connected component of $P_\varphi - l$ as $L_0'$. Therefore, we have shown that the lozenges $L_0, L_1, \dots, L_k$ must all be adjacent on their $s$-sides, and therefore $L_n = L$ must be contained in $S$, so $x\in S$. This shows $T\cap T' \subseteq S$, and we conclude that $T\cap T' = S$.

\begin{figure}[h]
  \centering
  \includegraphics[width=0.85\linewidth]{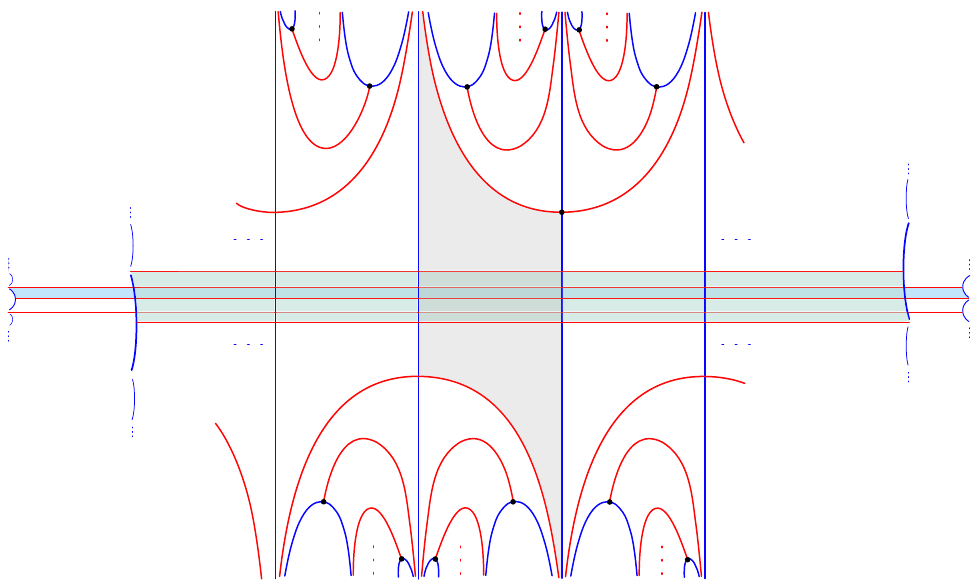}
  \caption{intersection between two different trees of scalloped regions}
  \label{fig:tree_intersection_1}
\end{figure}

Now, let $T \neq T_0$ be a tree of scalloped regions such that $T \cap T_0 \neq \emptyset$. Then, there exist $N_i, N_j$ and connected components of their lifts $U, U_0$ such that $T = p(U), T_0 = p(U_0)$ (note that here possibly $N_i = N_j$, but $U\neq U_0$). The fact that $T \cap T_0 \neq \emptyset$ means that there exists an orbit $\tilde\varphi_t(x)$ of $\tilde\varphi$ that intersect both $U$ and $U_0$. This is possible only if there exists a finite sequence $U_0, U_1, \dots, U_m$ such that:
\begin{itemize}
\item Each $U_k$ is a connected component of the lift of some piece $N(X_k)$ to $\tilde{M}$.
    \item $U_m = U$.
    \item The orbit $\tilde\varphi_t(x)$ intersects all the $U_k$.
\end{itemize}

Therefore, for each $k=0,1,\dots, n$ there exist lifts of transverse tori $\widetilde{T}_{k,1}, \widetilde{T}_{k,2} \subset \partial \widetilde{N(X_k)}$ such that the flow is incoming on $T_{k,1}$, outgoing on $T_{k,2}$, and $\tilde{T}_{k,2}$ is identified with $\tilde{T}_{k+1, 1}$ for $k=0,1,\dots,k-1$ in $\tilde{M}$.   

Note that there can exist at most one such sequence of lifts of tori, by the Lemma above. That is, every orbit of $\tilde\varphi$ that intersects both $U_0$ and $U_k$ must intersect the planes $\widetilde{T}_{k,i}$ defined above. This shows that $T\cap T_0 \subseteq p(\tilde{T}_{1,1})$, which is a scalloped region $S_0\subset T_0$, as we wanted.

\end{proof}

\subsection{Levels of trees of scalloped regions}\label{subsection:BF_levels_of_trees}

Here we define the notion of \emph{level} of a tree of scalloped regions. As we have seen before, each tree of scalloped regions is intersected by infinitely many others, with each intersection contained in some scalloped region. We use the way that they intersect to define their \emph{difference in levels} (and later extend this to non-intersecting trees), which is an invariant under the automorphism group of the bifoliated plane. The fact that the difference in level between any two trees is finite will be used later when showing that all bifoliated planes $P_\varphi$ coming from a totally periodic Anosov flow $\varphi$ are isomorphic.

\begin{definition}\label{definition:BF_difference_level_1}
    Let $T$ be a tree of scalloped regions, and let $S \subset T$ be a scalloped region such that the lozenges in $\mathcal{L}(T)$ that intersect $S$ are adjacent on their stable sides.

    Let $T'$ be the tree of scalloped regions such that $T\cap T'= S$. That is, $T'$ is such that the lozenges in $\mathcal{L}(T')$ that intersect $S$ are adjacent on their unstable sides, and their union is $S$.

    Define the \emph{difference in level between} $T$ and $T'$ to be $\Delta(T,T') = 1$. 

    If $\Delta(T,T') = 1$, we define $\Delta(T', T) =-1.$
\end{definition}

The following is a direct consequence of Corollary \ref{corollary:BF_lifted_torus_projects_scalloped}, which tells us that a lift of a boundary component of a piece $N(X_i)$ projects to a scalloped region in the orbit space.
\begin{lemma}
    Let $T_1 = p(U_1), \, T_2 = p(U_2)$ where for $i=1,2,$ $U_i$ is a connected component of $\pi^{-1}(N(X_{j_i}))$ for $X_{j_1}, X_{j_2} \in \mathcal{X}$ (possibly $X_{j_1} = X_{j_2}$) such that a boundary component of $N(X_{j_1}) $ is identified with a boundary component of $N(X_{j_2})$ in the construction of $M$.

    Then, $\Delta(T,T') = \pm 1$.
\end{lemma}

\begin{proposition}\label{prop:level_pm1}
    Let $L \in \mathcal{L}(T), L' \in \mathcal{T'}$ be two lozenges in the trees $T, T'$. Then, there exists a unique finite sequence of trees of scalloped regions $T_0, T_1, \dots, T_n$ such that:
    \begin{itemize}
        \item $T_0 = T, T_n = T'$.
        \item $T_i \cap T_{i+1} \neq \emptyset$ and $\Delta(T_i,T_{i+1}) = \pm 1$ for $i = 0, \dots, n-1$.
    \end{itemize}
\end{proposition}

\begin{proof}
    Take points $\widetilde{x}, \widetilde{y} \in \widetilde{M}$ which project to points in $L$ and $L'$. Construct a path $\eta:[0,1]\to \widetilde{M}$ such that $\eta(0) = \widetilde{x}, \eta(1) = \widetilde{y}$. The path $\eta$ sequentially intersects a finite number of connected components $U_0, U_1, \dots, U_n$ of lifts $\pi^{-1}(N(X_i))$ of pieces $N(X_i)$. Taking $T_i = p(U_i)$, we are done by the lemma above.
    
\end{proof}

\begin{definition}
    Given any two trees of scalloped regions $T, T'$, define
    \[\Delta(T,T')=  \sum_{i=0}^n \Delta(T_i, T_{i+1}) \]
    where $T_i $ are as in Proposition \ref{prop:level_pm1}.
\end{definition}

\begin{definition}\label{definition:leaf_levels}
    Let $x, y \in \Lambda^+$ be non-separated leaves such that a half-leaf of $x$ bounds a lozenge $L \in \mathcal{L}(T)$ and $y$ bounds a lozenge $L' \in \mathcal{L}(T')$, for trees $T, T'$. Then, define the difference in levels between $x$ and $y$ to be
    \[\Delta(x,y) =  \Delta(T,T').\]

Additionally, if we fix a reference leaf $x_0 \in \Lambda^+$, we can define the level of $x\in \Lambda^+$ to be $\Delta(x_0, x)$.
\end{definition}

\section{Proof of Theorem \ref{introthm:totally_periodic_isomorphic}}\label{section:BF_equivalence_totally_periodic_planes}

In this section, we prove Theorem \ref{introthm:totally_periodic_isomorphic}, restated below as Theorem \ref{theorem:BF_equivalent_planes}. It tells us that the bifoliated planes corresponding to any two totally periodic Anosov flows with no surgeries are isomorphic.

We begin by stating some necessary definitions.
\begin{definition}
    Let $(P, \F^+, \F^-)$ and $  (P', \mathcal{G}^+, \mathcal{G}^-)$ be bifoliated planes, where we have fixed a choice of orientation for each of their foliations (the definition below will not depend on this choice).
    
    Let $L\subset P, L'\subset P'$ be lozenges. We say that they are of the same type if there exists a homeomorphism $f: L\to L'$ that preserves the foliations and their orientations.
\end{definition}

\begin{definition}
    Let $f: P \to P'$ be an isomorphism between bifoliated planes $(P, \F^+, \F^-)$ and $  (P', \mathcal{G}^+, \mathcal{G}^-)$, and let $T \subset P, T'\subset P'$ be trees of scalloped regions. Then, we denote by $\bar{f}$ the induced map $\bar{f}: \mathcal{L}(T) \to \mathcal{L}(T')$ sending a lozenge $L$ to $f(L)$.

\end{definition}

\begin{definition}
Given a bifoliated plane $(P, \F^+, \F^-)$, we define $\Lambda^\pm_{\mathrm{non-sep}} \subset \Lambda^{\pm}
$ to be the set of non-separated leaves in the leaf space of $\F^\pm$.

Given some subset $Z \subset P$, let $\Lambda^\pm_{\mathrm{non-sep}} (Z) \subset \Lambda^\pm_{\mathrm{non-sep}}$ be the set of non-separated leaves intersecting $Z$.
\end{definition}

Now, we make explicit some facts about non-separated leaves in these planes which have already been used in the previous section.

\begin{definition}
Let $X \subset [0,1]$ be the set $\bigcup_{i=3}^{+\infty}  \{ \frac{1}{i}, 1- \frac{1}{i} \}  $.    
\end{definition}
 It has the following properties:
\begin{enumerate}
    \item It has the order type of the integers. 
    \item Its only accumulation points are $0$ and $1$.
    
\end{enumerate}

Let $h: \Z \to X$ be the bijection given by
\[
h(n) = \begin{cases}
    1 - \frac{1}{2 +  n  } &\text{if } n > 0\\
    \frac{1}{3 + \left|n \right|} &\text{if } n \leq 0
\end{cases}
\]

We will enumerate the family of intervals with endpoints in $X$ which do not have elements of $X$ in their interior using this bijection.

Let $I_n = [h(n), h(n+1)] $ for $n\in \Z$, so that $I_0 = [1/3, 2/3]$. Let $\ell(n)$ denote the length of $I_n$. The following will be used later:

For $n\in \Z$, let $g_n:[0,1] \to [0,1]$ be given by $g_n(x) = \ell(n)x + h(n) $. That is, $g_n([0,1]) = I_n$ and $g_n$ is affine.

\begin{definition}
    For $n\in \Z$, let $X^{(1)}_n \subset I_n$ be the image of $X$ under the unique bijective affine map $g_n :[0,1] \to I_n$. Define, for $m\in \Z$, the interval $I_{(n,m)}$ to be $I_{(n,m)} = g_n(I_m) \subset I_n$.

    Suppose that the intervals $I_{(n_1, \dots, n_k)} \subset I_{(n_1, \dots, n_{k-1})} $ (for $k = 0$, take $I = [0,1]$) have been defined, where $k \geq 1$ and $(n_1, \dots, n_k) \in \Z^k$. Define then the set $X^{(k)}_{(n_1, \dots, n_k) }$ to be the image of $X$ under the unique affine map $g_{(n_1, \dots, n_k)} : [0,1] \to I_{(n_1, \dots, n_k)}$, and for $(n_1, \dots, n_{k+1}) \in \Z^{k+1}$ define $I_{(n_1, \dots, n_{k+1})}$ to be $I_{(n_1, \dots, n_{k+1})} = g_{(n_1, \dots, n_{k})}(I_{n_{k+1}})$.

    The above defines, inductively, sets \[X^{(k)}_{(n_1, \dots, n_k)} \subset I_{(n_1, \dots, n_k)} \subset I_{(n_1, \dots, n_{k-1}) }\subset \cdots \subset I_{n_1} \subset [0,1]\]
    for all $k\geq 1$ and all $(n_1, \dots, n_k) \in \Z^k$.

    Let $X^{(0)} = X$, and for $k\geq 1$ define $X^{(k)} = \bigcup_{(n_1, \dots, n_k) \in \Z^k} X^{(k)}_{(n_1, \dots, n_k)} $. 
    
    Finally,  define $X^{\infty} = \bigcup_{k\geq 0} X^{(k)} $.
\end{definition}

    Now, we show some properties of $X^{(k)}, X^\infty$.
\begin{proposition}\label{proposition:BF_X_infty_properties}

    \begin{itemize}
        \item $X^\infty$ is dense in $[0,1]$.
        \item $X^\infty$ has the order type of $\mathbb{Q}$ with its usual order.
        \item  There exists an order preserving bijection between $X^{\infty}$ and the set of non-separated leaves (in either the stable or unstable foliation)
intersecting the interior of a lozenge $L\subset P_\varphi$.
        
        Moreover, if $L\in \mathcal{L}(T)$ for some tree of scalloped regions $T$, this bijection can be chosen so that $X^{(k)}$ is mapped to the set of non-separated leaves which have levels differing from the level of $T$ by at most $k+1$.

    \end{itemize}
\end{proposition}
\begin{proof}
   
\end{proof}

\begin{lemma}\label{lemma:order_preserving_nonsep}
    Let $(P_\varphi, \F^+, \F^-)$ and $(P_\psi, \F^+, \F^-)$ be bifoliated planes corresponding to totally periodic bifoliated planes, and let $L\subset P_\varphi, L' \subset P_\psi$ be lozenges of the same type.

    Then, there exist order-preserving bijections $ h^\pm: \Lambda^\pm_{\mathrm{non-sep}} (L) \to \Lambda^\pm_{\mathrm{non-sep}}(L')$ that is \emph{level-preserving}, i.e. $\Delta(h^\pm(x),h^\pm(y)) =  \Delta(x,y)$ for all leaves $x,y\in \Lambda^\pm_{\mathrm{non-sep}} (L)$
\end{lemma}

\begin{definition}
    Under the conditions of Lemma \ref{lemma:order_preserving_nonsep}, let $l_1^\pm, l_2^{\pm}$ be leaves containing the sides of $L$, and let $l_1'^\pm, l_2'^{\pm}$ be the leaves containing the corresponding sides of $L'$, so that any homeomorphism $L\to L'$ preserving the foliations and their orientations must send $l_i^\pm$ to $l'^\pm_i$. Let $U_i \subset P_\varphi$ be the connected component of $P_\varphi \setminus l_i^-$ that does not contain $L$, and let $U_i'\subset P_\psi$ be the corresponding connected components of $P_\psi \setminus l_i'^-$.

    We say that a level-preserving, order-preserving bijection $ h^+: \Lambda^+_{\mathrm{non-sep}} (L) \to \Lambda^+_{\mathrm{non-sep}}(L')$ is \emph{admissible} if for every leaf $l^+ \in \Lambda^+_{\mathrm{non-sep}} (L)$, the unique corner of a lozenge that belongs to $l^+$ is in $\overline{U_i} \subset P_\varphi$ if and only if the unique corner of a lozenge that belongs to $h^+(l^+)$ $\overline{U_i'} \subset P_\psi$.

\end{definition}

\begin{theorem}\label{theorem:BF_equivalent_planes}
    Let $\varphi:M\to M$, $\psi:N\to N$ be two totally periodic Anosov flows. Then, the bifoliated planes $(P_\varphi, \F^+, \F^-)$ and $(P_\psi, \mathcal{G}^+, \mathcal{G}^-)$ are equivalent.
\end{theorem}

\begin{proof}
Let $(P_\varphi, \F^+, \F^-)$, $(P_\psi, \F^+, \F^-)$ be the bifoliated planes corresponding to $\varphi$ and $\psi$, respectively. 

We build an isomorphism $f:P_\varphi \to P_\psi$ by taking the union of inductively defined maps $f_n:P_n \to P_n
'$ for $n\geq 0$, where:

\begin{itemize}
\item $P_n \subset P_{n+1}$ for all $n$.
    \item $\bigcup P_n = P_\varphi, \, \,\bigcup P_n' = P_\psi$.
    \item $f_n|_{P_{n-1}} = f_{n-1}$ for all $n$.
    \item Each map $f_n$ sends the (un)stable foliation of $P_n$ to the (un)stable foliation of $P_n'$.
    \item Each $P_n$ is a connected union of trees of scalloped regions.
\end{itemize}

We begin by constructing $f_0 : P_0 \to P'_0$. Let $T_0\subset P_\varphi$ be a tree of scalloped regions. We will define $f_0$ on $P_0 = T_0$.

We construct $f_0$ on $T_0$ inductively by defining it lozenge by lozenge: we start by defining it on an initial lozenge $L_0\in \mathcal{L}(T_0)$. Then, at each subsequent step of the construction of $f_0$, if $f_0$ has been defined on a lozenge $L$, we extend it to all lozenges in $\mathcal{L}(T_0)$ which are adjacent to $L_0$. Since all lozenges in $\mathcal{L}(T_0)$ can be connected to $L_0$ by a chain of adjacent lozenges, this will define $f_0$ on all of $T_0 = P_0$.

Choose a lozenge $L_0 \in \mathcal{L}(T_0)$, and a lozenge $L_0' \in \mathcal{L}(T_0')$, where $T_0' \subset P_\psi$ is some tree of scalloped regions in $P_\psi$ and $L_0'$ is a lozenge of the same type as $L_0$. Our map $f_0$ will satisfy $f_0(L_0)=L_0'$, so our choices of $L_0$ and $L_0'$ completely determine the induced map $\bar{f_0}:\mathcal{L}(T_0)\to \mathcal{L}(T_0')$

Pick admissible level-preserving bijections $h_0^\pm: \Lambda^\pm_{\mathrm{non-sep}} (L_0) \to \Lambda^\pm_{\mathrm{non-sep}}(L_0')$.

Then, we define $f_0|_{L_0}:L_0 \to L_0'$ by first mapping each point $z\in L_0$ such that $\{ z\} = l^+\cap l^- $ (for $l^\pm \in \Lambda^\pm_{\mathrm{non-sep}} (L_0)$) to the point $h_0^+(l^+)\cap h_0^-(l^-)$. Since points of this form are dense in $L_0$ and $L_0'$ and the $h_0^\pm$ are order preserving, we can then define $f_0(z)$ for all points in $L_0$, so that $f_0|_{L_0}:L_0 \to L_0$ is a homeomorphism which preserves the bifoliations.

Now we describe the inductive step in the construction of $f_0$. Let $L, L_1 \in \mathcal{L}(T_0)$ be adjacent lozenges such that $f_0$ has already been defined in $L_1$. Without loss of generality, we assume that they are adjacent along a stable side, i.e. a half-leaf $l$ in $\F^-$. Let $L_1' = f_0(L_1) \in  \mathcal{L}(T_0')$ be the image of $L_0$ under $f_0$, and let $l' = f_0(l\cap L_1)$. Then, let $L' \in \mathcal{L}(T_0')$ be the lozenge adjacent to $L_1'$ along the side $l'$. We will extend $f_0$ to $L$ so that $f_0(L) = L'$. Note that once $f_0$ has been extended to $L$, we will have defined admissible level-preserving bijections $h^\pm: \Lambda^\pm_{\mathrm{non-sep}} (L) \to \Lambda^\pm_{\mathrm{non-sep}}(L')$.

Since $L$ and $L_1$ are adjacent along a stable side, we have that $\Lambda^+_{\mathrm{non-sep}} (L) = \Lambda^+_{\mathrm{non-sep}} (L_1)$. Analogously, $\Lambda^+_{\mathrm{non-sep}}(L') = \Lambda^+_{\mathrm{non-sep}}(L_1') $. Then, we define $h^+: \Lambda^+_{\mathrm{non-sep}} (L) \to \Lambda^+_{\mathrm{non-sep}}(L')$ by simply setting $h^+ = h_0^+$. 

There are now two possible cases. Either in a previous step we have defined an admissible level-preserving bijection $h^-: \Lambda^-_{\mathrm{non-sep}} (L) \to \Lambda^-_{\mathrm{non-sep}}(L')$, or we have not. The case where such a map has already been defined occurs exactly when, in a previous step, we have already defined $f_0$ on a lozenge adjacent to $L$ along an unstable side.

If the admissible level-preserving bijection $h^-$ has been defined, then both maps $h^\pm: \Lambda^\pm_{\mathrm{non-sep}} (L) \to \Lambda^\pm_{\mathrm{non-sep}}(L')$ have been defined, so we can define $f_0$ on all of $L$ as we did for $L_0$. 

If the map $h^-$ has not been defined, we pick an admissible level preserving bijection $\Lambda^-_{\mathrm{non-sep}} (L) \to \Lambda^-_{\mathrm{non-sep}}(L')$ and define $h^-$ as this map. Then, as we did above, we extend $f_0$ to $L$ using the admissible level-preserving bijections $h^\pm$.

This concludes the inductive step in the definition of $f_0$ on the tree of scalloped regions $T_0 = P_0$.

We have then described a procedure that allows us to inductively define, for trees of scalloped regions $T$ and $T'$, a homeomorphism $f:T\to T'$ which preserves the bifoliation. Note that in the base case above we chose how this map was defined on the initial lozenge, but the procedure would in fact allow us to extend any such choice of map on an initial lozenge to a map defined in the whole tree of scalloped regions.

Now, we describe how we extend $f_0$ to other trees of scalloped regions.

Let $\mathcal{T}_0 = \{ T_0 \}$. For $n\geq 1$, let $\mathcal{T}_n$ be the set of all trees of scalloped regions $T\subset P_\varphi$ such that there exists $T_1 \subset P_{n-1}$ such that $T_1 \cap T = S$, where $S$ is a scalloped region.

Then, define $P_n = P_{n-1} \bigcup \left( \bigcup_{T\in \mathcal{T}_n} T \right)$.

The collections of trees of scalloped regions $\mathcal{T}_n$ are always countably infinite for $n\geq 1$. For each $n\geq 1$, we pick some enumeration of the trees of scalloped regions in $\mathcal{T}_n$ and write $\mathcal{T}_n = \{ T^n_0, T^n_1, T^n_2,\dots \}$.

Now, we describe how to define $f_n:P_n \to P_n'$, assuming that $f_{n-1}:P_{n-1}\to P_{n-1}'$ has been defined.

We once again construct the extension $f_n$ of $f_{n-1}$ via an inductive procedure. First, we define $f_{n}$ on $T^n_0$. Then, assuming that $f_{n}$ has been defined on $T^n_0, T^n_1, \dots, T^n_{k-1}$, we extend it to $T^n_k$. This will define $f_n$ in $P_n$.

We will make use of the following:
\begin{claim}
    For each $k\geq 0$, there exists a unique tree of scalloped regions $T_k \subset P_{n-1}$ such that $T_k^n$ and $T_k$ intersect on a scalloped region $S_k$.

    Moreover, for all $k,j$ such that $k\neq j$ we have $T_k^n \cap T_j^n \neq \emptyset$ if and only if they intersect on a nonempty subset of $P_{n-1}$.
\end{claim}
This is a consequence of Lemma \ref{lemma:BF_no_tree_loop}, which we can see by projecting to the orbit space the sets $U_0, \dots,U_n $ in the statement of the lemma.

The second part of the claim above tells us that our extension of $f_{n-1}$ to $T_k^n$ will be independent of the extension to $T_j^n$ if $k\neq j$. For this reason, it's enough to only show how we define $f_n$ on $T^n_0$, extending $f_{n-1}$.

Let $T \subset P_{n-1}$ such that $T\cap T^n_0 $ is a scalloped region.

First, we need to determine the tree of scalloped regions $(T^n_0)' \subset P_\psi$ such that we will have $(T^n_0)' = f_{n}(T^n_0)$. Since $T\cap T^n_0$ is a scalloped region $S_0$, some boundary leaf of $T$ contains the side $l$ of a lozenge $L$ in $\mathcal{L}(T^n_0)$. 

Since $f_{n-1}$ is already defined on $P_{n-1} \supset T_1$, it is already defined on $l$. Therefore, we will choose $(T^n_0)'\subset P_\psi$ to be the tree of scalloped regions such that $L' \in \mathcal{L}((T^n_0)')$, where $L' \subset P_\psi$ is the lozenge with side $f_{n-1}(l)$. 

Now, the domain of the map $f_{n-1}$ contains the scalloped region $S_0 = T^n_0 \cap T$, so we do not need to define $f_n$ on $S_0$. In fact, this must be the only subset of $T$ where $f_{n-1}$ is defined, by our claim above.

We can then define $f_n$ on $T^n_0$ lozenge by lozenge, by the same inductive procedure by which we defined $f_0$ on all of $P_0$ after having defined it on a single lozenge. 

Given a lozenge $L\in\mathcal{L}(T)$ adjacent to some lozenge $L_1 \in \mathcal{L}(T)$ such that $L_1 \subset S$, exactly one of the maps $h^\pm: \Lambda^\pm_{\mathrm{non-sep}}(L) \to \Lambda^\pm_{\mathrm{non-sep}}(f_{n-1}(L))$ has been defined previously, so once we make a choice for the other map, we will be able to define $f_n$ on $L$, mapping $L$ into a lozenge $L' = f_n(L)$ adjacent to $f_{n-1}(L_1)$. 

One can then define $f_{n}$ on the remaining lozenges in $\mathcal{L}(T^n_0)$, as we did for the definition of $f_0$ on $T_0$.

\end{proof}

\section{Action of $\pi_1(M)$ on $P_\varphi$}\label{section:BF_action_pi1}

From the construction of the manifold $M$ as a gluing of pieces $N(X_i)$, one can get an explicit presentation for the fundamental group $\pi_1(M)$ of $M$. Then, by using this presentation and the construction of $M$ it is possible to understand the action of the generators of $\pi_1(M)$ on $P_\varphi$. In this section, we will show how to do this in a simple example where there is only one piece $N(X)$, and the fat graph $X$ has only one vertex and two edges. The same techniques apply for the general case, although the presentation of the group becomes more unwieldy as the complexity of the manifold grows.

\subsection{The Bonatti-Langevin example}\label{subsection:BF_BL_example}

First, we construct the Bonatti-Langevin flow, originally defined in \cite{bonatti1994exemple}, in terms of the Barbot-Fenley construction of totally periodic Anosov flows.

Let $\Sigma$ be the surface with boundary obtained by removing two open disks from a projective plane. One can see that $\Sigma$ is homeomorphic to a Moebius band (with boundary) minus an open disk.
\begin{figure}[h]
  \centering
  \includegraphics[width=0.6\textwidth]{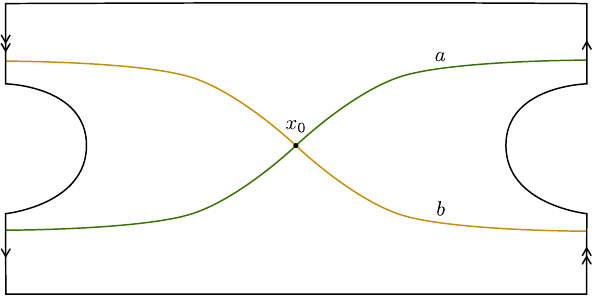}
  \caption{The figure-8 fatgraph $X \subset \Sigma$.}
  \label{fig:fatgraph}
\end{figure}

Let $X \subset \Sigma$ be a figure-8 fat graph embedded in $\Sigma$, with its two edges labeled $a$ and $b$ as in Figure \ref{fig:fatgraph}, and its vertex labeled $x_0$. We can see that this is an admissible fat graph as in Definition \ref{definition:fatgraph}. Since $x_0$ is the unique vertex in $X$, in this case we see that each edge of $X$ defines an element of the fundamental group of $\Sigma$. In fact, since $\Sigma$ deformation retracts onto $X$, we have
\[
\pi_1(\Sigma,x_0)  \cong F_2 = \langle a,b\rangle
\]

Now, take blocks $N_a, N_b$ equipped with a local flow as in Section \ref{section:BF_background}, corresponding to $a$ and $b$. After gluing their boundary components which are tangent to the local flows according to the Barbot-Fenley construction (see Figure \ref{fig:N(X)} for a vertical view of $N(X)$), we obtain the $3$-manifold with boundary $N(X)$ which is the nontrivial circle bundle over $\Sigma$ such that the holonomy $\rho : \pi_1(\Sigma) \to \mathrm{Homeo}(S^1)$ satisfies $\rho(a) = \rho(b) = x \mapsto -x$, where we think of $S^1$ as $S^1 = \R/ \Z$.

\begin{figure}[h]
  \centering
  \includegraphics[width=0.6\linewidth]{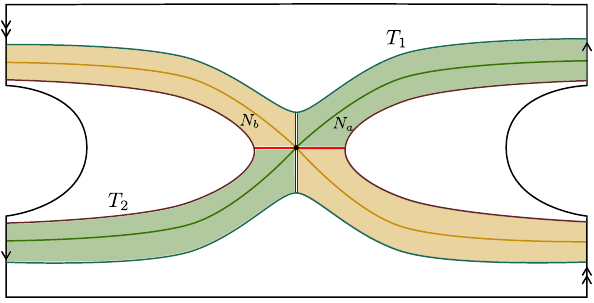}
  \caption{View of $N(X)$ from above.}
  \label{fig:N(X)}
\end{figure}

We can then get a presentation for the fundamental group of $\Sigma$. First, note that the vertical circle fiber over $x_0\in \Sigma$ is a periodic orbit of the local flow on $N(X)$. In Figure \ref{fig:N(X)}, this orbit is oriented so that in the future it goes into the page. Let $h$ denote the homotopy class of the loop given by traversing this orbit exactly once. Then, we have
\[
\pi_1(N(X)) = \langle a,b,h : a^{-1}ha =  b^{-1}h b = h^{-1} \rangle.
\]
We have $\partial N(X) = T_1 \cup T_2$, where each $T_i$ is a torus transverse to the flow. We label them so that the flow is incoming on $T_1$ and outgoing on $T_2$. It will be useful to define generators for their fundamental groups. Let $c_1 = ab$, and let $c_2 = ab^{-1}$. These are freely homotopic to loops in $T_1$ and $T_2$ respectively. Then, define $H_i$ to be the subgroup of $\pi_1(N(X))$ generated by $\{ h,c_i\}$. We can see from the relations above that $H_i \cong \Z^2$.

Now, we look at the universal cover $\widetilde{N(X)}$ of $N(X)$. Understanding the action of the deck group on $\widetilde{N(X)}$ will later allow us to understand the $\pi_1(M)$ action on the bifoliated plane corresponding to the Bonatti-Langevin flow. 

As we discussed in Section \ref{subsection:BF_trees_of_scalloped}, we can think of $\widetilde{N(X)}$ as a product $\widetilde{\Sigma} \times \R$. The surface $\widetilde{\Sigma}$ deformation retracts onto a fat graph $\tilde{X} \subset \widetilde{\Sigma}$, which is an infinite tree where each vertex has valence four. We may then identify $\tilde{X}$ with a graph embedded in $\widetilde{\Sigma} \times \{ 0\} \subset \widetilde{N(X)}$. We pick deck transformations representing $a$ and $b$ so that they preserve the surface $\widetilde{\Sigma} \times \{ 0\}$ (and flip the vertical fibers), and pick a deck transformation representing $h$ so that $h$ translates all of $\widetilde{\Sigma}\times \R$ vertically, via the map $(x,y, z) \mapsto (x,y,z-1)$.

The graph $\widetilde{X}$ is essentially the Cayley graph of the subgroup $\langle a,b \rangle < \pi_1(N(X))$: if we let $\tilde{x}_0 \in \widetilde{X}$ be one of the vertices (a lift of $x_0 \in N(X)$), we may uniquely identify each other vertex with an element $g\in \langle a, b \rangle $ via the map $\tilde{x}_0 \mapsto g \tilde{x}_0$. In the general case this will not be the case, since one might have edges in the fat graph that do not represent loops in $N(X)$.

Each of the vertices of $\widetilde{\Sigma}$ (and the vertical fiber over them, an orbit of the flow) lies in the intersection of four blocks of the form $I\times \R \times I$, for $I = [-\pi/2, \pi/2]$. These will correspond to the four lozenges meeting at this orbit in the bifoliated plane of the manifold. The action of $h$ preserves all these blocks, since it is only a vertical translation. In particular, each vertical orbit passing through a vertex of the graph is preserved by $h$.

The boundary components of $\widetilde{N(X)}$ are planes, each of them a lift of $T_1$ or $T_2$, so that the flow is transverse to these planes. Each block $I \times \R \times I$ intersects two boundary components, one of them a lift of $T_1$ (where the flow is incoming), and one a lift of $T_2$ (where the flow is outgoing). Two blocks of this form either do not intersect, intersect along an infinite strip which is a stable or unstable manifold for the flow, or intersect in a vertical fiber. We say that they are adjacent if they intersect along an (un)stable manifold. Each boundary plane of $\widetilde{N(X)}$ is a union of transverse boundaries of a countable collection of blocks, so that each of the blocks in the collection is adjacent to two others, where either any two blocks are always adjacent along stable manifolds or always adjacent along unstable manifolds. These boundary planes will later project to scalloped regions in the bifoliated plane. We can see that all of them are invariant under the action of $h$, and each of them is invariant under a unique conjugate of the subgroups $H_1, H_2$.

Before discussing the gluing of the boundary components of $N(X)$ to obtain a closed $3$-manifold $M$, we will briefly describe some useful Birkhoff annuli in $N(X)$. Let $\pi: N(X) \to \Sigma $ be the bundle projection. Consider loops $c_1', c_2'$ which are homotopic to $c_1$ and $c_2$ and whose images are embedded circles in $\Sigma$. Abusing notation slightly by using the same name for $a,b$ as elements in the fundamental group of $\Sigma$ and as loops in $\Sigma$ in their homotopy class which are embedded circles, let 
\begin{itemize}
\item $A_i =  \pi^{-1}(c_i'), i = 1,2$.
       \item $B_1 = \pi^{-1}(a) $.
    \item $ B_2 = \pi^{-1}(b)$.
\end{itemize}
\begin{figure}[h]
  \centering
  \includegraphics[width=0.7\linewidth]{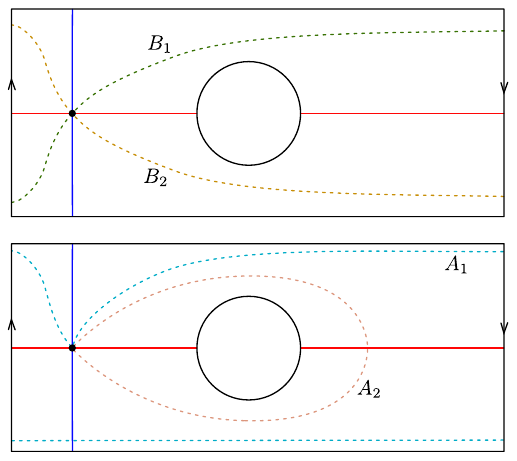}
  \caption{The Birkhoff annuli $A_i$ and $B_i$.}
  \label{fig:birkhoff_good}
\end{figure}

Then, $A_i, B_i$ are weakly embedded Birkhoff annuli in $N(X)$ (recall that this means that they are embedded except possibly at the periodic orbits contained in them). The Birkhoff annuli $A_i$ are tori while the $B_i$ are Klein bottles. All of them contain the periodic orbit $\pi^{-1}(x_0)$ which represents $h\in \pi_1(N(X))$, and we can see that by definition, each $A_i$ is $H_i-$invariant, while $B_1, B_2$ are invariant under $a,b$ respectively. Recall that by Proposition \ref{proposition:BF_background_block}, these are elementary Birkhoff annuli, that is, the foliations induced by the stable and unstable foliations of the flow on the interior of $A_i$ and $B_i$ have no closed leaves.

Now, we finally construct the manifold $M$ where the Bonatti-Langevin flow is defined. In order to do this, we introduce coordinates on the tori $T_1$ and $T_2$. We only need to determine the gluing map $A: T_1\to T_2$ up to isotopy, since Theorem D' in \cite{barbot2015classification} ensures that the resulting flow will be well defined up to orbit equivalence. Let $\mathcal{G}_1, $  $ \mathcal{G}_2$ be the horizontal and vertical foliations on $T_1$, and let $\mathcal{G}_1', \mathcal{G}_2' $ be the horizontal and vertical foliations on $T_2$. Orient the foliations using the vector fields $Z_1, Z_2$ defined on $N_a$ as $Z_1(x,y,z) = (1,0,0), Z_2(x,y,z) = (0,1,0) $. Now, let $A:T_1\to T_2$ be a homeomorphism switching the vertical and horizontal foliations, reversing orientation on each of them (so that the map $A$ preserves orientation). Then, let $M$ be the manifold obtained by identifying $T_1$ and $T_2$ using $A$.

Given our definition of the gluing map $A$, and how we defined the elements $h, c_1, $ and $c_2$, we can see that by taking a loop $t\in \pi_1(M)$ which is the union of two embedded straight line paths joining $x_0$ with $T_1$ and $T_2$ with $x_0$, we get that a presentation of $\pi_1(M,x_0) $ is given by
\begin{align*} 
\langle  a,b,c_1, c_2, h, t \mid c_1 = ab, c_2 = ab^{-1}, aha^{-1} = bhb^{-1} = h^{-1}, tc_2t^{-1} = h, tht^{-1} = c_1^{-1} \rangle.
\end{align*}

Let $(P, \F^+, \F^-)$ be the bifoliated plane associated to the Anosov flow in $M$. We will now study the action of $\pi_1(M)$ on $P$.

The universal cover $\widetilde{M}$ of $M$ consists of countably many copies of $\widetilde{N(X)}$, glued along their boundary planes, the lifts of $T_1$ and $T_2$. Let $\widetilde{N} \subset \widetilde{M}$ be the copy of $\widetilde{N(X)}$ containing $\widetilde{x}_0$, i.e. $\widetilde{N}$ is the connected component of $p^{-1}(N(X))$ that contains $\widetilde{x}_0$, where here $p: \widetilde{M} \to M$ is the universal covering map.

Let $q: \widetilde{M} \to P$ be the quotient map which collapses the orbits of the flow. We know that the image of $\widetilde{N}$ under $q$ is a tree of scalloped regions, which we will refer to as $T_0$. Let $\tilde{A_i}, \tilde{B_i} $ for $i=1,2$ be the connected components of  $\pi^{-1}(A_i), \pi^{-1}(B_i)$ that contain the orbit through $\widetilde{x}_0$, and let $\hat{A}_i = q(\tilde{A}_i), \hat{B}_i = q (\tilde{B}_i)$ be their projection to $P$.

The first statement in the following holds by our discussion of scalloped regions in Section \ref{subsection:BF_trees_of_scalloped}. The second statement is a consequence of our definition of the subgroups $H_i < \pi_1(M)$ and the element $h\in \pi_1(M)$.

\begin{proposition}
    The sets $\hat{A}_i \subset P$, for $i=1,2$ are scalloped regions, which are unions of lozenges in $\mathcal{L}(T_0)$.
    
    Moreover, $\hat{A}_i$ is preserved by the subgroup $H_i < \pi_1(M)$, and the element $h\in \pi_1(M)$ fixes all the corners of lozenges in $\mathcal{L}(T_0)$.
\end{proposition}

Let $L_0 \in \mathcal{L}(T_0)$ be the lozenge given by $L_0 = \hat{A}_1 \cap \hat{A}_2$. It has corners $x_0$ and $a\cdot x_0$. Let $L_0'$ be the lozenge adjacent to $L_0$ on its stable side, and which also has $a\cdot x_0$ as a corner. We can see that $L_0$ is an odd lozenge, $L_0'$ is an even lozenge, and the union $L_0 \subset L_0'$ is a fundamental domain for the action of $H_1 $ on $\hat{A}_1$. Moreover, any lozenge $L\in \mathcal{L}(T_0)$ is the image of one of $L_0$ and $L_0'$ under some element of the subgroup $H$ generated by $a,b$ and $h$. Using the labels on the vertices of the graph $\widetilde{X} \subset \widetilde{\Sigma}$, we can then label all corners of lozenges in $\mathcal{L}(T_0)$.

\begin{figure}[h]
  \centering
  \includegraphics[width=0.8\linewidth]{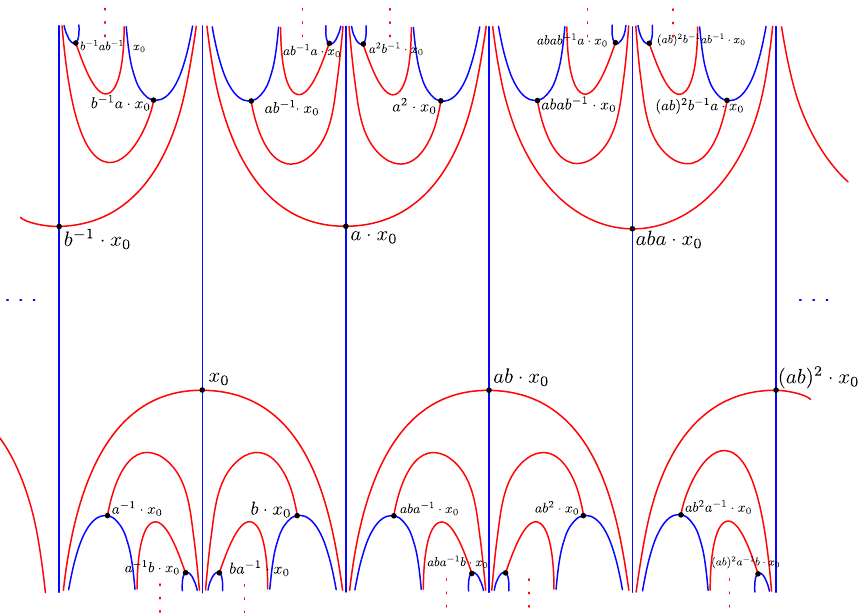}
  \caption{Labels for the corners of lozenges in $\mathcal{L}(T_0)$.}
  \label{fig:tree_labels}
\end{figure}

\begin{lemma}
We have 
    $\mathrm{Stab}(T_0) = H$.
\end{lemma}

The following allows us to understand the action of the elements $a,b$ on $T_0$:

\begin{definition}
    A chain of lozenges is a \emph{string} of lozenges if no two lozenges in the chain are adjacent.
\end{definition}

\begin{proposition}
    There exists a string of lozenges preserved by $a$, and no point in $P$ is fixed by $a$. The same is true for $b$.
\end{proposition}

\begin{figure}[h]
  \centering
  \includegraphics[width=0.7\linewidth]{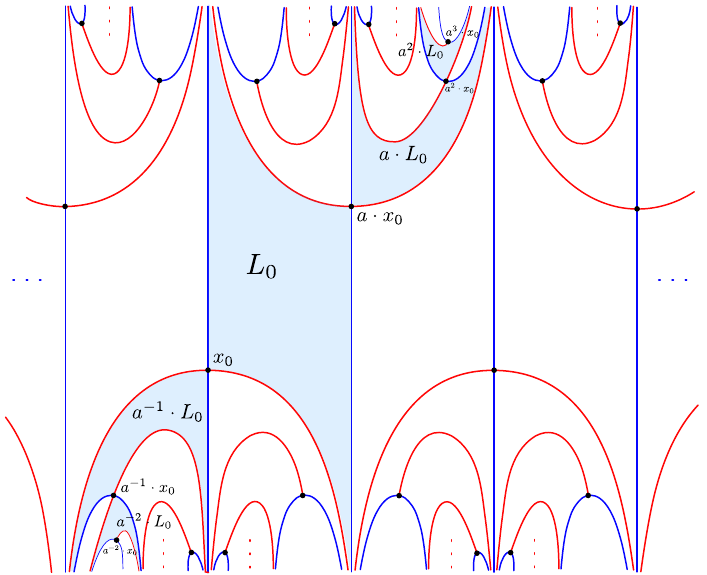}
  \caption{Invariant string of lozenges for $a$.}
  \label{fig:invariant_a}
\end{figure}

\begin{proof}
    From our identification of the corner points in the tree of scalloped regions $T_0$ (see Figure\ref{fig:invariant_a}), we can see that the lozenges $a^n \cdot L_0$ for $n\in \Z$ form a chain of non-adjacent lozenges which is by definition preserved by $a$.
    
    Suppose there existed some point $x\in P$ fixed by $a$. This would mean that the flow $\phi$ on $M$ has some periodic orbit which is freely homotopic to the loop represented by $a$. However, any such periodic orbit would have intersection number $0$ with the torus transverse to the flow, and we know that this is not the case for any periodic orbit except $h$. Since $a$ and $h$ are not freely homotopic, we conclude that $a$ cannot have any fixed points in $P$.
\end{proof}

We have then understood the action of $a,b, c_1,c_2$ and $h$ on the lozenges in $\mathcal{L}(T_0)$. The elements $a$ and $b$ preserve a string of lozenges\, $h$ preserves all lozenges (since it fixes all their corner points), and $c_1, c_2$ preserve scalloped regions $\hat{A}_1, \hat{A}_2$, in such a way that if we numbered the lozenges in $\hat{A}_i$ in an order preserving way, $c_i$ would act by translation by $2$ (recall that even lozenges can only be mapped to even lozenges, and the analogous statement holds for odd lozenges).

However, we are also interested in understanding \begin{enumerate}[(i)]
    \item How an element that preserves a lozenge or scalloped region acts in the interior of said lozenge or scalloped region.
    \item How the elements act on the \emph{closure} of $T_0$, that is, we want to understand the action on the boundary leaves of $T_0$.
\end{enumerate}

We begin by discussing the first item. We know that $h$ fixes all corners of lozenges in $\mathcal{L}(T_0)$. Additionally, we know that a loop representing $h$ is the loop formed by traversing the orbit of $\tilde{x}_0$ exactly once. Since the loop and the coincide in orientation, we can see (see Figure \ref{fig:action_h_leaf}) that $h$ must have $\hat{x}_0 = q(\tilde{x}_0)$ as an expanding fixed point on $\F^-(\hat{x}_0)$. Therefore, since the action of $\pi_1(M)$ is Anosov-like, the action of $h$ on $\F^+(\hat{x}_0)$ must have $\hat{x}_0$ as a contracting fixed point. 

\begin{figure}[h]
  \centering
  \includegraphics[width=0.6\linewidth]{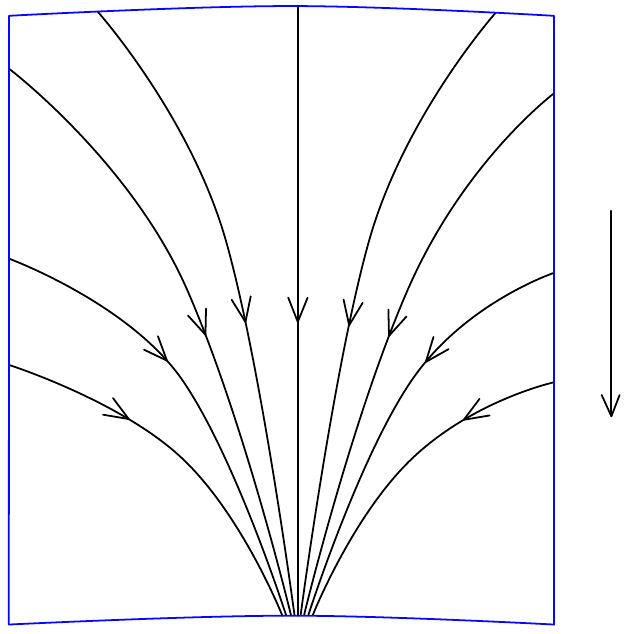}
  \caption{The action of $h$ on the stable leaf of $\widetilde{x_0}$.}
  \label{fig:action_h_leaf}
\end{figure}

Since $\hat{x}_0, a \cdot \hat{x}_0$ are corners of the lozenge $L_0$, the action of $h$ on the stable and unstable leaves of $a\cdot \hat{x}_0$ has to be contracting on the stable and expanding on the unstable leaf. We can then inductively find what the action of $h$ on the stable and unstable leaf of any corner of a lozenge in $\mathcal{L}(T_0)$ is. The result is depicted partially in Figure \ref{fig:action_h_interior}.

\begin{figure}[h]
  \centering
  \includegraphics[width=0.7\linewidth]{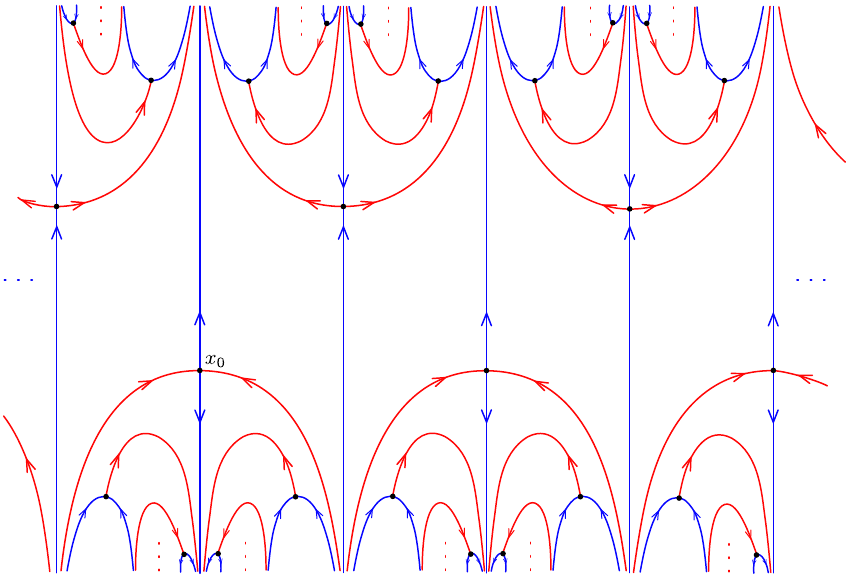}
  \caption{The action of $h$ on the interior of $T_0$.}
  \label{fig:action_h_interior}
\end{figure}

Now, we consider the elements that preserve a scalloped region which is a chain of adjacent lozenges in $\mathcal{L}(T_0)$. There are, up to the action of $H$, only two such scalloped regions: those in which the lozenges in $\mathcal{L}(T_0)$ that form them are adjacent on their stable sides, and those in which they are adjacent on their unstable sides. The first type can always be mapped to $\hat{A}_1$ via an element of $H$, and the second can be mapped to $\hat{A}_2$. Note, however, that there is no element of $H$ that maps $\hat{A}_1$ into $\hat{A}_2$.

Given the above, it is enough to understand the action of $c_1$ on the interior of $\hat{A}_1$, since the stabilizer of $\hat{A}_1$ is $H_1$, the subgroup of $H$ generated by $c_1$ and $h$. In order to do this, we will need to address point $ii)$ above. This is because the action of $c_1$ on the interior of $\hat{A}_1$ is essentially determined by its action on the $s$-boundary leaves of $\hat{A}_1$. These consist of two sets of non-separated leaves, where each set must be preserved by $c_1$, since $c_1$ preserves $\hat{A}_1$. However, a priori it is not true that $c_1$ preserves each of these boundary leaves individually (there exist elements of $H_1$ which do not, despite preserving $\hat{A}_1$). 

Before this, we show 
\begin{lemma}\label{lemma:unique_fixed_scalloped}
For $i =1,2$, the scalloped region $\hat{A}_i$ is the unique scalloped region which is preserved by the subgroup $H_i < \pi_1(M)$ generated by $h$ and $c_i$.
\end{lemma}

\begin{proof}
    Without loss of generality, we prove it for $i = 1$. We know that $c_i$ permutes the $u$-boundary leaves of $\hat{A}_1$, so for all $u$-boundary leaves $l^+$ of $\hat{A}_1$, any other scalloped region $\hat{A}_1'$ preserved by $H_i$ would need to be contained in the same connected component of $P\setminus l^+$ as $\hat{A}_1$ (it's easy to see that $\hat{A}_1'$ cannot intersect any of these leaves). 

    Now, consider the action of $h$ on the set of $s$-boundary leaves of $\hat{A}_1$. Suppose there existed one such leaf $l^{-}$ such that $l^-$ is preserved by $h$. Then, the Anosov-like axioms guarantee that there exists a fixed point for $h$ in $l^-$. Since the interior of $\hat{A}_1$ is trivially foliated and $h$ has fixed points in the $u$-boundary of $\hat{A}_1$, we must have that there exists an additional fixed point for $h$ on $l^-$, but this is not possible. Therefore, we have shown that $h$ does not preserve any leaf in the $s$-boundary of $\hat{A}_1$. Substituting $c_1$ by $h$ in the first part of the proof, we conclude that any scalloped region $\hat{A}_1'$ preserved by $h$ must lie in the same connected component $P \setminus l^-$ as $\hat{A}_1$, for all stable boundary leaves $l^-$ of $\hat{A}_1$.

    Thus, we conclude that any scalloped region invariant under $H_1$ must be contained in $\hat{A}_1$. This is only possible if it is the same as $\hat{A}_1$, so we are done.
    
\end{proof}

\begin{lemma}
    We have $t \cdot \hat{A}_2 = \hat{A}_1$.
\end{lemma}

\begin{proof}
    By Lemma \ref{lemma:unique_fixed_scalloped} above, we know that the scalloped region $\hat{A}_i$ is the unique scalloped region which is preserved by the subgroup $H_i < \pi_1(M)$.

Recall that from our presentation for $\pi_1(M)$, we have that the following relations are satisfied:
\begin{enumerate}[(1)]
    \item $ t c_2 t^{-1} = h$.
    \item $tht^{-1} = c_1^{-1} $.
\end{enumerate}
    This implies that $t H_2 t^{-1} = H_1$. Since $tH_2t^{-1}$ preserves the scalloped region $t \cdot \hat{A}_2$, we must have that , by uniqueness, $t\cdot \hat{A}_1 = \hat{A}_2$.
\end{proof}

\begin{figure}[h]
  \centering
  \includegraphics[width=1\linewidth]{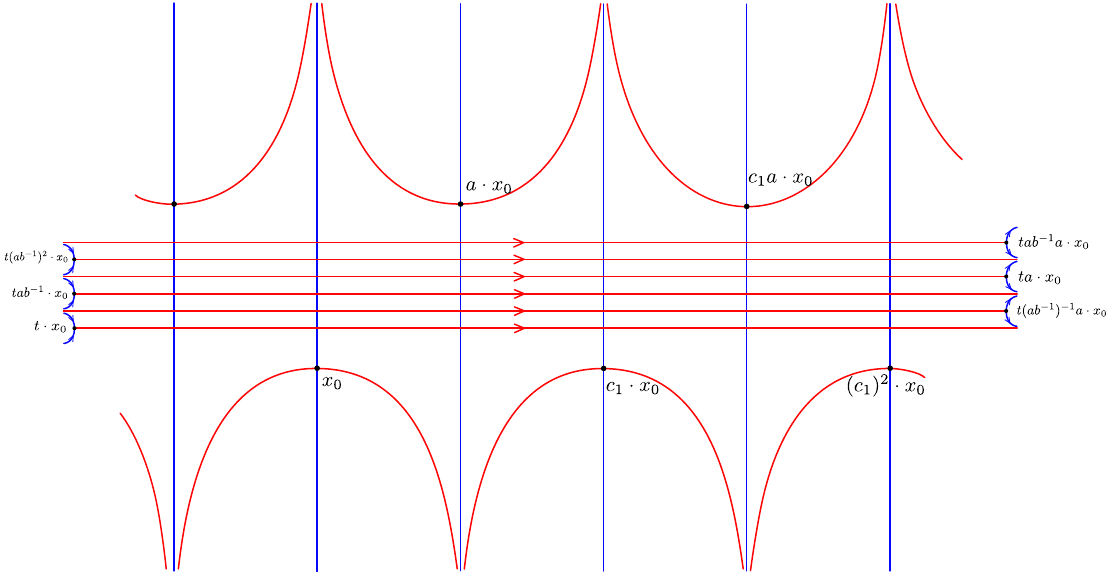}
  \caption{Boundary $s$-leaves of $\hat{A}_1$, and action of $c_1$ on them.}
  \label{fig:bf_action_c1_scalloped_boundary}
\end{figure}

\begin{remark}
    As a consequence of the above Lemma, we can now label all the points in the boundary of $\hat{A}_1$ and $\hat{A}_2$ that are in the $\pi_1(M)$-orbit of $\hat{x}_0$ (see figure \ref{fig:bf_action_c1_scalloped_boundary} for the case of $\hat{A}_1$). In fact (using that any scalloped region formed by lozenges in $\mathcal{L}(T_0)$ can be mapped into $\hat{A}_1$ or $\hat{A}_2$ by an element of $H$) we can label all the points in the $\pi_1(M)$-orbit of $\hat{x}_0$ which are in the boundary leaves of $T_0$. 
\end{remark}
Using the Lemma, we can completely describe how $h$ and $c_i$ act on the boundary leaves of $\hat{A}_i$. The following is a direct consequence of the relations $tc_2t^{-1} = h, \, tht^{-1} = c_1^{-1}$.
\begin{proposition}
The action of $h$ on the two sets of non-separated $s$-boundary leaves of $\hat{A}_1$ is transitive, and $c_1$ fixes each of these boundary leaves.

\begin{figure}[h]
  \centering
  \includegraphics[width=1\linewidth]{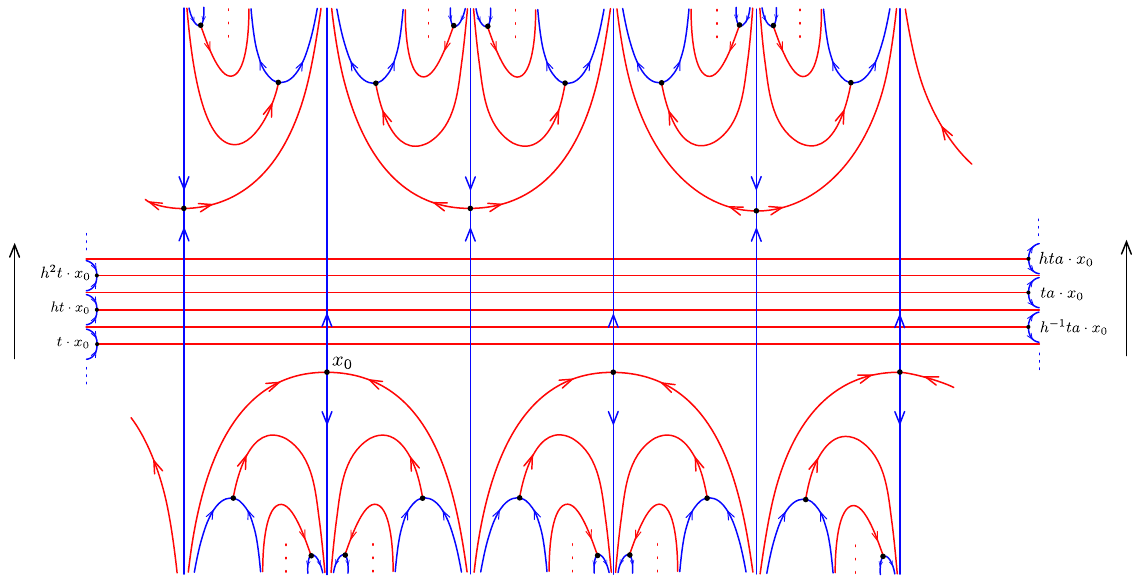}
  \caption{Action of $h$ on $s$-boundary leaves of $\hat{A}_1$}.
  \label{fig:bf_block}
\end{figure}

An analogous statement holds if we replace the $s$-boundary leaves of $\hat{A}_1$ by the $u$-boundary leaves of $\hat{A}_2$, and $c_1$ by $c_2$.
\end{proposition}

We have then understood the action of all elements of $H$ on $\overline{T_0}$. For any other tree of scalloped regions $T$, we know that $T$ can be sent to $T_0$ by some element of $\pi_1(M)$, so that the stabilizer of $T$ is just a conjugate of $H$. In order to complete our understanding of the action of $\pi_1(M)$ on $P$, we only need to understand elements that do not preserve any tree of scalloped regions, which means understanding the action of the element $t$.
We have seen above that $t^{-1} \hat{A}_1 = \hat{A}_2$. Since $\hat{A}_1$ is a chain of adjacent lozenges in $\mathcal{L}(T_0)$, this tells us what happens to $T_0$ under multiplication by $t^{-1}$:

\begin{proposition}
    The image $t^{-1} T_0$ of $T_0$ under $t^{-1}$ is a tree of scalloped regions such that $ \left(t^{-1}T_0 \right) \cap T_0 = \hat{A}_2$. 
    
    That is, $\hat{A}_2$ is a chain of lozenges in $\mathcal{L}(t^{-1}T_0)$ which are adjacent on their stable sides, so that $\Delta(t^{-1}T_0, T_0) = \- 1$ 
\end{proposition}

\begin{remark}
    By a similar reasoning, $t T_0$ is a tree of scalloped regions such that $\left( t T_0\right) \cap T_0 = \hat{A}_1$, and $\Delta( tT_0, T_0 ) = 1$. 
\end{remark}

By repeatedly applying $t$ we get
\begin{proposition}
    For $n\geq 1$, $t^n T_0$ is a tree of scalloped regions such that $\Delta(t^nT_0, T_0) = n$ and $ \left( t^{n}T_0 \right) \cap T_0 \subset \left(t^{n-1}T_0\right) \cap T_0 \subset \cdots \subset \left( t T_0\right) \cap T_0 \subset \hat{A}_1$.
\end{proposition}

An analogous statement holds for $n \leq -1$, replacing $\hat{A}_1$ with $\hat{A}_2$. 

\begin{proposition}
    There exists a point $z$ in the interior of $L_0$ which is fixed by $t$, and $\mathrm{Fix}(t) = \{ z\}$.
\end{proposition}

\begin{proof}
  
Since $L_0$ is a lozenge, the leaves of $\F^+$ ($\F^-$) that intersect $L_0$ form closed intervals $I^+$ (resp. $I^-$) in the leaf space $\Lambda^+$ (resp. $\Lambda^-$). Since $t\cdot \hat{A}_2 = \hat{A}_1$, the action of $t$ on $\Lambda^+$ maps $I^+$ into its interior. 

Therefore, there exists a leaf $l^+$ in the interior of $I^+$ which is preserved by $t$. This implies that $l^+$ contains a unique fixed point for $t$. Using the same argument for $t^{-1}$, we see that there exists a leaf $l^-$ in the interior of $I^-$ which is preserved by $t^{-1}$, and therefore by $t$. The intersection of $l^+$ and $l^-$ must then be a fixed point $z$ for $t$ and lie in the interior of $L_0$. 

By Lemma \ref{lemma:BF_interior_scalloped_no_corner}, the point $z$ cannot be the corner of a lozenge in $P$. Therefore, Proposition \ref{proposition:intro_fixed_points_joined_lozenges} tells us that $z$ is the unique fixed point of $t$, as we wanted.
    
\end{proof}

We have then described the action of all elements in our generating set on $T_0$ (or, in the case of the element $t$, what the image of $T_0$ under its action is). Their action on the rest of the plane can be deduced from what we have shown here: in a sense, for our choice of generators, the ``interesting'' dynamics occurs in the trees of scalloped regions $t\cdot T_0$, $T_0$ and $t^{-1}\cdot T_0$.

\chapter{Bifoliated planes of Franks-Williams flows}\label{chapter:FW}

\section{Introduction}

The first goal of this chapter is to describe the bifoliated plane $(P_A, \F^+_A, \F^-_A)$ associated to a Franks-Williams Anosov flow $\varphi_A: \R\times M_A \to M_A$, where $A \in \mathrm{SL}(2,\Z)$ is a hyperbolic matrix. Our second goal is to give a proof of Theorem \ref{introthm:FW_main_theorem} and Corollary \ref{introcor:FW_cor}.

We begin by briefly reviewing the construction of the Franks-Williams flows in Sections \ref{subsection:DA_blowup} and \ref{subsection:construction_of_FW_flow}, following the procedure outlined in \cite{FW1980anomalous} and \cite{YangYu2022classifying}. Section \ref{subsection:DA_blowup} discusses the ``derived from Anosov'' blowup applied to a linear hyperbolic torus map, and Section \ref{subsection:construction_of_FW_flow} the construction of the flow $\varphi_A$ and the $3$-manifold $M_A$ via a surgery procedure.

Then, in Sections \ref{subsection:intersections_leaves_universal_cover_DA} and \ref{subsection:cover_of_M1} we establish some facts regarding the configurations of stable and unstable leaves of the lift of $\varphi_A$ to a cover of a JSJ piece of $M_A$. The main goal of these sections is the proof of Proposition \ref{proposition:bijection_M1_integers}, which will later be a key component of the proof of Theorem \ref{theorem:FW_main_theorem}.

In Section \ref{subsection:infinite_perfect_fits} we first describe infinite perfect fits, a particular type of chain of lozenges which cover most of the bifoliated planes $(P_A, \F^+_A, \F^-_A)$. We describe the infinite perfect fits in $P_A$, and the regions of $\widetilde{M_A}$ which project to them. 

In Section \ref{subsection:crossing_and_non_crossing_IPFs} we complete our description of $(P_A, \F^+_A, \F^-_A)$ by discussing a particular type of infinite perfect fit in $P_A$. These will be the connection between the bifoliated plane and algebraic data coming from the matrix $A$. They will allow us to associate with the bifoliated plane $(P_A, \F^+_A, \F^-_A)$ an invariant consisting of an ordered set together with two distinguished subsets, where these sets are defined in terms of patterns of intersections of leaves and lozenges in the plane.

The main results in this section are Proposition \ref{proposition:equivalence_crossing_integer_points} and Corollary \ref{corollary:FW_isomorphic_implies_equivalent_patterns}. The former connects our invariant to intersection patterns of the integer lattice and lines parallel to the eigenspaces of $A$. The latter shows that our invariant is indeed an invariant of the bifoliated planes associated to Franks-Williams flows.

In Section \ref{subsection:definitions_elementary_props_continued_fractions} we review basic definitions and properties from the theory of continued fractions. Results in this section can be found in or obtained by using the same techniques as in the classical text \cite{khinchin1964continued}, so we do not provide proofs for all of them. We would like to thank Christophe Leuridan \cite{leuridan_2025} for informing us of the connection between the continued fraction approximation of an irrational number $\alpha$ and the behavior of the sequence $\alpha n \mod 1$. 

In Section \ref{subsection:parallelograms_integer_points_continued_fractions} we define a bi-infinite sequence associated with the contracting eigenspace of a matrix $A$ in terms of parallelograms in the plane which do not contain integer points. Then, we relate this sequence to the slope of this eigenspace, showing that the sequence must be periodic and have periodic part equal to the periodic part of the continued fraction expansion of the slope. This relationship can essentially be deduced, with some work, from the contents of Chapter 3 of Karpenkov's book \cite{karpenkov2022geometry}. We present our own proof, firstly for completeness, and secondly since properly showing how one deduces the results from \cite{karpenkov2022geometry} would require us to introduce unnecessary concepts and definitions from integer geometry.

Finally, in Section \ref{subsection:proofs_of_FW_theorems} we prove Theorem \ref{introthm:FW_main_theorem}, restated more precisely as Theorem \ref{theorem:FW_main_theorem}, and also prove Corollary \ref{introcor:FW_cor} using the results from the previous two sections.

\section{The Franks-Williams construction from a linear torus map}\label{section:FW_construction}

The original construction of the Anosov flow commonly known as \emph{the Franks-Williams flow} by Franks and Williams in \cite{FW1980anomalous} starts by considering the hyperbolic linear map $f_A:T^2\to T^2$ defined by the matrix $A= \begin{pmatrix}
    2&1\\1&1
\end{pmatrix} \in \mathrm{SL}(2,\Z) $, where $T^2$ is the two-dimensional torus. 

One then performs a DA (``derived from Anosov'', see \cite{williams1970DA} blowup on this map, obtaining a diffeomorphism $g_A$ of the torus which is semiconjugate to $f_A$, but (in contrast with $f_A$) has a proper hyperbolic attractor $\Lambda \subsetneq T^2$. 

However, this procedure can be carried out for any choice of hyperbolic matrix $A\in \mathrm{SL}(2,\Z)$. We begin this section by describing, in Section \ref{subsection:DA_blowup}, the relevant features of the map $g_A$ and its lift $\widetilde{g_A}$ to the universal cover of the torus.

In \ref{subsection:construction_of_FW_flow} we briefly describe how the flow $\varphi_A$ is constructed as well as some important features of the flow, following \cite{FW1980anomalous} and \cite{YangYu2022classifying}.

Then, in \ref{subsection:intersections_leaves_universal_cover_DA} we compare the intersection patterns of stable and unstable leaves for $\widetilde{f_A}$ and $\widetilde{g_A}$. We then define the notion of crossing and non-crossing points, which will be used later to define the invariant that allows us to determine certain bifoliated planes obtained from Franks-Williams flows are not isomorphic. The contents of subsection \ref{subsection:intersections_leaves_universal_cover_DA} are essential to the definition of this invariant and therefore to the proof of Theorem \ref{theorem:FW_main_theorem}, but are not a prerequisite for subsections \ref{subsection:construction_of_FW_flow} and \ref{subsection:infinite_perfect_fits}, so may be skipped initially if one wishes to first visualize a partial picture of the bifoliated plane $(P_A, \F^+_A, \F^-_A)$.

In \ref{subsection:cover_of_M1}, we define a cover of a JSJ piece of $M$ which will be useful when working with the universal cover $\widetilde{M}$ in order to understand the bifoliated plane $(P_A, \F^+_A, \F^-_A)$. We define crossing and non-crossing leaves of the foliation lifted to this cover, and relate these in Proposition \ref{proposition:bijection_M1_integers} to intersections of certain stable and unstable leaves for $\widetilde{g_A}$.

\subsection{DA blowup of a hyperbolic linear map}\label{subsection:DA_blowup}

The goal of this section is to give, without proofs, a brief description of the classical construction of the so-called ``DA'' map starting with a hyperbolic linear map on the torus. In Proposition \ref{proposition:DA_properties} we summarize some important properties of this map.

Let $A \in \mathrm{SL}(2,\Z)$ be a matrix satisfying $\mathrm{tr}(A) > 2$. Then, $A$ has two real eigenvalues $\lambda_+$ and $\lambda_-$, such that $0 < \left| \lambda_- \right| < 1< \left|\lambda_+\right|$. Since $A \in \mathrm{SL}(2,\Z)$, it induces a diffeomorphism $f_A :T^2 \to T^2$ of the torus $T^2 = \R^2/\Z^2$. 

The foliations of the plane by straight lines of irrational slope parallel to the eigenspaces $E^+, E^-$ of $A$ corresponding to $\lambda_+$ and $\lambda_-$ descend to foliations $\mathcal{G}^u, \mathcal{G}^s$ of the torus. Each foliation consists of dense leaves  in the torus which are immersed topological lines. 

The map $f_A$ is an Anosov diffeomorphism of the torus, as can be seen via a brief calculation, and the foliations $\G^s, \G^u$ are the stable and unstable foliations of $f_A$, respectively.

The first step in the construction of the flow $\varphi_A$ consists in doing a ``DA blowup'' to the map $f_A$. This procedure was first described by Smale in \cite{smale1967differentiable} and then by Williams who made the construction explicit in \cite{williams1970DA}.

The DA blowup applied to the map $f_A$ yields a diffeomorphism $g_A : T^2\to T^2$, which is semiconjugate to $f_A$ via a homeomorphism isotopic to the identity $h:T^2 \to T^2$, as a consequence of a theorem of Franks in \cite{franks1970anosov}.

A brief description of the procedure for defining the map $g_A$ is as follows: one takes a small neighborhood $U$ of the point $(0,0)\in T^2$ and modifies $f_A$ only on $U_0$. This is done in such a way that $(0,0)$ becomes a repelling hyperbolic fixed point for the modified map $g_A$, and $g_A$ has two hyperbolic fixed points of saddle type $z, z'\in \G^s(0,0)$, the stable manifold of $(0,0)$ for the original map $f_A$. 

 Considering the orbit $W = \bigcup_{n\in \Z} g_A^n(U_0)$ of a ball $U_0$ centered at $(0,0)$ that does not contain $z$ or $z'$, then with appropriate choices in the blowup procedure, one obtains a hyperbolic attractor $\Lambda = T^2 \setminus W $ for $g_A$. 
 
 The attractor $\Lambda$ is saturated by the unstable leaves of its points. For $x\in \Lambda$, we denote by $\F^s(x)$ and $\F^u(x)$ the stable and unstable leaves for $g_A$ which contain $x$. 
 
 Note that the semiconjugacy $h: T^2 \to T^2$ maps $z, z' \in T^2$ to the point $(0,0)$, and therefore maps $W$ and the unstable manifolds $\F^u(z), \F^u(z')$ to the unstable manifold $\G^u(0,0)$.

\begin{proposition}\label{proposition:DA_properties}
    The diffeomorphism $g_A$ satisfies (see for instance \cite{williams1970DA} or \cite{fisher2019hyperbolic}):

    \begin{itemize}
        \item The set $W$ is open and dense in $T^2$.
        \item The stable leaves of points in the attractor $\Lambda$ coincide with the leaves of the foliation $\mathcal{G}^s$.
        \item There exists a constant $\gamma <1$ such that any vector $v$ in the unstable distribution $F^u(z)$ for $g_A$ at each point $x$ in $\Lambda$ satisfies $v \in C_\gamma(E^u, E^s) =  \{u+w : u\in E^u, w\in E^s : \| w \| < \gamma \| u\|  \}$.

        \item Exactly one connected component of $\F^s(z)\setminus \{ z\}$ is contained in $W$ and limits onto $(0,0)$, and the same is true for $\F^s(z')$.
    \end{itemize}
\end{proposition}

The following proposition will be used in a later section:

\begin{proposition}\label{proposition:no_nullhomotopic_DA}
There is no closed non-nullhomotopic loop $\gamma \subset T^2$ contained in $W$. 

Consequently, if we let $\widetilde{g_A}: \widetilde{T^2} \to \widetilde{T^2}$ be the lift of $g_A$ to the universal cover $\widetilde{T^2} \cong \R^2$ of the torus, then $\pi|_{\widetilde{W}_0} : \widetilde{W}_0\to W$ is injective for every connected component $\widetilde{W}_0$ of $\pi^{-1}(W)$. In particular, for each such component $\widetilde{W}_0$ there exists a unique integer point $(n,m) \in \Z^2$ contained in $\widetilde{W}_0$.
\end{proposition}

\begin{proof}
    This is a straightforward consequence of the definition of $W$ as $W = \bigcup_n g_A^n(U_0)$ and the fact that $U_0 \subset g_A(U_0)$, where $U_0$ is an open disk: applying an appropriate power of $g_A$ to a closed loop $\gamma \subset W$ gives us a closed loop $\gamma' \subset U_0$, which must be nullhomotopic, so the same must be true of $\gamma$.
\end{proof}

\subsection{Construction of the flow $\varphi_A$}\label{subsection:construction_of_FW_flow}
Here we continue with the construction of the flow $\varphi_A$ from the DA map $g_A:T^2 \to T^2$, as done in \cite{FW1980anomalous}.

Consider the suspension flow $\psi_A : \R  \times N_A \to N_A$, where the 3-manifold $N_A$ is the mapping torus $T^2\times [0,1]/\sim_{g_A}$ of $g_A : T^2 \to T^2$. 

This flow has a hyperbolic attractor $\Lambda_A$, such that the properties stated above for the map $g_A$ and its attractor $\Lambda$ translate to analogous properties of $\psi_A$ and $\Lambda_A$. 

We identify $T^2$ with the set $T^2 \times \{ 0\} \subset N_A$. Then, $\Lambda$ is identified with $\Lambda \times \{0 \} \subset \Lambda_A$, and for $x\in \Lambda$ denote by $\F^s_A(x), \F^u_A(x) \subset N_A$ the weak stable and unstable leaves of each point $x\in \Lambda$. These leaves are immersed planes or annuli in $N_A$ which intersect along the orbit $\mathcal{O}(z) \subset \Lambda_A$ of $x$.

The next step in the construction is to remove a solid torus $U$ from $N_A$ centered around the orbit of $(0,0,0)$. Since the point $(0,0)$ is a repelling point for the map $g_A$, an appropriate choice of solid torus ensures that its boundary $\partial U = T$ is transverse to $\psi_A$, with the flow coming into the manifold-with-boundary $N_A \setminus U$. Note that we also have that $\Lambda_A \subset N_A \setminus U$.

Since the flow is transverse to the torus $T$, the stable leaves of points in $\Lambda_A$ intersect $T$ transversely, inducing a 1-dimensional foliation $f^s$ on $T$.
\begin{proposition}[\cite{FW1980anomalous}]
For an appropriate choice of $U$, the torus $T = \partial U$ satisfies the following properties:
    \begin{itemize}
        \item The foliation $f^s$ has exactly two compact leaves corresponding to $T\cap \F^s_A(z) $ and $T\cap \F^s_A(z')$.
        \item In the regions of the torus bounded by the two compact leaves, $f^s$ has Reeb components. These are oriented in such a way that the vertical coordinate of each leaf of $f^s$ in the interior of the components is bounded above in the universal cover of $N_A$.
    \end{itemize}
\end{proposition}

\begin{figure}[h!]
  \centering
  \includegraphics[width=0.55\linewidth]{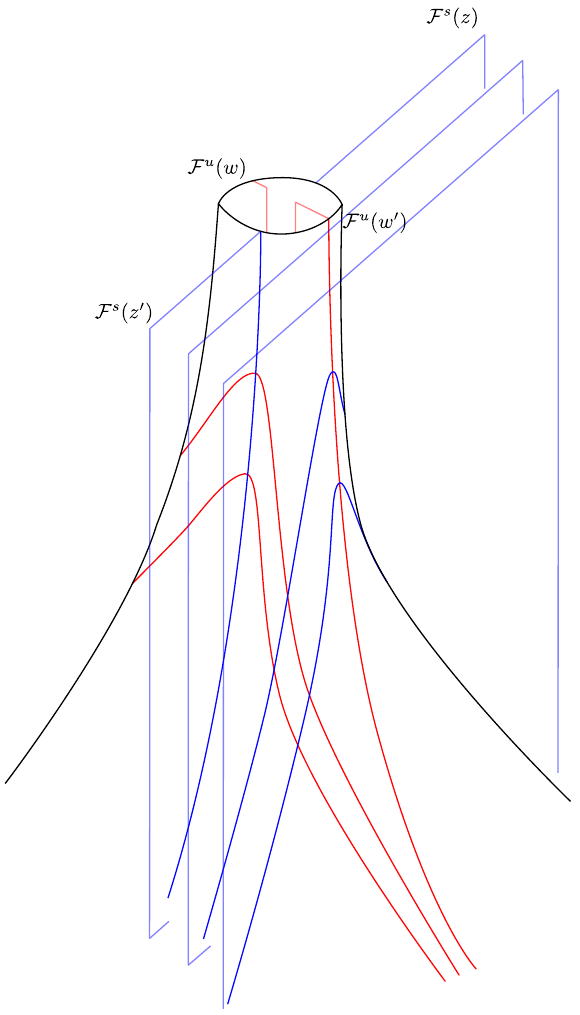}
  \caption{The transverse torus $T$ and the induced foliations.}
  \label{fig:FW_torus_foliation}
\end{figure}

The final step in the construction is to consider a new copy $\overline{N_A}$ of $N_A$ equipped with the reversed flow $\overline{\psi_A}$, which is transverse to the torus $T$, with orbits exiting the manifold $\overline{N_A} \setminus U$ in the future. The flow $\overline{\psi_A}$ has the same properties as $\psi_A$, replacing stable leaves with unstable leaves and the attractor for a repeller. In particular, the induced foliation $\overline{f^u}$ on the torus consists of a pair of Reeb components oriented in the same way as those of $f^s$.

One then glues to $N_A \setminus U$, $\overline{N_A}\setminus U$ collar neighborhoods of $\partial (N_A \setminus U), \partial(\overline{N_A}\setminus U)$ and extends the flow and the foliations $\F^s_A, \overline{\F^u}_A$ to this collar, in such a way that the flow enters or exits (in the case of $\psi_A$, enters) the collars at a right angle. 

One obtains in that way two $3-$manifolds equipped with semiflows with orbits perpendicular to their boundaries, entering the manifold in one case, and exiting in the other case. These manifolds can be glued together along their boundaries by a 90 degree rotation around the vertical axis, obtaining a closed $3$-manifold $M_A$ together with a flow $\varphi_A$ on $M_A$, which has a hyperbolic attractor and repeller. 

The foliation $\F^u_A$ induces a foliation $f^u$ of the transverse torus $T$, which also consists of a pair of Reeb components. The 90 degree rotation used as the gluing map then ensures that $f^s$ and $f^u$ are transverse Reeb foliations of the torus (see Figure \ref{fig:FW_torus_foliation}).

\begin{figure}[h]
  \centering
  \includegraphics[width=0.95\linewidth]{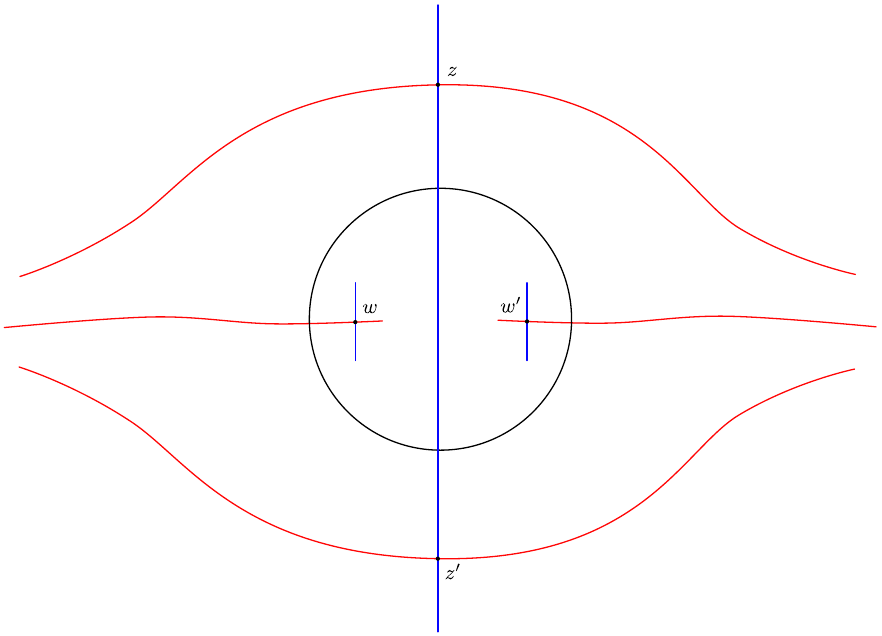}
  \caption{A schematic view of the flow $\varphi_A$ from above, with the transverse torus $T$ in black.}
  \label{fig:fw_flow_above}
\end{figure}

The transversality of the foliations $f^s$ and $f^u$ allows one to apply an argument of Mañé (originally shown in \cite{mane1977quasi} for the case of diffeomorphisms, see Section 6.2 of \cite{fisher2019hyperbolic} for a proof in the case for flows) which gives necessary and sufficient conditions for a flow to be an Anosov flow. Therefore, we get that $\varphi_A$ is an Anosov flow, such that leaves of the weak stable/unstable foliations restricted to the attractor and repeller coincide with the original leaves of $\F^s_A, \overline{\F^u_A}$ coming from $\psi_A$ and $\overline{\psi_A}$.

Note that $N_A$ and $\overline{N_A}$ are the two JSJ pieces of $M_A$. In future sections, when considering Franks-Williams flows, we will refer to these two pieces as $M_1$ and $M_2$, intersecting on the embedded torus $T$. Orbits of the flow cross the torus $T$ transversely, from $M_2$ and into $M_1$. The piece $M_1$ contains the attractor of the flow, $\Lambda_1$, and $M_2$ the repeller, $\Lambda_2$. These are the only basic pieces of the flow.

We will denote by $w$ and $w'$ the periodic points in $M_2$ whose unstable leaves $\F^u_A(w),\,\F^u_A(w')$ intersect the torus $T$ in closed leaves of the foliation $f^u$ (see Figure \ref{fig:fw_flow_above}).

\subsection{Intersections of leaves in the universal cover}\label{subsection:intersections_leaves_universal_cover_DA}

In this section, we study the way that stable and unstable leaves of the lift $\widetilde{g_A}$ to the universal cover of $T^2$ intersect. We define the notion of crossing (integer) points. In Lemma \ref{lemma:DA_crossing_points_iff}, we show that this property is equivalent to a property stated in terms of parallelograms bounded by leaves of the straight line foliations $\G^s, \G^u$. Then, in Proposition \ref{proposition:intersections_of_leaves_order_preserving_bijection} we show that there is an order preserving bijection from the intersections of certain stable and unstable leaves for $\widetilde{f_A}$ to the intersections of corresponding stable and unstable leaves for $\widetilde{g_A}$. These two results will be important in later sections.

We will denote by $\alpha$ the slope of the lines in $\G^s$, and by $\beta$ the slope of the lines in $\G^u$.

Let $\widetilde{g_A} : \R^2 \to \R^2$ of $g_A$ be the lift of $g_A$ to the universal cover $\R^2$ of $T^2$. Let $\wLambda \subset \R^2$ be the preimage of $\Lambda\subset T^2$ under this cover.

For a stable (or unstable) leaf $\F^s(x)$ (resp. $\F^u(z)$) of a point $x\in \Lambda$, and a point $\widetilde{x}\in \pi^{-1}(x) \subset \wLambda$, we denote by $\wF^s(\widetilde{x})$ (resp. $\wF^u(\widetilde{x})$) the lift of the leaf $\F^s(x)$ (resp. $\F^u(x)$) to $\R^2$ that contains $\widetilde{x}$. 

Recall that $z, z' \in T^2$ are the hyperbolic points of saddle type for $g_A$. Since $z, z'$ are at distance less than $\frac{1}{8} < \frac{1}{2}$ from $(0,0) \in T^2$, we are able to define the following:

\begin{definition}
    For $(m,n)\in \Z^2$, let $\wz_{(m,n)}, \wz'_{(m,n)}$ be the lifts of $z$ and $z'$ which are closest to $(m,n)\in \R^2$.

    For the case $(m,n) = (0,0)$, we write $\wz_{(0,0)} = \wz$, $\wz'_{(0,0)} = \wz'$.
\end{definition}

\begin{figure}[h]
  \centering
  \includegraphics[width=\textwidth]{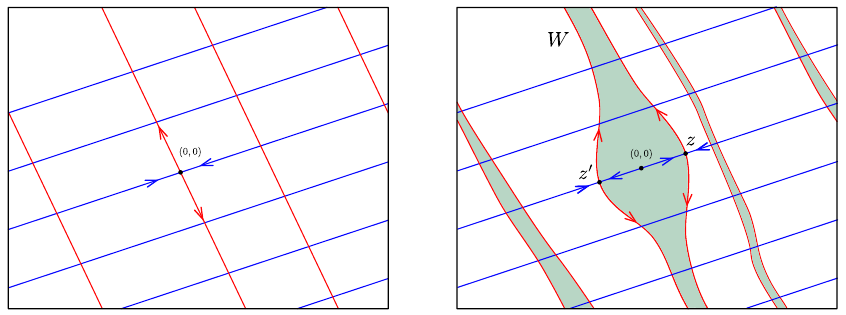}
  \caption{Stable and unstable manifolds for $f_A$ (left) and $g_A$ (right). Part of $W\subset T^2$ in green.}
  \label{fig:DA_torus}
\end{figure}

\begin{remark}
 The three points $\wz_{(m,n)}, (m,n)$ and $\wz'_{(m,n)}$ lie in a straight line which corresponds to a stable leaf of the original linear map $f_A$, and is also the union $\{ (m,n)\}\bigcup \wF^s(\wz_{(m,n)})\bigcup \wF^s(\wz'_{(m,n)})$.
\end{remark}

\begin{definition}
    For $(m,n) \in \Z^2$, define $\widetilde{W}_{(m,n)}$ to be the lift of $W\subset T^2$ to the universal cover $\R^2$, such that $\widetilde{W}_{(m,n)}$ contains $(m,n)$.
\end{definition}

\begin{remark}
    Proposition \ref{proposition:no_nullhomotopic_DA} tells us that the definition above makes sense.
\end{remark}

\begin{proposition}
For all $(m,n)\in \Z^2$, $\widetilde{W}_{(m,n)} \subset \R^2$ is an open set bounded by $\wF^u\left(\wz_{(m,n)}\right)$ and $\wF^u(\wz'_{(m,n)})$ (see Figure \ref{fig:da_crossings}).     
\end{proposition}

\begin{proof}

It is enough to show this for $\widetilde{W}_{(0,0)}$.
    The semiconjugacy $h$ maps $z$, $z'$ and $(0,0)$ to $(0,0)$, and therefore maps their unstable manifolds $\F^u(z), \F^u(z'), W$ to the unstable manifold $\G^u(0,0)$. Therefore, the lift $\widetilde{h}:\R^2 \to \R^2$ satisfying $\widetilde{h}(0,0)= (0,0)$ maps $\wF^u(\wz), \widetilde{W}_{(0,0)}, \wF^u(\wz')$ to $\wG^u(0,0)$.

    We know that $\widetilde{W}_{(0,0)}$ is an open subset of $\R^2$. Since $\widetilde{h}$ is continuous,\linebreak $(\widetilde{h})^{-1}\left((\wG^u)(0,0)\right) = \wF^u(\wz)\cup \widetilde{W}_{(0,0)} \cup \wF^u(\wz')$ is a closed subset of $\R^2$, and it contains $\widetilde{W}_{(0,0)}$ as an open subset. Thus, the unstable manifolds $\wF^u(\wz)$ and $\wF^u(\wz')$ bound $\widetilde{W}_{(0,0)}$.
\end{proof}

Below, we will prove some facts relating to the way that stable and unstable leaves of points in $\widetilde{\Lambda}$ intersect the stable and unstable leaves of $\wz$. The same arguments apply to intersections with the stable and unstable leaves of $\wz'$. These facts then also hold for all other lifts $\wz_{(m,n)}$ and $\wz'_{(m,n)}$ of $z$ and $z'$. This can be easily seen by applying the appropriate deck transformations. 

\begin{lemma}\label{lemma:DA_all_three_intersect_iff}
     Let $\wz'_{(m,n)}$ be a lift of $z'$ such that $\wz'_{(m,n)} \neq \wz'$. Then, \linebreak$\wF^u(\wz'_{(m,n)}) \cap \wF^s(\widetilde{z}) \neq \emptyset $ if and only if $\wF^u(\widetilde{z}_{(m,n)}) \cap \wF^s(\widetilde{z}) \neq \emptyset $.

\end{lemma}
\begin{proof}

The semiconjugacy $h:T^2 \to T^2$ maps $\F^u(z)$ and $\F^u(z')$ to $\G^u(0,0)$. Since $h$ is isotopic to the identity map of $T^2$, there is a proper isotopy between the lifts $\wF^u(\wz)$, $\wF^u(\wz')$ and $\wG^u(m,n)$, so either they all intersect the straight line $\wG^s(0,0) = \wF^s \left( \wz \right) \cup \{ 0,0\}\cup \wF^s(\wz')$ or neither of them does. 

Consider the case where $\wF^u(\wz)$, $\wF^u(\wz')$ intersect $\wG^s(0,0)$. In this case, they must all intersect the same connected component of $\wG^s(0,0)\setminus(\widetilde{W}_{(0,0)})$, since otherwise the set $\widetilde{W}_{(m,n)}$ bounded by $\wF^u(\wz)$, $\wF^u(\wz')$ would contain $(0,0)$, which is not possible. 

Therefore, since one of the connected components of $\wG^s(0,0)$ is completely contained in $\wF^s \left( \wz \right)$ and the other connected component does not intersect $\wF^s \left( \wz \right)$, we conclude that either both of $\wF^u(\wz)$, $\wF^u(\wz')$ intersect $\wF^s \left( \wz \right)$ or neither of them does.

\end{proof}

Recall from Proposition \ref{proposition:DA_properties} above that exactly one connected component $r_1$ of $\F^s(z)\setminus\{z\}$ is contained in $W$. Let $r_2$ be the other connected component of $\F^s(z)$, and let $\widetilde{r}_2$ be its lift containing $\wz$. Note that the unstable leaf of a point in $\wLambda$ intersects $\wF^s \left( \wz \right)$ if and only if it intersects it in a point of $\wr_2$.

\begin{definition}

Let $(m,n) \neq (0,0)$, and let $\wz_{(m,n)}, \wz'_{(m,n)} $ be the lifts of $z, z'$ closest to $(m,n)$. By the lemma above, we know that $\widetilde{r}_2$ intersects $\wF^u(\wz_{(m,n)})$ if and only if it intersects $\wF^u(\wz'_{(m,n)})$.

 If $\wF^s \left( \wz \right)$ does not intersect $\wF^s(\wz_{(m,n)})$, let $J(\widetilde{r}_2, (m,n)) = \emptyset$.
 
    Otherwise, define $J(\widetilde{r}_2, (m,n) ) \subset \widetilde{r}_2$ to be the closed interval in $\widetilde{r}_2$ bounded by $\wF^u(\widetilde{z}_{(m,n)})\cap \widetilde{r}_2$ and $\wF^u(\widetilde{z}'_{(m,n)}) \cap \widetilde{r}_2$.

\end{definition}

\begin{remark}
    An equivalent way of defining $J(\widetilde{r}_2, (m,n))$ is as \linebreak $J(\widetilde{r}_2, (m,n)) = \overline{\widetilde{r}_2 \cap \widetilde{W}_{(m,n)}} $, where $\widetilde{W}_{(m,n)}$ is the connected component of $\pi^{-1}(W)$ which contains $(m,n)$.
\end{remark}

The following proposition follows immediately from the Remark above.
\begin{proposition}
    Let $\widetilde{z}$ be a lift of $z$, and let $\widetilde{r}_2$ be the lift of $r_2$ based at $\widetilde{z}$. Then:

    \begin{itemize}
        \item For $(m,n) \neq (m',n')$, we have $J(\wr_2, (m,n)) \cap J(\wr_2, (m',n')) = \emptyset$.

        \item 
 The subset
        $$ \bigcup_{(m,n)\in \Z^2}  J(\widetilde{r}_2, (m,n)) \subset \widetilde{r}_2 $$
        is dense in $\widetilde{r}_2$.
        
    \end{itemize}   
 \end{proposition}

We will be particularly interested in some of the intervals $J(\widetilde{r}_2, (m,n))$:
\begin{figure}[h]
  \centering
  \includegraphics[width=.95\textwidth]{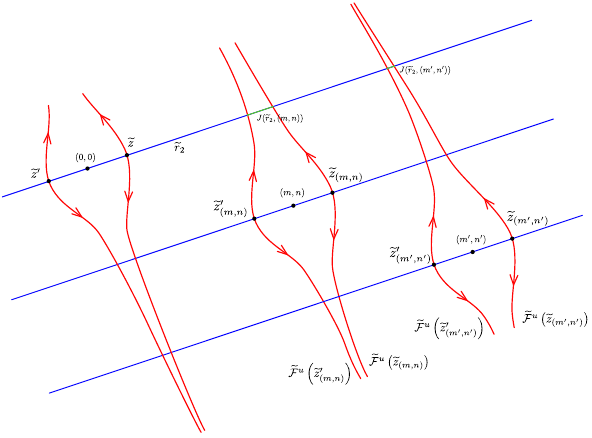}
  \caption{Lift of the map $g_A$. Here $(m,n)$ is a crossing point for $\wz$, and $(m',n')$ is not.}
  \label{fig:da_crossings}
\end{figure}
 \begin{definition}\label{definition:FW_crossing_point}

Let $(m,n)\in \Z^2$ be a point such that $J(\wr_2, (m,n)) \neq \emptyset$, and let $\wz'_{(m,n)}$ be the lift of $z'$ closest to $(m,n)$. 

We say that $(m,n)$ is a \emph{crossing point} for $\wz$ if the region of the plane bounded by the leaves $\wF^s \left( \wz \right), \wF^u(\wz), \wF^s\left(\wz'_{(m,n)}\right), \wF^u\left(\wz'_{(m,n)}\right)$ does not contain any integer points. Otherwise, we say that $(m,n)$ is a \emph{non-crossing point} for $\wz$.

We denote by $CP(\wz)$ the set of crossing points for $\wz$.
 \end{definition}

 \begin{definition}
     Choose an orientation for $\wF^s \left( \wz \right)$ and the plane. Then, a crossing point $(m,n) $ for $\wz$ is said to be a \emph{left} crossing point if $(m,n)$ is to the left of $\wF^s \left( \wz \right)$ with respect to this orientation. Otherwise, $(m,n)$ is said to be a \emph{right} crossing point.

     We denote by $CP_r(\wz)$ and $CP_l(\wz)$ the sets of right and left crossing points for $\wz$, respectively.
 \end{definition}

In the proof of the following Lemma, we will compare the intersection patterns of the leaves $\wF^s(\wz_{(m,n)}), \wF^u(\wz_{(m,n)})$ for $\wz \in \wLambda$ with the leaves of the foliations $\wG^s, \wG^u$ by straight lines parallel to the eigenspaces of the matrix $A$.

Recall that we have $\wG^s(\widetilde{x}) = \wF^s(\widetilde{x})$ for all $\widetilde{x} \in \wLambda$. 

 \begin{lemma}\label{lemma:DA_crossing_points_iff}
Let $(m,n) \in \Z^2$, and let $\wz'_{(m,n)}$ be the lift of $z'$ closest to $(m,n)$. Then, $(m,n)$ is a left (right) crossing point for $\wz$ if and only if:
\begin{enumerate}
    \item The point $(m,n)$ is to the left (resp. right) of $\wF^s \left( \wz \right)$,

    \item The line $\wG^u(m,n)$ intersects the component of $\wG^s(0,0)\setminus \{ (0,0)\}$ which contains $\wz$,

    \item There are no integer points in the interior of the parallelogram with corners $(0,0)$ and $(m,n)$ which is bounded by straight lines in $\wG^s, \wG^u$.
    
\end{enumerate}

\end{lemma}
\begin{proof}
This follows from the fact that the lift $\widetilde{h}:\R^2 \to \R^2$ of the semiconjugacy between $g_A$ and $f_A$ preserves orientation, maps integer points to integer points, and satisfies:
\begin{itemize}
    \item $\widetilde{h}(\wF^s \left( \wz \right))$ is the connected component of $\wG^s(0,0) \setminus \{ (0,0)\}$ which contains $\wz$.
    \item $\widetilde{h}(\wF^s(\wz'_{(m,n)})) $  is the connected component of $ \wG^s(m,n)\setminus \{ (m,n)\}$ which contains $\wz'_{(m,n)}$.
    \item $\widetilde{h}(\wF^u(\wz)) = \wG^u(0,0) $.
    \item $\widetilde{h}(\wF^u(\wz')) = \wG^u(m,n) $.
\end{itemize}

This implies that the region of the plane bounded by the leaves \linebreak $\wF^s \left( \wz \right), \wF^u(\wz), \wF^(\wz'_{(m,n)}), \wF^u \left( \wz'_{(m,n)} \right)$ contains integer points if and only if the parallelogram with corners $(0,0)$ and $(m,n)$ bounded by straight lines in $\wG^s, \wG^u$ contains integer points.

\end{proof}

We now define a parametrization of $\wG^s(0,0)$ (and therefore of $\wF^s \left( \wz \right) \subset \wG^s(0,0)$) in order to compare intersection points of these leaves and unstable leaves.

Let the parametrization $c:[0, +\infty) \to \R^2$ be given by $c(t) = t(1,\alpha)$. This induces orientations and therefore an order on $\wG^s(0,0)$ and on $\wF^s \left( \wz \right)$.

\begin{definition}\label{definition:order_on_crossing_point_intersections}
Given points $\widetilde{x}, \widetilde{x}' \in \wG^s(0,0)$ (or in $\wF^s \left( \wz \right)$), we say $\widetilde{x} < \widetilde{x}'$ if $c^{-1}(\widetilde{x}) <  c^{-1}(\widetilde{x}')$.

Then, given $(m,n), (m',n')\in \Z^2$ such that $(m,n), (m',n')\in CP(\wz)$, we let $ \wG^u(m,n) \cap \wG^s(0,0)  = \{x\}, \, \wG^u(m',n') \cap \wG^u(0,0) = \{x'\} $ and we say $$\wG^u(m,n) \cap \wG^s(0,0) < \wG^u(m',n') \cap \wG^u(0,0) \iff  x < x' $$
\end{definition}

\begin{proposition}\label{proposition:intersections_of_leaves_order_preserving_bijection}
    The function 
\begin{align*}
     & S: \left\{ \wG^u(m,n) \cap \wG^s(0,0) : (m,n) \in CP(\wz) \right\} \\ & \,\,\,\,\,\,\,\,\,\,\,\,\,\,\,\,\,\,\,\,\,\,\,\,\,\,\,\,\,\,\, \to  \left\{ \wF^u \left( \wz'_{(m,n)} \right) \cap \wF^s \left( \wz \right) : (m,n) \in CP(\wz)  \right\}
\end{align*}
    defined by $S\left(\wG^u(m,n) \cap \wG^s(0,0) \right) = \wF^u \left( \wz'_{(m,n)} \right) \cap \wF^s \left( \wz \right) $
    is an order preserving bijection.
\end{proposition}

\begin{proof}
    We know $S$ is a bijection by Lemma \ref{lemma:DA_crossing_points_iff}. Let $\widetilde{h}:\R^2 \to \R^2$ be the lift of the semiconjugacy satisfying $\widetilde{h}(\wz) = (0,0)$ The fact that the function $S$ is order preserving follows from the fact that $\widetilde{h}|_{\wF^s \left( \wz \right)}: \wF^s \left( \wz \right) \to \wG^s(0,0) $ is a (not strictly) monotone map.
\end{proof} 

\subsection{A cover of $M_1$}\label{subsection:cover_of_M1}

In the previous section, we studied intersections of stable and unstable leaves for the lift $\widetilde{g_A}$ of the map $g_A$ to the universal cover. The stable and unstable leaves of the suspension flow of $g_A$, defined on the mapping torus $N_A$ of $g_A$ are simply the suspension of the stable and unstable leaves of the diffeomorphism $g_A$. Therefore, the results in the previous section have direct analogues for the lift to the universal cover $\widetilde{N_A}$ of the suspension flow of $g_A$.

Our goal in this chapter is to understand the orbit space of the Franks-Williams flow $\varphi_A$ defined on $M_A$, and thus it is necessary to understand intersections of leaves of the lifted flow $\varphi_A$ to the universal cover $\widetilde{M_A}$. The manifold $N_A$ is not a submanifold of $M_A$: recall that in our construction of $M_A$, we remove a solid torus $U$ from $N_A$, and along the resulting torus boundary we glue a new copy of $N_A\setminus U$. Then, we can't directly apply the results mentioned in the previous paragraph. 

However, since the submanifold $M_1 = N_A \setminus U \subset M_A $ is both a submanifold of $M_A$ and of $N_A$, we can still take advantage of these results. Contained in the universal cover $\widetilde{N_A}$, we have a cover $\widehat{M_1}$ of $M_1$, obtained by removing from $\widetilde{N_A}$ all the lifts of the solid torus $U$. On the other hand, in the universal cover $\widetilde{M_A}$ of $M_A$ there are submanifolds which are copies of the universal cover $\widetilde{M_1}$ of $M_1$. These copies then must cover $\widehat{M_1}$. Through the cover $\widetilde{M_1} \to \widehat{M_1}$, we are able to use results regarding intersections of stable and unstable leaves of $\widetilde{g_A}$ or equivalently, of its suspension defined in $\widetilde{N_A} \supset \widehat{M_1}$, as long as we work with leaves which do not intersect the lifts of the removed solid torus $U$.

This is the motivation for studying the cover $\widehat{M_1}$ of $M_1$ and translating the results from the previous section into results concerning intersection of stable and unstable leaves of the flow $\widehat{\varphi}$ on $\widehat{M_1}$. This is the goal of this section.

As said above, the manifold $M_1$ is obtained by removing an open solid torus $U$ from the mapping torus $N_A$ of the DA map $g_A:T^2 \to T^2$, resulting in the manifold with torus boundary $M_1$. Then, we obtain a cover of $N_A$ by considering the universal cover $\widetilde{N_A} \cong \R^2 \times \R$ of $N_A$, and removing from $\widetilde{N}_A$ all the lifts of $U$. The lifts of $U$ to $\widetilde{N_A}$ are infinite solid cylinders intersecting each horizontal plane $\R^2 \times \{ t\}$ in an open disk.

In this way, we get a manifold $\widehat{M}_1 \subset \R^2 \times \R$ with infinitely many cylindrical boundary components, which covers $M_1$. A horizontal cross-section of $\widehat{M}_1$ is a plane minus an infinite family of open disks, one for each integer point on the plane. 

The flow $\varphi_A$ restricted to $M_1$ is simply the suspension flow of the map $g_A: T^2 \to T^2$. This suspension flow lifts to $\R^2 \times \R$, and restricting the lifted flow on $\R^2 \times \R$ (obtaining then a semiflow, since orbits cross the boundary cylinders) to $\widehat{M}_1 \subset \R^2 \times \R$ we get a semiflow $\widehat{\varphi_A}$ on $\widehat{M}_1$ that projects to the restriction of the flow $\varphi_A$ on $M_1$.

Recall that for a point $x$ in the hyperbolic attractor $\Lambda \subset T^2$ of the map $g_A$ there exist one-dimensional stable and unstable manifolds $\F^s(x), \F^u(x)$ respectively, which are immersed in $T^2$. These lift to stable and unstable manifolds $\wF^s(\widetilde{x}), \wF^u(\widetilde{x}) \subset \widetilde{N_A} \subset \R^2 \times \R$ for $\widetilde{g_A}$ of points $\tilde{x} \in \widetilde{\Lambda} \subset \R^2$, where $\widetilde{\Lambda}$ is the lift of $\Lambda$ to $\R^2$.

Let $q:\R^2\times \R \to \R^2$ be the projection to the first factor. For any point $\widetilde{x} \in \R^2$, let $\widehat{x} = (\widetilde{x},0)\in \R^2 \times \R$. Then, for $\widetilde{x}\in \wLambda$, let $\hF^s(\widehat{x}) = q^{-1}(\wF^s(\widetilde{x}))$, $\hF^u(\widehat{x}) = q^{-1}(\wF^u(\widetilde{x}))$. These are topological planes in $\widetilde{N_A} \subset \R^2\times \R$, saturated by orbits of the flow $\widehat{\varphi_A}$.

By how we have defined the manifolds $\hF^s(\hx), \hF^u(\hx)$ for $\hx \in \widehat{\Lambda}$, it's clear that given points $\hx, \hz  \in \widehat{\Lambda} $ we have that $\hF^s(\hx) \cap \hF^u(\hz) \neq \emptyset$ if and only if $\wF^s(\widetilde{x}) \cap  \wF^u(\wz) \neq \emptyset $, and moreover, that everything we have shown about the patterns of intersections of the stable and unstable leaves $\wF^s \left( \wz \right), \wF^u(\wz)$ of points in $\wLambda$ holds for the leaves $\hF^s(\hx), \hF^u(\hx)$ of points in $\widehat{\Lambda}$. More explicitly, we have a version of Lemma \ref{lemma:DA_crossing_points_iff} which we state below.

Recall that $\wz, \wz' \in \wLambda$ are defined to be the lifts to $\R^2$ which are closest to $(0,0)$ of the hyperbolic fixed points $z, z' \in T^2$ of the map $g_A$. Also, $\wr_2$ is defined to be the half-leaf based at $\wz$ and contained in $\wF^s \left( \wz \right)$ that intersects infinitely many unstable leaves of points in $\wLambda$.

Let $\hz = (\wz, 0)$ and $\hz' = (\wz', 0)$.

As before, everything below applies to the other lifts of $z, z'$, which we can see by applying the necessary deck transformations.

 \begin{definition}

Let $(m,n)\in \Z^2$, and let $\hz_{(m,n)} = (\wz_{(m,n)}, 0), \hz'_{(m,n)} = (\wz'_{(m,n)}, 0) \in \widehat{M}_1$ where $\wz_{(m,n)}, \wz'_{(m,n)}\in \R^2$ are the lifts of $z, z'$ closest to $(m,n)$. Suppose that $\hF^s \left( \hz \right) \cap \hF^u \left( \hz'_{(m,n)} \right) \neq \emptyset$.

We say that $\hF^u \left( \hz'_{(m,n)} \right)$ is a \emph{crossing leaf} for $\hz$ if the region of $\widehat{M}_1$ bounded by the leaves $\hF^s \left( \hz \right), \hF^u(\hz), \hF^s \left( \hz'_{(m,n)} \right), \hF^u \left( \hz'_{(m,n)} \right)$ does not contain any points in $ \Z^2 \times  \R$. Otherwise, we say that  $\hF^u \left( \hz'_{(m,n)} \right)$ is a \emph{non-crossing leaf} for $\hz$.

We denote by $CL(\hz)$ the set of crossing leaves for $\hz$.
 \end{definition}

 \begin{definition}
    Give $\R^2\times \R$ the standard product orientation. Then, a crossing leaf $\hF^u \left( \hz'_{(m,n)} \right)$ for $\hz$ is said to be a \emph{left} crossing leaf for $\hz$ if $((m,n),0)$ is to the left of $\hF^s \left( \hz \right)$ with respect to the chosen orientation. Otherwise, $\hF^u \left( \hz'_{(m,n)} \right)$ is said to be a \emph{right} crossing leaf for $\hz$.

    We denote by $CL_r(\hz)$ and $CL_l(\hz)$ the sets of left and right crossing leaves for $\hz$, respectively.
 \end{definition}

Recall that $\wG^s, \wG^u$ are the foliations of $\R^2$ by straight lines parallel to the eigenspaces of $A$.

The following lemma and proposition are the counterpart on $\widehat{M_1}$ of Lemma \ref{lemma:DA_crossing_points_iff} and Proposition \ref{proposition:intersections_of_leaves_order_preserving_bijection}, and they follow straightforwardly from them. 

For Proposition \ref{proposition:bijection_M1_integers}, we order intersections of unstable leaves $\hF^u \left( \hz'_{(m,n)} \right)$ with the leaf $\hF^s \left( \hz \right)$ as follows: we say $$\hF^u \left( \hz'_{(m_1,n_1)} \right)\cap \hF^s \left( \hz \right) <  \hF^u \left( \hz'_{(m_2,n_2)} \right)\cap \hF^s \left( \hz \right)$$ if $\hz$ and $\hF^u(\hz'_{(m_2,n_2)}) $ lie in different connected components of $\widehat{M}_1 \setminus \hF^u \left( \hz'_{(m_1,n_1)} \right)$. This is consistent to the way we ordered intersections of $\wF^s \left( \wz \right)$ with the leaves $\wF^u\left(\wz'_{(m,n)
}\right)$.

 \begin{lemma}
Let $(m,n) \in \Z^2$, let $\wz'_{(m,n)}$ be the lift of $z'$ closest to $(m,n)$ and let $\hz'_{(m,n)} = (\wz'_{(m,n)}, 0)$. Then, $\hF^u \left( \hz'_{(m,n)} \right)$ is a left (right) crossing leaf for $\hz$ if and only if $(m,n)$ is to the left (resp. right) of the line through the origin $\G^s(0,0)\subset \R^2$ and there are no integer points in the interior of the parallelogram with corners $(0,0)$ and $(m,n)$ which is bounded by straight lines in $\G^s, \G^u$.
\end{lemma}

\begin{proposition}\label{proposition:bijection_M1_integers}

The map
\begin{align*}
     & \widehat{S}:\left\{ \hF^u \left( \hz'_{(m,n)} \right) \cap \hF^s \left( \hz \right) : \hF^u \left( \hz'_{(m,n)} \right) \in CL(\hz) \right\} \\ & \,\,\,\,\,\,\,\,\,\,\,\,\,\,\,\,\,\,\,\,\,\,\,\,\,\,\,\,\,\,\,\to \left\{ \wG^u(m,n) \cap \wG^s(0,0) : (m,n) \in CP(\wz) \right\} 
\end{align*}

$$    $$
defined by $\widehat{S}\left(\hF^u \left( \hz'_{(m,n)} \right) \cap \hF^s \left( \hz \right)\right) =   \wG^u(m,n) \cap \wG^s(0,0)$  is an order preserving bijection.

Moreover, $\widehat{S}$ maps intersections $ \hF^u \left( \hz'_{(m,n)} \right) \cap \hF^s \left( \hz \right)$ where $\hF^u \left( \hz'_{(m,n)} \right)$ is a left (right) crossing leaf to intersections $ \wG^u(m,n) \cap \wG^s(0,0)$ where $(m,n)\in CP_l(\wz)$ (resp. $(m,n)\in CP_r(\wz)$).
    
\end{proposition}

\section{The bifoliated plane $\left(P_A, \F^+_A, \F^-_A\right)$ associated to $\varphi_A$}

We will denote by $\F_A^+$ the projection of the unstable foliation $\wF^u$ to the orbit space $P_A$, and by $\F^-_A$ the projection of $\wF^s$.

\subsection{Infinite perfect fits in $\left(P_A, \F^+_A, \F^-_A\right)$}\label{subsection:infinite_perfect_fits}

In this section, we first define infinite perfect fits in a bifoliated plane, which are a particular type of chain of lozenges. Then, we show that an open and dense set of points in $P_A$ belong to some such chain of lozenges. For each infinite perfect fit $I \subset P_A$, we identify the region in the universal cover $\widetilde{M_A}$ that projects to $I$.

\begin{definition}
   
Let $(P, \F^+, \F^-)$ be a bifoliated plane. An infinite perfect fit $I$ in $P$ is an infinite chain of adjacent lozenges, such that:
\begin{enumerate}
    \item Suppose $L_1, L_2, L_3 \subset I$ are lozenges in the chain such that $L_1, L_2$ are adjacent and $L_2, L_3$ are adjacent. Then, if $L_1$ and $L_2$ are adjacent on a stable side, we must have that $L_2$ and $L_3$ are adjacent on an unstable side, and viceversa.
    \item Each corner point of a lozenge in $I$ is a corner of exactly two lozenges in $I$.
\end{enumerate}
\end{definition}

\begin{remark}
    The definition implies that given an infinite perfect fit $I\subset P$, there exists a point $x\in \partial P$ such that for each lozenge $L\subset I$, two half-leaves bounding $L$ have $x$ as an endpoint in $\partial P$.
\end{remark}

\begin{definition}

We say that a leaf $l\in \F^+$ (or $\F^-$) makes an infinite perfect fit if there exists an infinite perfect fit $I \subset P$ and two lozenges $L, L' \subset I$ such that $L, L'$ are both bounded by a half-leaf $s\subset l$.

In that case, we also say that the half-leaf $s$ makes an infinite perfect fit.
\end{definition}

Now, we show that an open and dense subset of $P_A$ is covered by interiors of infinite perfect fits.

 We will denote by $\pi_\mathcal{O}: \widetilde{M_A} \to P_A$ the projection to the orbit space of the lifted flow $\widetilde{\varphi_A}:\R\times \widetilde{M_A} \to \widetilde{M_A}$. The key to understanding the stable and unstable foliations in the orbit space of such a flow lies in understanding the torus transverse to the flow and the induced foliations on it. This is because an open dense subset of orbits intersects this torus transversely. The only orbits which do not pass through the torus are those contained in the attractor and the repeller, which are compact sets with empty interior, with each of them contained in one JSJ piece of the manifold $M_A$. 

Recall from the previous section that the JSJ piece $M_1$ contains two periodic points $z, z'$ such that their stable leaves $\F^s_A(z), \F^s_A(z')$ intersect $T$ in the only compact leaves of the induced foliation $f^s$ on $T$. We denote by $\alpha_1$ and $\alpha_2$ the orbits of $z$ and $z'$, respectively. Analogously, $M_2$ contains periodic points $w, w'$ whose unstable leaves intersect $T$ in the only compact leaves of $f^u$, and we denote their orbits by $\beta_1, \beta_2$. Note that $\alpha_1, \alpha_2 \subset \Lambda_1$ and $\beta_1, \beta_2 \subset \Lambda_2$, where $\Lambda_1 \subset M_1$ is the attractor and $\Lambda_2 \subset M_2$ is the repeller of $\varphi_A$.

Consider a single connected component $V$ of the lift $\pi^{-1}(M_1) \subset \widetilde{M_A}$ to the universal cover $\widetilde{M_A}$. This is a 3-dimensional manifold with boundary, whose boundary consists of an infinite disjoint union of topological planes, each of them a lift of the torus $T$. The lifted flow $\widetilde{\varphi_A}$ on $\widetilde{M_A}$ is transverse to each of these boundary components, with orbits going into $U$ in the future. 

\begin{proposition}
    The projection of each lift $\widetilde{T}$ to the orbit space is an infinite perfect fit. 
\end{proposition}

\begin{proof}
Applying Proposition \ref{proposition:birkhoff_projects_lozenges}, we can see that each lift $\widetilde{T} \subset \partial V$ of the transverse torus $T$ must project to an infinite chain of lozenges in the orbit space. In order to see this, it's enough to consider two Birkhoff annuli joining $z$ and $z'$, such that their union is a torus $T'$ which is isotopic to the transverse torus $T$. 

The type of the chain of adjacent lozenges is determined by the transverse foliations on the interior of the Birkhoff annuli induced by the stable and unstable foliations of the flow $\varphi_A$. By isotoping $T'$ into $T$, we can instead consider the transverse foliations $f^s$ and $f^u$ on the torus $T$.

These foliations on the torus are, as we have seen above, a pair of Reeb foliations, each with exactly two closed leaves which correspond to the stable (unstable) leaves of $\alpha_1, \alpha_2$ ($\beta_1, \beta_2$). Note that, as shown in Figures \ref{fig:fw_flow_above} and \ref{fig:FW_torus_foliation}, these closed leaves are arranged on the torus in such a way that between the closed leaves of each foliation, we find a closed leaf of the other foliation.

Therefore, given three lozenges $L_1, L_2, L_3$ such that $L_1, L_2$ are adjacent and $L_2, L_3$ are adjacent, we must have that the leaves that separate the pairs of lozenges $L_1, L_2$ and $L_2, L_3$ must belong to different foliations. 

Moreover, if we let $l_1$ be the leaf separating $L_1$ from $L_2$ and $l_2$ the leaf separating $L_2$ from $L_3$, then we can assume without loss of generality that $l_1$ is the projection of a lift of $\F^s(\alpha_1)$ and $l_2$ the projection of a lift of $\F^u(\beta_2)$. In particular, $l_2$ cannot contain the point $\walpha_1$, so $l_2$ cannot be a boundary leaf of $L_2$. 

Since the choice of adjacent lozenges $L_1, L_2, L_3$ in the chain was arbitrary, the two paragraphs above imply that the chain is an infinite perfect fit, as we wanted.
    
    \end{proof}



Let $\widetilde{T} \subset \partial V$ be a lift of $T$. Consider lifts $\widetilde{\F}^s(\wz), \widetilde{\F}^s(\wz'), \widetilde{\F}^u(\widetilde{w}), \wF^u(\widetilde{w}')$ of the stable leaves of $z,z'$ and the unstable leaves of $w,w'$. Choose these lifts so that they intersect $\widetilde{T}$ in consecutive closed leaves of the foliations $\widetilde{f}^s, \widetilde{f}^u$ of $\widetilde{T}$, and contain lifts $\widetilde{\alpha}_i$, $\widetilde{\beta}_i$ of the orbits $\alpha_i, \beta_i$ for $i = 1,2$.

\begin{remark}
    Any two choices of lifts satisfying the properties above are the same, up to applying a deck transformation on $\widetilde{M_A}$. Therefore, the results and arguments in the rest of this section will not depend on the specific choice made.
\end{remark}

The leaves $\widetilde{\F}^s(\wz), \wF^s(\wz')$ project to leaves $l_1, l_2$ of the foliation $\F^-$ in the orbit space $P_A$, while $\widetilde{\F}^u(\widetilde{w}), \wF^u(\widetilde{w}')$ project to leaves $s_1, s_2$ of $\F^+$. These leaves are sides of lozenges in an infinite perfect fit. The sides $l_3, l_4$ shown in the figure are projections of the stable leaves $\widetilde{\F}^s(\widetilde{w}),\widetilde{\F}^s(\widetilde{w}') $ containing $\widetilde{\beta}_1, \widetilde{\beta}_2$, and $s_3, s_4$ are projections of of the unstable leaves $\widetilde{\F}^u(\wz), \widetilde{\F}^u(\wz')$ which contain $\widetilde{\alpha}_1, \widetilde{\alpha}_2$. 

\begin{figure}[h!]
  \centering
  \includegraphics[width=.85\linewidth]{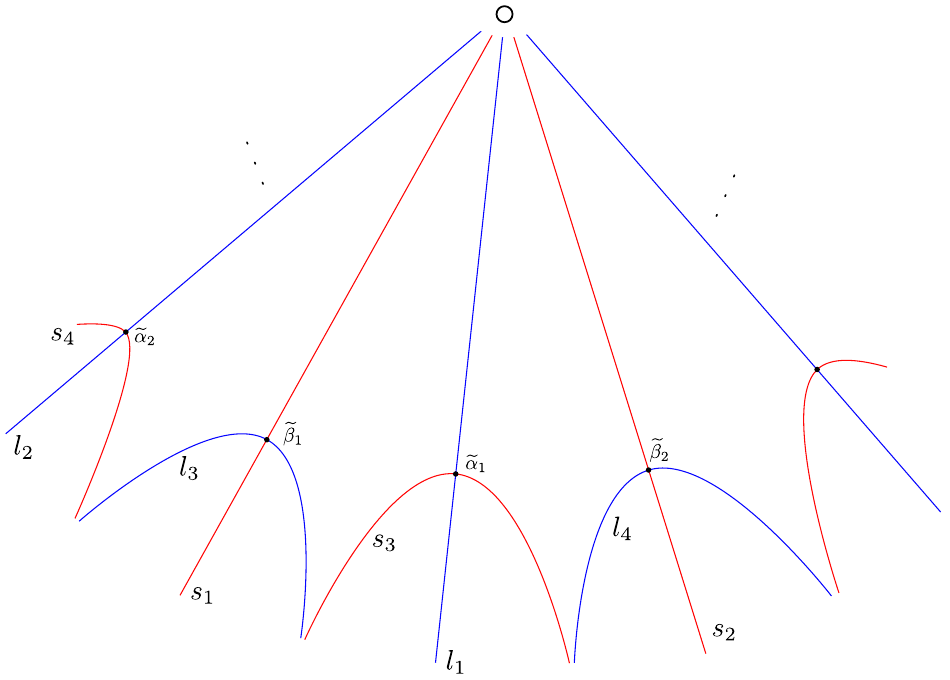}
  \caption{An infinite perfect fit}
  \label{fig:inf_perf_fit}
\end{figure}

\begin{proposition}

   If a leaf $l^- \in \F^-_A $ makes an infinite perfect fit, then $l^- = \pi_{\mathcal{O}}(\wF^s(\wz_1)) $, for some $\wz_1 \in \widetilde{M}$ which is a lift of $z$ or of $z'$. 

   If a leaf $l^+ \in \F^+_A$ makes an infinite perfect fit, then $l^+ = \pi_\mathcal{O}(\wF^u(\widetilde{w}_1))$, for some $\widetilde{w}_1 \in \widetilde{M}$ which is a lift of $w$ or $w'$.
    
\end{proposition}

\begin{proof}

The definition of an infinite perfect fit implies that if a leaf $l^- \in \F^-_A$ makes an infinite perfect fit, then it is a nonseparated leaf, i.e. a non-Hausdorff point in the leaf space of $\F^-_A$. 

From the construction of the flow, we can see that if $l^- = \pi_{\mathcal{O}}(\wF^s(\widetilde{\gamma}))$ for some orbit $\widetilde{\gamma} \subset \widetilde{M}$ and $\widetilde{\gamma}$ is not a lift of $\alpha_1$ or $\alpha_2$, then we can separate $l$ from any other leaf $l'\in \F^-_A$: it's enough to check it for the leaves in each piece $M_1$ and $M_2$, where the flow is the suspension of the DA map $g_A$. If the leaves $\wF^s(\widetilde{x}_1)\in \wF^s$ and $\wF^s(\widetilde{x}_2)\in \wF^s$ which are lifts of leaves intersecting $M_1$ do not contain a lift of $\wz$ or $\wz'$, then they both intersect some leaf $\wF^u(\widetilde{y})\in \wF^u$, and therefore they can be separated.

We conclude that the leaf $l^-$ cannot be a nonseparated leaf. 

The proof for $l^+ \in \F^+_A$ is analogous.
\end{proof}

The following is a consequence of the proposition above and the fact that each leaf of $\wF^u$ and $\wF^s$ can contain at most one lift of one of $z, z', w$ or $w'$.
\begin{corollary}\label{corollary:FW_only_one_side}
    Every leaf in $P_A$ that makes an infinite perfect fit, makes a perfect fit on only one side.
\end{corollary}

Now, consider the leaves $\wF^s \left( \wz \right), \wF^u(\wz)$, which intersect in the orbit $\walpha_1$. There are two connected components of $\wF^s \left( \wz \right)\setminus \walpha_1$. Let $E_1$ be the component that intersects the transverse torus, and let $E_2$ be the other component. Then, $E_1$ projects to a half-leaf $\wr_1$ in $P_A$ that bounds a lozenge in the infinite perfect fit obtained by projecting $\widetilde{T}$. Let $\wr_2 = \pi_\mathcal{O}(E_2)$ be the projection of the other half-leaf.

For the rest of this section, fix a choice of half-leaf $\wr_2 = \pi_\mathcal{O}(E_2)$ as described in the proof of Corollary \ref{corollary:FW_only_one_side}. As remarked above, all results we prove in this section will apply for any such choice.


\begin{proposition}
    Any two infinite perfect fits in $(P_A, \F^+_A, \F^-_A)$ are disjoint.
\end{proposition}

\begin{proof}
    This is a consequence of the fact that each infinite perfect fit consists of the set of orbits that intersect a given lift $\widetilde{T}$ of the torus $T$ transverse to the flow $\varphi_A$.
    
    Since an orbit of $\varphi_A$ that intersects $T$ does so only once, it is not possible for any orbit of $\widetilde{\varphi_A}$ to intersect more than one lift of $T$ in $\widetilde{M}$. Therefore, the projection of each obit to $P_A$ must be contained in at most one infinite perfect fit, showing that infinite perfect fits in $P_A$ are disjoint.
\end{proof}

Projecting the lifts of $T$ contained in $\partial V$ yields a collection of disjoint infinite perfect fits. In order to understand how these are arranged in the orbit space, we will now consider a lift $\widetilde{T}' \subset \partial V$ such that $\widetilde{T} \neq \widetilde{T}' $, and we will look at whether it intersects the ray $r_2$ (the same argument would apply for any analogous ray based at a different lift of $z_1$ or $z_2$), and what this intersection looks like.

Recall that $V$ is a connected component of $\pi^{-1}(M_1)\subset \widetilde{M}$. Equivalently, we could think of $V$ as a connected component of $\overline{\pi}^{-1}(\widehat{\pi}^{-1}(M_1))$, where $\overline{\pi}: \widetilde{M_1} \to \widehat{M}_1$ is the universal cover of $\widehat{M}_1$.

The following is a straightforward consequence of Lemma \ref{lemma:DA_all_three_intersect_iff}.
\begin{proposition}\label{proposition:plane_all_three_intersect}
    Let $s^+ \in \F^+$ be a leaf which makes a infinite perfect fit. Let $s_1^+, s_1^+$ be the other unstable sides of the lozenges $L_1, L_2$ which have a boundary half-leaf contained in $s^+$. Then, one of $s_1^+, s_2^+, s^+$ intersects $\wr_2$ if and only if all of them do. 
\end{proposition}

After choosing an orientation for the foliations $\F^+$ and $\F^-$, we can think of the leaves making infinite perfect fits and intersecting $\wr_2$ as being divided into those that make a perfect fit to the right of $\wr_2$ and those that make it to the left of $\wr_2$. 

\begin{proposition}
Each infinite perfect fit that intersects $\wr_2$ does so along the interior of exactly two adjacent lozenges, and the set of intersections of infinite perfect fits with $\wr_2$ is an open and dense subset of $\wr_2$.
\end{proposition}
\begin{proof}
First, we show that any leaf of $\F^-$ (or $\F^+$) that intersects the interior of an infinite perfect fit can intersect at most two lozenges $L_1, L_2 \subset I$: suppose without loss of generality that $l\in \F^-$ intersects the interior of an infinite perfect fit. 

Then, $l$ intersects the interior of some lozenge $L$ at some point $\widetilde{\gamma}$. Therefore, since the foliations in the interior of a lozenge are product foliations, we must have that $l$ intersects both unstable sides $s_1^+, s_2^+$ of $L$. One of the unstable sides, say $s_1^+$, is in the boundary of $I$. This implies that the connected component of $l\setminus s_1^+$ which contains $\widetilde{\gamma}$ must be contained in the complement of $I$, since a leaf that exits an infinite perfect fit cannot enter it again. The other connected component intersects $s_2^+$, which bounds lozenges $L$ and $L' $ contained in $I$. Then, by the same reasoning as before, this component must intersect the boundary leaf $s^+_3$ of $L'$ which is in the boundary of $I$. Therefore, it exits $L'$ and $I$ through this boundary leaf. We conclude that the leaf $l$ only intersects $I$ on the lozenges $L $ and $L'$.

The fact that it must intersect at least two lozenges follows from Proposition \ref{proposition:plane_all_three_intersect}.

\end{proof}

The proposition below follows from the fact that on both sides of the contracting eigenspace of the matrix $A$ we can find infinitely many integer points.
\begin{proposition}
    There are infinitely many leaves intersecting $\wr_2$ that make infinite perfect fits to the right, and infinitely many leaves intersecting $\wr_2$ that make infinite perfect fits to the left.
\end{proposition}


\subsection{Crossing and non-crossing infinite perfect fits}\label{subsection:crossing_and_non_crossing_IPFs}

Analogously to how we distinguished between crossing and non-crossing leaves earlier, we will distinguish between two types of infinite perfect fits intersecting $\wr_2$, also named ``crossing'' and ``non-crossing''. In fact, we will later see that crossing leaves are in one-to-one correspondence with crossing infinite perfect fits.

\begin{figure}[h!]
  \centering
  \includegraphics[width=.75\linewidth]{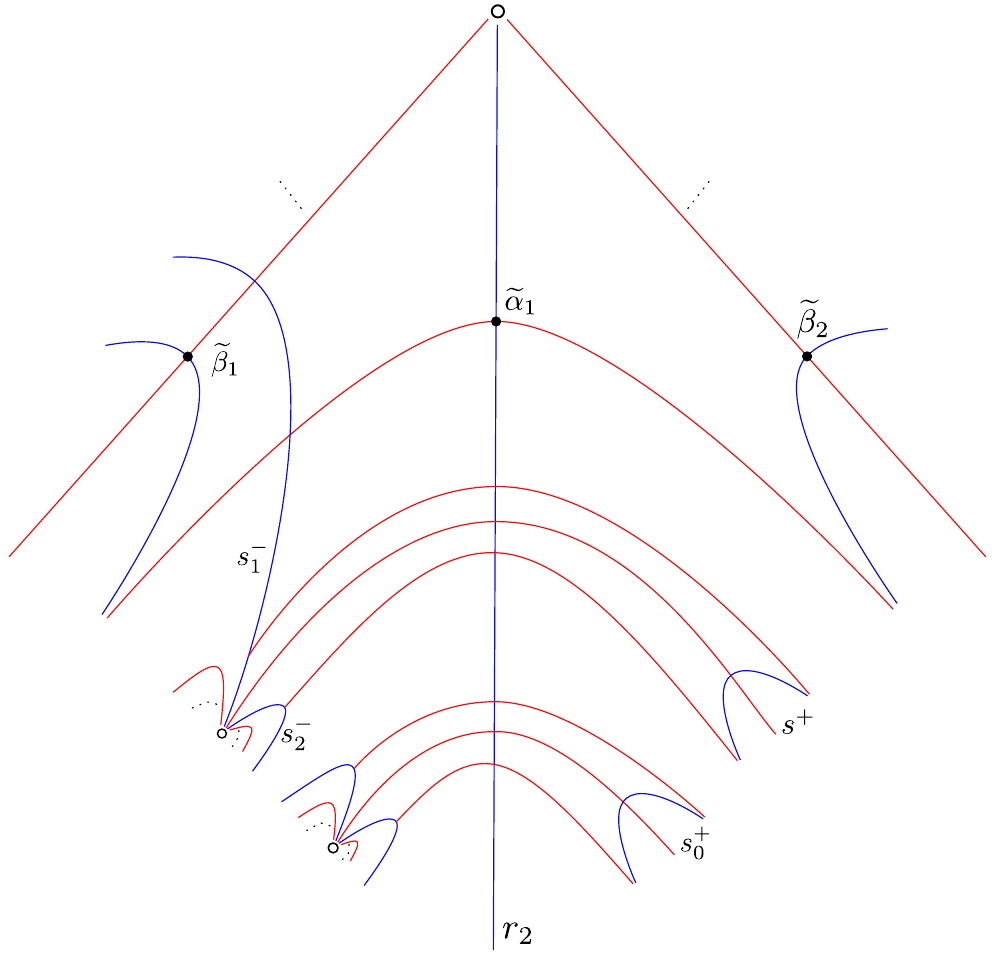}
  \caption{$s^+$ bounds a lozenge in a crossing infinite perfect fit for $\wr_2$, while $s_0^+$ does not.}
  \label{fig:crossing_non_crossing}
\end{figure}

\begin{definition}(Crossing infinite perfect fits)
    Let $s^+ \in \F^+$ be a leaf making an infinite perfect fit and intersecting $\wr_2$. Let $s_1^-, s_2^-$ be leaves making an infinite perfect fit $I$ with $s^+$, such that the pairs $s^+, s_1^-$ and $s^+, s_2^-$ bound lozenges $L_1$ and $L_2$, respectively.

    Let $L_1 \subset I$ be the lozenge closest to $\walpha_1$ along $\wr_2$, in the sense that $\wr_2\setminus L_1$ contains $\walpha_1$ in one connected component and $L_2$ in the other connected component.

    Then, if $s_1^-$ intersects the unstable leaf $\F^+(\walpha_1)$ of $\walpha_1$, we say that $I$ is a \emph{crossing} infinite perfect fit for $\wr_2$.

    Moreover, in the case $I$ is a crossing infinite perfect fit for $\wr_2$, let $c_1, c_2\subset \F^+(\widetilde{\alpha}_1)$ be the half-leaves based at $\widetilde{\alpha}_1$ to the left and right of $\wr_2$, respectively. If one of $s_1^-, s_2^-$ intersects $c_1$ ($c_2$), we say $I$ is a \emph{left} (resp. \emph{right}) crossing infinite perfect fit for $
    \wr_2$. 
    
\end{definition}

Identify $V$ with the universal cover of $\widehat{M}_1$ (in such a way that stable and unstable leaves of $\widetilde{\Lambda}$ are identified with the corresponding leaves of $\widehat{\Lambda}$), and let $\overline{\pi}: V \to M_1$ be the universal covering map. 

We use the transverse orientation on $\overline{\pi}(\wF^s(\walpha_1))$ induced by the transverse orientation of the foliation $\wF^s$.

\begin{proposition}\label{proposition:crossing_fits_crossing_leaves}

Let $s^+ \in \F^+$ be a leaf intersecting $\wr_2  = \pi_\mathcal{O}(E_2) \subset \pi_\mathcal{O}(\wF^s \left( \wz \right))$ which makes an infinite perfect fit $I$, where $s^+ = \pi_\mathcal{O}(\wF^u(\widetilde{w}_0))$ for some lift $\widetilde{w}_0\in \widetilde{M}$ of $w\in M$. 

Let $s_1^+ = \pi_\mathcal{O}(\wF^u(\wz_0)) $ be the other unstable side of the lozenge $L_1 \subset I$ bounded by $s^+$, where we choose $L_1$ so that the intersection point $s_1^+ \cap \wr_2$ separates $\wr_2$ into two connected components, one containing $\widetilde{\alpha_1}$ and one containing $s^+\cap \wr_2$.

Then, $I$ is a left (right) crossing infinite perfect fit for $\wr_2$ if and only if $\overline{\pi}(\wF^u(\wz_0))$ is a left (resp. right) crossing leaf for $\overline{\pi}(\wz)$ in $\widehat{M_1}$.
\end{proposition}

\begin{proof}

Suppose that $I$ is an infinite perfect fit intersecting $\wr_2\subset \F^+(\walpha) = \pi_\mathcal{O}(\wF^s \left( \wz \right))$. Then, we know that $s_1^+ \cap \wr_2 \neq \emptyset$, and therefore $\wF^u(\wz) \cap \wF^u(\wz_0) \neq \emptyset $. 


If we also know that $I$ is a crossing infinite perfect fit for $\wr_2$, we have by definition that $s_1^-$ intersects $\F^+(\walpha)$, where $s_1^- = \pi_\mathcal{O}(\wF^s(\wz_0))$.

We can then see that the leaves $\F^+(\walpha), \F^-(\walpha), s_1^+ = \F^+(\widetilde{\beta}_0), s_1^- = \F^-(\widetilde{\beta}_0)$ bound a region in $P_A$ which is homeomorphic to a disk. Then, the corresponding leaves $\wF^u(\wz), \wF^s \left( \wz \right), \wF^u(\wz_0), \wF^s(\wz_0)$ must bound a region  $U \subset V \subset \widetilde{M}$ which is homeomorphic to $\R^3$.

In particular, our identification of $V$ with $\widehat{M}_1 \subset \R^2\times \R$ maps \linebreak $\wF^u(\wz), \wF^s \left( \wz \right), \wF^u(\wz_0), \wF^s(\wz_0)$ to leaves $\hF^u(\hz), \hF^s \left( \hz \right), \hF^u(\hz_0), \hF^s(\hz_0)$ respectively, which bound a region $U\subset \widehat{M_1}$ homeomorphic to $\R^3$. Since $\Z^2 \times \R$ is in the complement of $\widehat{M_1}$ in $\R^2 \times \R$, we conclude that the region $U$ bounded by \linebreak $\hF^u(\hz), \hF^s \left( \hz \right), \hF^u(\hz_0), \hF^s(\hz_0)$ does not contain any points in $\Z^2 \times \R$. This shows that $\overline{\pi}(\wF^u(\wz_0)) = \hF^u(\hz_0) $ is a crossing leaf for $\overline{\pi}(\wz) = \hz$.


\end{proof}

\begin{definition} Let $$C(\wr_2) = \{\widetilde{x} \in I\cap \wr_2 : I \text{ is a crossing infinite perfect fit for } \wr_2 \},$$ and $C_r(\wr_2), C_l(\wr_2) \subset C(\wr_2)$ be the subsets consisting of intersections of right and left crossing infinite perfect fits with $\wr_2$. 

   We define an order in $C(\wr_2)$ as follows: $I \cap \wr_2 < I'\cap \wr_2$ if and only if $\wr_2 \setminus (I\cap \wr_2)$ contains $\walpha_1$ in one connected component, and $I' \cap \wr_2 $ in the other connected component.
\end{definition}

Recall that we denoted by $\wG^s$ and $\wG^u$ the lifts to $\R^2$ of the foliations by lines parallel to the contracting and expanding eigenspaces of $A$. 

For $(m,n)\in \Z^2$, let $R_{(m,n)} \subset \R^2$ be the parallelogram having $(0,0)$ and $(m,n)$ as vertices, and with sides parallel to the eigenspaces of $A$.

Let $r\subset \G^s(0,0)\subset \R^2$ be the ray based at $(0,0)$ that contains a lift of $z\in T^2$.

\begin{corollary}\label{proposition:equivalence_crossing_integer_points}

There is an order preserving bijective correspondence
$$  \widetilde{S}: C(\wr_2) \to  \left\{\wG^u(m,n) \cap \wG^s(0,0): (m,n)\in CP(\wz) \} \right\} $$
given by
$$ \widetilde{S}(I\cap \wr_2) = \widehat{S}\left(  \overline{\pi}\left(\wF^u(\wz'_{(m,n)})\right) \cap \overline{\pi}\left(\wF^s \left( \wz \right) \right)\right) $$

where the order on the target set is as in Definition \ref{definition:order_on_crossing_point_intersections}.

Moreover, this correspondence maps $C_r(\wr_2)$ ($C_l(\wr_2)$) to the subset $CP_r(\wz)$ (resp. $CP_l(\wz)$) of $ \left\{\wG^u(m,n) \cap \wG^s(0,0): (m,n)\in CP(\wz) \} \right\}$ consisting of intersections \linebreak $\wG^s(0,0)\cap \wG^u(m,n)$ such that $(m,n)$ is to the right (resp. left) of $\wG^s(0,0)$.

\end{corollary}

\begin{proof}
    This follows from the fact that $\overline{\pi}: V \to \widehat{M}_1$ is injective in a neighborhood of $\wr_2$, and then by applying Proposition \ref{proposition:bijection_M1_integers}.
    \end{proof}

    Recall that the half-leaf $\wr_2$ was defined in the previous section as $\wr_2 = \pi_\mathcal{O}(E_2)$, where $E_2 \subset \wF^s \left( \wz \right) \setminus \walpha_1 \subset \widetilde{M_A}$ is a half-leaf of the lift $\wF^s \left( \wz \right)$ of $\F^s(z) \subset M_A$ to $\widetilde{M_A}$, and where $\walpha_1$ is the lift of the orbit of $z$ which is contained in the lifted leaf $\wF^s \left( \wz \right)$.
    
    In order to define our invariant for the bifoliated plane $(P_A, \F^+_A, \F^-_A)$, we consider what happens under a different choice of lift $\wF^s \left( \wz \right)$.
    \begin{proposition}\label{prop:FW_defined_up_to_order_pres}

    Let $\wr_2' = \pi_\mathcal{O}(E_2')$, for $E_2' \subset \wF^s(\wz_2) \subset \widetilde{M_A}$ the connected component of $\wF^s(\wz_2) \setminus \walpha_2$, where $\wz_2\in \widetilde{M_A}$ is a lift of $z\in M_A$ different from $\wz$ and $\walpha_2$ is the orbit of $\wz_2 $ under $\widetilde{\varphi}_A$.

    Then, there is an order preserving isomorphism $C(\wr_2) \to C(\wr_2')$ which maps $C_r(\wr_2)$ to $(C_r(\wr_2'))$ and $C_l(\wr_2)$ to $C_l(\wr_2')$.
    \end{proposition}

    \begin{proof}
        Let $T_g :\widetilde{M_A} \to \widetilde{M_A}$ denote the deck transformation mapping $\wz \in \widetilde{M_A}$ to $\wz_2\in \widetilde{M_A}$, corresponding to an element $g\in \pi_1(M)$. 
        
        Then, the induced action of $g$ on $P_A$ maps $\walpha \in P_A$ to $\walpha_2\in P_A$. In particular, it maps the lozenges $L_1, L_2$ having $\walpha$ as a corner to the lozenges $L_1', L_2'$ having $\walpha_2$ as a corner, and therefore maps the infinite perfect fit containing lozenges $L_1, L_2$ to the infinite perfect fit containing lozenges $L_1', L_2'$. This means that $\wr_2$ is mapped to $\wr'$ by the action of $g$. 
        
        Since $g$ must map crossing infinite perfect fits for $\wr_2$ to crossing infinite perfect fits for $\wr_2'$, and $g$ induces a homeomorphism $\wr_2 \to \wr_2'$, by definition of $C(\wr_2)$ there is an order preserving isomorphism $C(\wr_2) \to C(\wr_2')$. Moreover, since $g$ preserves orientation, $g$ maps points to the left of $\wr_2$ to points to the left of $\wr_2'$, and therefore this order preserving isomorphism maps $C_r(\wr_2) $ to $C_r(\wr_2')$ and $C_l(\wr_2)$ to $C_l(\wr_2')$, as we wanted.

    \end{proof}

Therefore, Proposition \ref{prop:FW_defined_up_to_order_pres} allows us to define the following:

\begin{definition}
We define the \emph{intersection pattern of crossing infinite perfect fits} of $\left( P_A, \F^+_A, \F^-_A\right)$ to be the triple $(C^A, C_r^A, C_l^A) = (C(\wr_2), C_r(\wr_2), C_l(\wr_2))$, for any choice of $\wr_2$ as described above. 
\end{definition}

\begin{definition}
Let $A, B \in \mathrm{SL}(2,\Z)$ be hyperbolic matrices.

    If there exists an order preserving isomorphism $C^A \to C^B$ which maps the subsets $C_r^A$ and $C_l^A$ to the subsets $C_r^B$ and $C_l^B$ respectively, we say that the triples $(C^A, C_r^A, C_l^A)$ and $(C^B, C_r^B, C_l^B)$ are equivalent.  
\end{definition}

Now, we prove that the triple $(C^A, C_r^A, C_l^A)$ up to equivalence is an invariant of $(P_A, \F^+_A, \F^-_A)$.

\begin{proposition}\label{lemma:crossings_to_crossings}

Let $H: (P_A, \F^+_A, \F^-_A)\to(P_B, \F^+_B, \F^-_B)$ be an isomorphism that preserves orientations of the foliations. 

Let $(C^A, C_r^A, C_l^A) = (C(\wr_2), C_r(\wr_2), C_l(\wr_2))$.

Then, for $I\cap \wr_2 \in C(\wr_2)$ we have that $H(I) \cap H(\wr_2) \in C(H(\wr_2))$, and the induced map 
$$\mathcal{P}(H) : C(\wr_2) \to C(H(\wr_2))$$ 
given by $\mathcal{P}(H)(I\cap \wr_2) = H(I) \cap H(\wr_2) $ is an order isomorphism and satisfies $\mathcal{P}(H)(C_r(\wr_2)) = C_r(H(\wr_2)), \, \mathcal{P}(H)(C_l(\wr_2)) = C_l(H(\wr_2))$.

\end{proposition}

\begin{proof}
     Let $H: (P_A, \F^+_A, \F^-_A)\to(P_B, \F^+_B, \F^-_B)$ be an isomorphism that preserves orientations of the foliations. 

  We know that $H$ must map lozenges to lozenges and preserve adjacency of lozenges. Therefore it maps infinite perfect fits to infinite perfect fits, and leaves making infinite perfect fits to leaves making infinite perfect fits. Since $H$ preserves orientations, if a leaf $s$ makes an infinite perfect fit to the right (left) of $\wr_2$, then $H(s)$ must make an infinite perfect fit to the right (left) of $H(\wr_2)$.

We now show that $H$ maps crossing infinite perfect fits for $\wr_2$ to crossing infinite perfect fits for $H(\wr_2)$.

Let $I$ be an infinite perfect fit for $\wr_2$, and let $s^+\in \F^+_A$ be the leaf bounding a lozenge in $I$, making an infinite perfect fit and intersecting $\wr_2$. Let $s_1^-\in \F^-_A$ be the leaf bounding a lozenge $L_1$ in $I$ and making an infinite perfect fit, where $L_1$ is the lozenge closest to $\walpha_1$.

By definition, $I$ is a crossing infinite perfect fit for $\wr_2$ if and only if $s_1^- \cap \F_A^+(\walpha_1) \neq \emptyset$. This happens if and only if $H(s_1^-) \cap \F_B^+(H(\walpha)) \neq \emptyset$, where $H(s_1^-) \in \F^-_B$ is the leaf of $\F^-_B$ making the infinite perfect fit $H(I)$ and bounding the lozenge $H(L_1)$ which is closest to $H(\walpha_1)$. Therefore, $H(I)$ is a crossing infinite perfect fit for $H(\wr_2)$ if and only if $I$  is a crossing infinite perfect fit for $\wr_2$, and this shows that $H(I \cap \wr_2) = H(I) \cap H(\wr_2) \in C(H(\wr_2))$.

Since $H$ restricts to a homeomorphism from $\wr_2$ to $H(\wr_2)$, the map $\mathcal{P}(H)$ must be order preserving on $C(\wr_2)$. By the second paragraph in the proof, $\mathcal{P}(H)$ must map $C_r(\wr_2)$ to $C_r(H(\wr_2))$ and $C_l(\wr_2)$ to $C_l(H(\wr_2))$, as we wanted.

\end{proof}

\begin{corollary}\label{corollary:FW_isomorphic_implies_equivalent_patterns}
    Let $A, B \in \mathrm{SL}(2,\Z)$ be hyperbolic matrices, and let \linebreak $\left(P_A, \F^+_A, \F^+_B\right)$ and $\left(P_B, \F^+_B, \F^-_B\right)$ be the bifoliated planes associated to the corresponding Franks-Williams flows.

 Then, if these bifoliated planes are isomorphic, they have equivalent intersection patterns of crossing infinite perfect fits $(C^A, C_r^A, C_l^A), \, (C^B, C_r^B, C_l^B)$. 
\end{corollary}

\begin{remark}

In a later section, we show that all the information in the invariant $(C^A, C_r^A, C_l^A)$ can be expressed by a bi-infinite sequence of natural numbers. The key to showing this will be Corollary \ref{proposition:equivalence_crossing_integer_points}.

\end{remark}

\section{Continued fractions and the integer lattice}

\subsection{Definitions and elementary properties of continued fractions}\label{subsection:definitions_elementary_props_continued_fractions}

We briefly recall some elementary definitions and facts from the theory of continued fractions, following closely the exposition in \cite{khinchin1964continued}. Although in \cite{khinchin1964continued} there is no explicit mention of lower and upper approximations, the results relating to these can be obtained by adapting the arguments presented there. 

The most important result stated here is Proposition \ref{proposition:good_approx_form_convergents}, which identifies exactly what the lower and upper good approximations of the second kind are, for $\alpha > 0$.

\begin{definition}

Given a natural number $a$, let $T_a: \R \to \R$ be the linear fractional transformation given by $T_a(x) = a + \frac{1}{x} $.

Given $a_0, a_1,\dots, a_n$ natural numbers, define $[a_0; a_1, \dots, a_n] = T_{a_0}\circ T_{a_1} \circ \cdots \circ T_{a_{n-1}}(a_n)$.
That is,
\[
[a_0; a_1, a_2, \dots, a_n] = 
a_0+ \cfrac{1}{\displaystyle a_1 +
  \cfrac{1}{\displaystyle \ddots +
  \cfrac{1}{\displaystyle a_{n-1} + \cfrac{1}{a_n}}}}
\]
    
\end{definition}

\begin{proposition}

    If $a_0, a_1, a_2, \dots$ are natural numbers, the limit \linebreak $\lim_{n\to +\infty} [a_0; a_1, \dots, a_n] $ exists.
\end{proposition}

\begin{definition}
    For natural numbers $a_0, a_1, \dots$, define $$[a_0; a_1,a_2, \dots] =  \lim_{n\to +\infty} [a_0; a_1, \dots, a_n]. $$
\end{definition}

\begin{proposition}
An irrational number $\alpha > 0$ can be represented by a unique continued fraction.
\end{proposition}
\begin{remark}
    If we have $\alpha < 0$, then we represent $\alpha$ as $\alpha = - [a_0; a_1, a_2,\dots]$, where $a_0 \geq 0$ and $a_i\geq 1$ for all $i\geq 1$. 
\end{remark}

\begin{definition}

We say a continued fraction $\alpha = [a_0; a_1,a_2,  \dots]$ is periodic if there exist natural numbers $k_0$ and $l$ such that for all $k\geq k_0$ we have $a_{k+l} = a_k$.

In that case, we denote the fraction corresponding to this sequence as 
$$[a_0; a_1, \dots, a_{k_0 - 1}, \overline{a_{k_0}, a_{k_0+1}, \dots, a_{k_0+l - 1}}].$$
\end{definition}

\begin{proposition}
    A number $\alpha \in \R \setminus \Q $ is a quadratic algebraic integer, i.e. a root of a polynomial in $\Z[x]$ of degree $2$, if and only if $\alpha$ is represented by a periodic continued fraction.
\end{proposition}

\begin{definition}
    If $\alpha = [a_0; a_1, a_2, \dots]$, then for $n\geq 0$ the $n$\emph{-th convergent} $\frac{p_{n}}{q_n}$ of $\alpha$ is defined to be $\frac{p_n}{q_n} = [a_0; a_1, \dots, a_n] $. Here we assume that the fraction $[a_0; a_1, \dots, a_n]$ is reduced, so that $p_n$ and $q_n$ are coprime.

    We also define the (formal) $(-1)$-th convergent of $\alpha$ to be $\frac{p_{-1}}{q_{-1}} = \frac{1}{0}$.

\end{definition}

\begin{proposition}
    Let $\alpha \in \R$. For all $k\geq 0$, we have 
    \[
    \frac{p_{2k}}{q_{2k}} < \frac{p_{2k + 2}}{q_{2k+2}} \leq \alpha \leq \frac{p_{2k+3}}{q_{2k+3}} < \frac{p_{2k +1}}{q_{2k+1}}.
    \]
    That is, the even convergents form an increasing sequence with upper bound $\alpha$, and the odd convergents form a decreasing sequence with lower bound $\alpha$. Moreover, both sequences converge to $\alpha$.

\end{proposition}

Not only is it the case that the sequences of convergents of a number $\alpha$ will converge to it: they also give good rational approximations to $\alpha$, in the sense defined below.

    \begin{definition}
        We say a fraction in lowest terms $\frac{p}{q}\in \Q$ is a best approximation of the second kind of $x\in \R$ if: for all $\frac{p'}{q'} \in \Q$ such that $q'\leq q$, we have $\lvert p - q\alpha \rvert \leq \lvert p' - q'\alpha\rvert $.

    \end{definition}

\begin{remark}
    The reason that we say ``of the second kind'' in the definition below is that we are measuring the approximation error using the expression $\lvert p - q\alpha \rvert$ instead of using $\lvert \frac{p}{q} - \alpha \rvert$. If we used the latter, then we would speak of best approximations of the first kind; however, we will not be needing these here. 
    
    One should think of $\lvert p - q\alpha \rvert$ as measuring the vertical distance between the integer point $(q,p)\in \Z^2$ and the line $y = \alpha x$ of slope $\alpha$, while $\lvert \frac{q}{p} - \alpha \rvert$ would instead measure the difference in slopes between this line and the line connecting $(q,p)$ to the origin.
\end{remark}
    
For our later work, we will want to compare an approximation of a number $\alpha$ from below (above) only to approximations from below (resp. above). We also slightly relax the definition by avoiding the comparison with other fractions with the same denominator; this only makes a difference in the case where the denominator equals $1$.

    \begin{definition}
        We say a fraction in lowest terms $\frac{p}{q}\in \Q$ is a good lower approximation of the second kind of $\alpha \in \R$ if:

        \begin{itemize}
            \item $p - q\alpha < 0$.
            \item For all $\frac{p'}{q'} \in \Q$ such that $q' < q$ and $p' -q' \alpha<0$, we have $\lvert p-q\alpha \rvert < \lvert p' -q'\alpha\rvert $.
        \end{itemize}    
    Given $\frac{p}{q}, \frac{p'}{q'}\in \Q$ such that $p - q\alpha < 0, $ and $ p'-q'\alpha < 0$, we say that $\frac{p}{q}$ is a better lower approximation of the second kind of $\alpha$ than $\frac{p'}{q'}$ if $\lvert p - q\alpha\rvert< \lvert p'-q'\alpha\rvert$.
    \end{definition}

\begin{proposition}\label{proposition:inequalities_convergents}
Let $\alpha = [a_0; a_1, \dots]$.
For all even $k\geq 0$ and $0 \leq r < a_{k+2}$, we have
\[
   \frac{p_k+ rp_{k+1}}{q_k + r q_{k+1}} < \frac{p_k+ (r+1)p_{k+1}}{q_k + (r+1) q_{k+1}} < \alpha.
\]
If instead $k$ is odd, we get 
\[
\frac{p_k+ rp_{k+1}}{q_k + r q_{k+1}} > \frac{p_k+ (r+1)p_{k+1}}{q_k + (r+1) q_{k+1}}>\alpha
\]
and in either case,
\[
 \frac{p_k+ a_{k+2}p_{k+1}}{q_k + a_{k+2} q_{k+1}} = \frac{p_{k+2}}{q_{k+2}}.
\]
\end{proposition}

\begin{proposition}\label{proposition:good_approx_form_convergents}
    Let $\alpha > 0$. Then, $\frac{a}{b} \in \Q$ is a good lower (upper) approximation of the second kind of $\alpha = [a_0; a_1, a_2, \dots, ]$ if and only if $\frac{a}{b}$ is of the form
    \[
    \frac{a}{b} = \frac{p_k + rp_{k+1}}{q_k + r q_{k+1}},
    \]
    where $k$ is even (resp. odd), $\frac{p_k}{q_k}$ is the $k$-th convergent of $\alpha$, and $r\in \{ 0,1, \dots, a_{k+2}\}$.
\end{proposition}

\subsection{Relating continued fractions to parallelograms not containing integer points}\label{subsection:parallelograms_integer_points_continued_fractions}

Fix a hyperbolic matrix $A\in \mathrm{SL}(2,\Z)$.

We will denote by $\alpha$ and $\beta$ the slopes of the contracting and expanding eigenspaces of $A$. Then, $\alpha$ and $\beta$ are quadratic irrationals, with $\alpha = \frac{b + \sqrt{c}}{d}$ and $\beta = \overline{\alpha} = \frac{b - \sqrt{c}}{d}$, for $b,c,d \in \Z,\, d\neq 0$. 

Some of the results hold in more generality for $\alpha$ and $\beta$ arbitrary irrational numbers, but we will not make use of this.

 We assume without loss of generality that $\alpha > \beta$ and $\alpha > 0$. The other cases can be handled by applying appropriate symmetries, keeping in mind that these will possibly convert good upper approximations into good lower approximations, and viceversa.

In the first part of this section, we work towards defining a bi-infinite sequence $\sigma(\alpha)$ associated to $\alpha$. This is done in Definition \ref{definition:sigma_bi_infinite_sequence}. Then, we show in Proposition \ref{proposition:sigma_continued_fraction_equal} and Corollary \ref{corollary:periodic_sigma_sequence} that this sequence is determined by the periodic part of the continued fraction of $\alpha$.

\begin{definition}
    Let $(m,n) \in \Z^2$. Let $R_{\alpha,\beta, (m,n)}  \subset \R^2$ be the parallelogram bounded by the four lines of slopes $\alpha$ and $\beta$ which pass through $(m,n)$ or through the origin. 
    
\end{definition}

Since here $\alpha $ and $\beta$ are fixed, we will write $R_{(m,n)} = R_{\alpha, \beta, (m,n)}$.

\begin{lemma}\label{lemma:epsilon_delta}
Let $\alpha, \beta \in \R$ with $\alpha \neq \beta$. For a point $x\in \R^2$, denote by $r_\beta(x)$ the line through $x$ with slope $\beta$, and by $r_\infty(x)$ the vertical line through $x$.

For any $\varepsilon>0$ there exists $\delta > 0$ such that: for any point $x$ on the plane and any line $r_\alpha$ of slope $\alpha$ such that $d(x, r_\alpha) < \delta$ we have 
$$\max \left[ d(x, r_\alpha \cap r_\beta(x)), d(x, r_\alpha \cap r_{\infty}(x)), d(x, r_\alpha \cap r_0(x)) \right]< \varepsilon .$$
\end{lemma}

The proof of the lemma is elementary: one only needs to note that the distance between $x $ and the points $r_\alpha\cap r_\beta(x)$, $r_\alpha \cap r_\infty(x)$ and $r_\alpha \cap r_0(x)$ varies continuously with $d(x,r_\alpha)$, and converges to $0$ as $d(x,r_\alpha) \to 0$.

\begin{definition}\label{definition:t_m_n}
    Let $(m,n)\in \Z^2 $. Let $c:[0, +\infty) \to \R^2$ be the parametrization of a ray in the line $r_\alpha(0)$ with slope $\alpha$, given by $c(t) = t (1, \alpha) $. Let $t_{(m,n)} \in \R$ be defined by $t_{(m,n)} = \frac{n - \beta m}{\alpha - \beta} $, the parameter such that $c(t_{(m,n)})$ is the intersection of the ray $r_\alpha$ and the line of slope $\beta$ passing through $(m,n)$.

\end{definition}

\begin{proposition}\label{proposition:characterization_good_approx}
 Let $(m,n)\in (\Z_{\geq 0})^2 $.

There exists $t_0 \in \R^+$ (which depends on $\alpha$ and $\beta$) such that if $t_{(m,n)} \geq t_0$, then $R_{(m,n)}$ does not contain any points of $\Z^2$ in its interior if and only if $\frac{n}{m}$ is a good lower or upper approximation of the second kind for $\alpha$. 

Moreover, we may choose $t_0$ so that:
\begin{enumerate}
    \item 
$t_0 = t_{(M_0, N_0)} $ for $(M_0, N_0)\in (\Z_{\geq 0})^2$, where $\frac{N_0}{M_0} < \alpha$ is a good lower approximation for $\alpha$.

\item For $(m,n) \in (\Z_{\geq 0})^2$ such that $t_{(m,n)} \geq t_0$ and $\frac{n}{m}$ is a good lower or upper approximation for $\alpha$, we have that
\begin{itemize}
    \item $d((m,n)), r_\alpha(0) \cap r_\beta(m,n))$,
    \item $d((m,n), r_\alpha(0) \cap r_{\infty}(m,n)) $,
    \item $d((m,n), r_\alpha(0) \cap r_0(m,n))$
    
\end{itemize}
are all less than $\frac{1}{4}$.

\end{enumerate}

\end{proposition}

\begin{proof}
We prove it in the case $\alpha > \beta$: the other case can be handled in a symmetric way.

For $\varepsilon = \frac{1}{4}$, take $\delta >0$ as in Lemma \ref{lemma:epsilon_delta} above.

Let $(M_0,N_0), (M_1, N_1) \in \Z^2$ be such that: 
\begin{itemize}
    \item $\frac{N_0}{M_0} > \beta$.
    \item $\frac{N_0}{M_0}$ is a good lower approximation of the second kind for $\alpha$, and $\frac{N_1}{M_1}$ a good upper approximation of the second kind for $\alpha$.
    \item $d((M_i, N_i), r_\alpha(0)) < \min(1, \delta)$ for $i=0,1$.
    \item $t_{(M_0, N_0)} \geq t_{(M_1, N_1)}$.
\end{itemize}
The points $(M_0, N_0), (M_1, N_1)$ can be chosen, for instance, to have as coordinates the denominators and numerators of appropriate even and odd convergents of $\alpha$, respectively.

Let $t_0 =\max( t_{(M_0, N_0)}, t_{(M_1, N_1)}) = t_{(M_0, N_0)}$.
\begin{claim}
    For $(m,n)\in (\Z_{\geq 0})^2$ such that $t_{(m,n)} \geq t_0$, if the parallelogram $R_{(m,n)}$ does not contain a point of $\Z^2$ in its interior, then any point $(m',n')\in (\Z_{\geq 0})^2$ in the interior of the region bounded by the lines $r_\alpha(0)$ and $r_\alpha((m,n))$ must satisfy $m'>m$. 
\end{claim}

We now prove the claim:

Let $(m,n)\in (\Z_{\geq 0})^2$ be a point such that $t_{(m,n)} \geq t_0$.

Suppose first that $(m,n)$ is above the line $r_\alpha(0)$, that is, $n > m\alpha > 0$. Note that since $\frac{n}{m}> \alpha > \beta$, the line $r_\infty(0)$ (i.e. the vertical axis) intersects the side of $R_{(m,n)}$ which has slope $\alpha$ and contains $(m,n)$. Then, any point $(m',n')$ of $(\Z_{\geq 0})^2$ in the interior of the region bounded by the lines $r_\alpha(0)$ and $r_{\alpha}(m,n)$ that satisfies $m' \leq m$ must be in the interior of the parallelogram $R_{(m,n)}$. But we know that this parallelogram does not contain any integer points in its interior, so no such $(m',n') \in (\Z_{\geq 0})^2$ can exist.

Now suppose that $(m,n) \in (\Z_{\geq 0})^2$ is below that line $r_\alpha(0)$, i.e. $0 < n < \alpha m$, satisfies $t_{(m,n)} \geq t_0 = t_{(M_0,N_0)}$ and $R_{(m,n)}$ does not contain any point of $\Z^2$ in its interior. First, note that the fact that $t_{(m,n)} \geq t_{(M_0,N_0)}$ and $R_{(m,n)}$ does not contain $(M_0,N_0)$ implies that $d(r_\alpha(0), (m,n) ) < d(r_\alpha(0), (M_0,N_0))$. Therefore, $d((m,n), r_\alpha(0)) < \delta$.

Let $z\in \R^2$ be the corner of the parallelogram $R_{(m,n)}$ given by $\{ z\} = r_\alpha(m,n) \cap r_\beta(0)$.

We have:

\begin{align*}
   \{(m',n') \in (\Z_{\geq 0})^2 : (m',n') \text{ is between } r_{\alpha}(0), \,r_\alpha(m,n) \text{ and } m'\leq m  \} \\
   \subseteq (T_1 \cap R_{(m,n)} \cap T_2) \cap (\Z_{\geq 0})^2
\end{align*}

where:
\begin{itemize}
    \item $T_1$ is the triangle with corners $(0,0)$ and $z$ and bounded by $r_\beta(0), r_\alpha(z)$ and $r_\infty(0)$ (the $y$-axis).
    \item $T_2$ is the triangle with corner $(m,n)$ bounded by $r_\alpha(0), r_\beta(m,n)$ and $r_\infty(m,n)$.
\end{itemize}

Note that: the triangles $T_1$ and $T_2$ are congruent, and if $\beta \leq 0$, $T_1 \cap (\Z_{\geq 0})^2 = \emptyset$ (we will not use this latter fact, but will use the former).

Since $d((m,n), r_\alpha(0)) < \delta$, we have that, by Lemma \ref{lemma:epsilon_delta}, distances \linebreak $d((m,n), r_\alpha(0) \cap r_\beta(m,n))$ and $d((m,n), r_\alpha(0) \cap r_\infty(m,n))$ are less than $\frac{1}{4}$. These are the distances between $(m,n)$ and the other two corners of $T_2$. Since $(m,n) \in \Z^2$, no other point of $\Z^2$ can be contained in $T_2$. By congruence of $T_1$ and $T_2$ (with the point $(m,n)$ in $T_2$ corresponding to $(0,0)$ in $T_1$), the same is true of $T_1$.

Since we know that $R_{(m,n)} \cap (\Z_{\geq 0})^2 = \{ (0,0), (m,n)\}$, we conclude that there are no points $(m',n')\in (\Z_{\geq 0})^2$ that are in the interior of the region bounded by the lines $r_\alpha(0)$ and $r_\alpha(m,n)$ and satisfy $m'\leq m$.

This concludes the proof of the claim.


Suppose that $t_{(m,n)} \geq t_0$ and $R_{(m,n)}$ does not contain any points of $\Z^2$ in its interior. We show that $\frac{n}{m}$ is a good lower or upper approximation of the second kind for $\alpha$. By the claim above, any point 
$(m',n')$ in $(\Z_{\geq 0})^2$ on the same side of $r_\alpha(0)$ as $(m,n)$ which is closer than $(m,n)$ to the line $r_\alpha(0)$ must satisfy $m'>m$. This shows that $\frac{n}{m}$ is a good lower or upper approximation of the second kind for $\alpha$, since $d((m,n), r_\alpha(0)) < d((m',n'), r_\alpha(0)) $ implies that $\left| n - m \alpha  \right|< \left| n' - m\alpha' \right| $.

Suppose now that $(m,n)$ is such that $t_{(m,n)} \geq t_0$ and $\frac{n}{m}$ is a good lower approximation of the second kind for $\alpha$. We show that there are no points of $(\Z_{\geq 0})^2$ in the interior of $R_{(m,n)}$. Suppose that a point $(m', n')\in  \Z^2$ is in the interior of $R_{(m,n)}$. This implies $m'< m$ (here we use that $(m,n)$ and $(m',n')$ are below the line $r_\alpha(0)$), $n' - m'\alpha < 0$ and $d((m',n'), r_\alpha(0)) < d((m,n), r_\alpha(0))$, but then the vertical distances between the points $(m',n'), \, (m,n)$ and the line $r_\alpha(0)$ must be related in the same way, i.e. $\left| n - \alpha m \right| > \left| n' - \alpha m' \right| $. This contradicts the fact taht $\frac{n}{m}$ is a good lower approximation of the second kind for $\alpha$, so no such point $(m',n')$ can exist.

Finally, suppose that $(m,n)$ is such that $t_{(m,n)} \geq t_0$ and $\frac{n}{m}$ is a good upper approximation of the second kind for $\alpha$, and we show that there are no points of $(\Z_{\geq 0})^2$ in the interior of $R_{(m,n)}$. By the same argument as above, since $\frac{n}{m}$ is a good upper approximation of the second kind for $\alpha$, there cannot be any points $(m',n')$ in the interior of $R_{(m,n)}$ that satisfy $m' < m$. It only remains to show that there are no points $(m',n')$ in the interior of $R_{(m,n)}$ such that $m' \geq m$. Any such point would have to be contained in the triangle $T$ with corner $(m,n)$ and sides $r_{\infty}(m,n)$, $r_\alpha(0)$ and $r_\beta(m,n)$. Using the same argument as in the proof of the claim above (here we use the fact that $t_{(m,n)} \geq t_0$), we can see that the triangle $T$ does not contain any integer points other than $(m,n)$. Therefore, there are no integer points in the interior of $R_{(m,n)}$, and we are done.

\end{proof}

\begin{proposition}\label{proposition:order_integers_Tpm}

Let $r$ be the ray based at $0$ parametrized by $c: [0, +\infty)\to [0, +\infty)$, $c(t) =  t(1, \lfloor \alpha \rfloor)$, $r\subset r_\alpha$ with $r_\alpha$ the line through the origin of slope $\alpha$. 

Let 
\begin{align*}
    T^-_A = \left\{t_{(m,n)}:  t_{(m,n)} > 0,\, n-\alpha m < 0, \, \mathrm{int}\left(R_{(m,n)}\right) \cap \Z^2 = \emptyset  \right\}\\
T^+_A = \left\{ t_{(m,n)}: t_{(m,n)} >0,\, n-\alpha m > 0, \, \mathrm{int}\left(R_{(m,n)}\right) \cap \Z^2 = \emptyset \right\}\\
\end{align*}
where $t_{(m,n)}$ is defined as in Definition \ref{definition:t_m_n}. 

Then, if we order the elements of $T^+_A$, $T^-_A$ and $T^+_A \cup T^-_A$ according to the usual order on $\R$, all of these sets are order isomorphic to the integers with the usual order.
    
\end{proposition}
\begin{remark}
    When the matrix $A$ is implicit, as in the rest of this chapter, we will simply write $T^\pm = T^\pm_A$.
\end{remark}

Before proving Proposition \ref{proposition:order_integers_Tpm}, we state a lemma which is a consequence of the definition of $t_{(m,n)}$ and the fact that $\alpha, \beta$ are the slopes of the contracting and expanding eigenvectors of the hyperbolic matrix $A \in \mathrm{SL}(2,\Z)$.

\begin{lemma}\label{lemma:action_A_Tsets}
    For the action of the matrix $A$ on $T = T^+ \cup T^-$ defined by $A \cdot t_{(m,n)} = t_{(m,n) A^t}$, the following hold:
    \begin{enumerate}
    \item $A \cdot T^+ = T^+$, and $A\cdot T^- = T^-$.
        \item The action is order preserving on $T$.
        \item For $k\to + \infty$, we have $A^k\cdot t_{(m,n)} \to 0$.
        \item For $k \to -\infty$, we have $A^k \cdot t_{(m,n)} \to +\infty$.
    \end{enumerate}
\end{lemma}

\begin{proof}[Proof of Proposition \ref{proposition:order_integers_Tpm}]
  
We show it for $T^-$, with the proof for $T^+$ being analogous.

It suffices to show two properties of $T^-$: between any two elements there exist only finitely many other elements, and $T^-$ is unbounded above and below.

Let $t_1 = t_{(m_1,n_1)}, t_2 = t_{(m_2, n_2)} \in T^-$, and assume without loss of generality that $t_2 > t_1$. Note that since $t_2 \in T^-$, we must have that $(m_1, n_1)$ is not contained in $R_{(m_2, n_2)}$.

Let then $R$ be the region bounded by the lines \linebreak $r_\alpha(0), r_\alpha(m_1, n_1), r_{\beta}(m_1, n_1), r_\beta(m_2, n_2)$. There exist only finitely many points of $\Z^2$ which are contained in $R$. 

If $(m,n)$ is in the same connected component of $\R^2 \setminus r_\alpha(0)$ as $(m_1, n_1)$ and we have $ t_1< t_{(m,n)} < t_2 $, then $R_{(m,n)}$ contains no points of $\Z^2$ in its interior only if $(m,n)$ is contained in $R$: otherwise $R_{(m,n)}$ would contain $(m_1, n_1)$. Therefore, there exist finitely many elements $t_{(m,n)} \in T_0^-$ such that $t_1 < t_{(m,n)} < t_2$.

In order to show that $T^-$ has the same order type as the integers, it only remains to show that it is unbounded above and below. This is implied by the last two items in Lemma \ref{lemma:action_A_Tsets}, so we have shown what we wanted for $T^-$ (and via the same argument, for $T^+$).

Now, in order to show that $T$ has the same order type as the integers, it is enough to show that between any two elements of $T^+$ there exist finitely many elements of $T^-$, and viceversa, since we know that $T^+$ and $T^-$ are subsets of $T$ with the order type of the integers. Let $t_1 = t_{(m_1, n_1)}, \,t_2 = t_{(m_2, n_2)}\in T^+$ such that $t_1 < t_2$. We can assume without loss of generality that $t_1 = A\cdot t_2$, since otherwise we can replace $t_1$ by $t_3 = A^k \cdot t_2$ for a large enought $k> 0$. An argument analogous to the one used above shows that there exist finitely many $t_{(m,n)}\in T^-$ such that $t_1 < t_{(m,n)} < t_2$. Therefore, since the interval $[t_1, t_2]\subset T$ is a fundamental domain for the action of $A$ on $T$, we conclude that given any two elements of $T^+$ there exist finitely many elements of $T^-$ between them. This shows what we wanted.

\end{proof}

\begin{definition}
    Using the order on $T^\pm$ described in the proposition above, we order and relabel elements of $T^\pm$, so that
    \begin{align*}
        T^- &= \{ \dots,t_{-1}^-, t_0^-, t_1^-, t_2^-,\dots\} \text{ with } \dots < t_{-1}^- < t_0^- < t_1^- < t_2^- < \dots\\
        T^+ &= \{\dots,t_{-1}^+, t_0^+, t_1^+, t_2^+, \dots\} \text{ with } \dots < t_{-1}^+ < t_0^+<t_1^+ < t_2^+ < \dots
    \end{align*}
    where $t_0^- = t_0$, for $t_0$ as in the statement of Proposition \ref{proposition:characterization_good_approx}. 
\end{definition}

\begin{figure}[h]
    \centering
    \includegraphics[width=.99\textwidth]{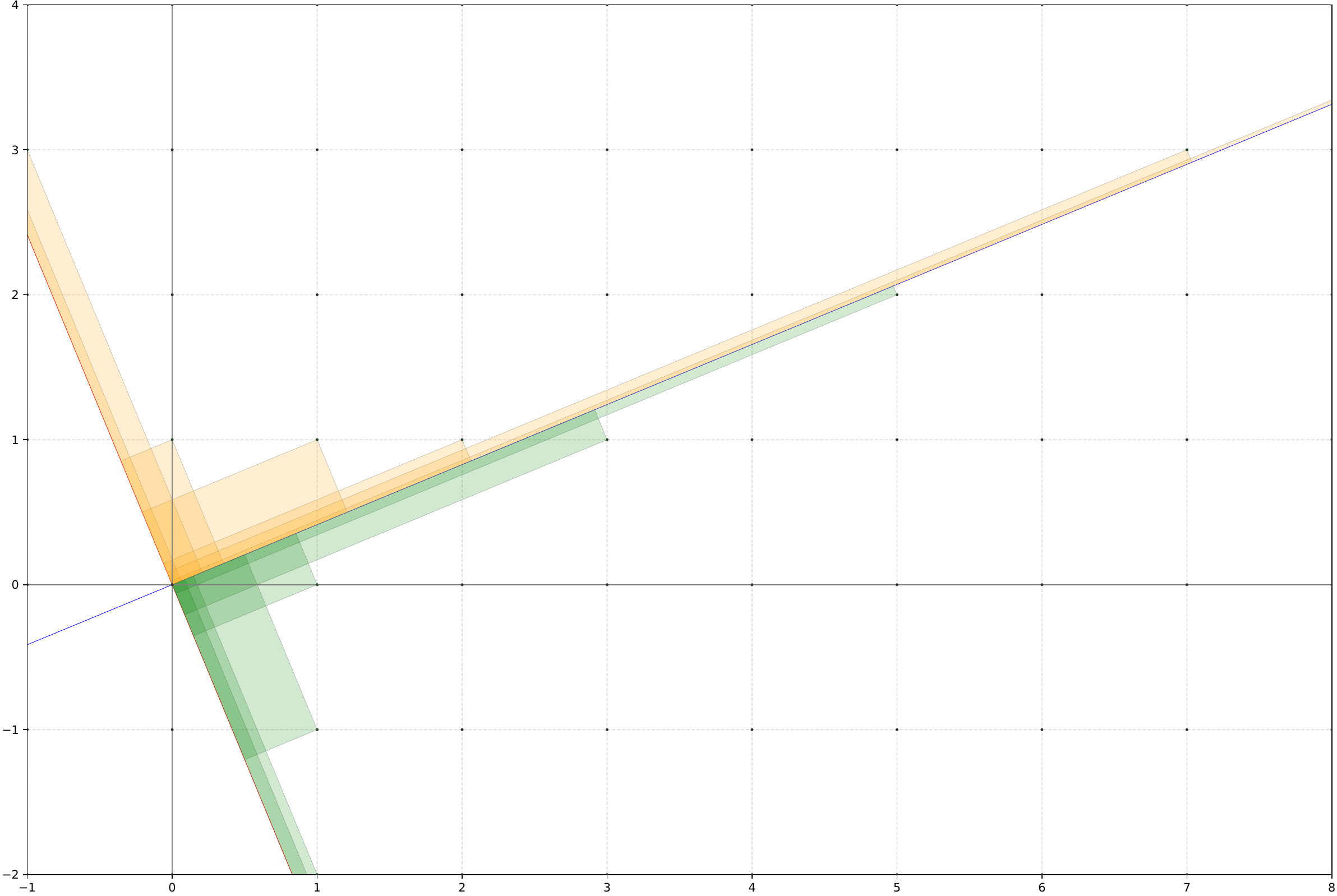}
    \caption{Eigenspaces of $\begin{pmatrix}
        1&2\\-2&5
    \end{pmatrix} = \begin{pmatrix}
        5&2\\2&1
    \end{pmatrix}^{-1}$, and the parallelograms $R_{(m,n)}$.}
    \label{fig:eigenspaces_example}
\end{figure}

We will be interested in understanding how the elements of $T^+, T^-$ are ordered inside $T$, since this is what will allow us to extract information about the number $\alpha$. However, we will not need to know the specific values of these elements: we are only interested in whether they belong to $T^-$ or to $T^+$.

For example: suppose that we know $ t_{2i}^- < t_{2i+1}^- < t_i^+ < t_{2i+2}^-$ for all $i$. Then, we can assign to $T$ a bi-infinite sequence $(\cdots, -, -, +, -, -, +, -,-, +,\cdots)$ that reflects this behavior, but does not specify what the values of the $t_i^+, t_i^-$ are.

For convenience, since we know that in such a sequence there will always be a ``$+$'' symbols after (and before) some finite number of ``$-$'' symbols and viceversa, we will instead construct a bi-infinite sequence of natural numbers $(\cdots, b_{-1}, b_0, b_1, b_2, \cdots)$, with $b_0 $ representing the size of the cluster $C_0^-$ of ``$-$'' symbols containing the one corresponding to $t_0^-$, $b_1$ representing the size of the cluster $C_1^+$ of ``$+$'' symbols adjacent to $C_0^-$ on the right, $b_2$ representing the size of the cluster $C_1^-$ of ``$-$'' symbols adjacent to $C_1^+$ on the right, and etc.

In the example given above, the bi-infinite sequence of natural numbers corresponding to $(\cdots, -, -, +, -, -, +, -,-, +,\cdots)$ would be given by \linebreak $(\cdots, 2,1,2,1,2,\cdots)$, with $b_0 = 2, b_1 = 1, b_2 = 2$, etc.

Below, we formalize this construction.

\begin{definition}
    Let $\sim^+, \sim^-$ be the equivalence relations on $T^+$ and $T^-$ given by $t^\pm_i \sim^\pm t^\pm_j$ if and only if $t^\pm_i \leq t^\pm_j$ and there does not exist $t^\mp_k$ such that $t^\pm_i \leq t_{k}^\mp \leq t_{j}^\pm$.
\end{definition}

\begin{proposition}
    If we let the set of equivalence classes $\tau^\pm = T^\pm/\sim^\pm$, then there is an induced total order on $\tau^\pm$ given by:

    For $C, C' \in \tau^\pm$, $C<C'$ if and only if $t < t'$ for all $t\in C, \, t'\in C'$.
\end{proposition}

\begin{definition}
    We relabel the elements of $\tau^+, \tau^-$ so that 
    \begin{align*}
        \tau^- = \{\cdots, C_{-1}^-,  C_0^-, C_1^-, \cdots\}\\
        \tau^+ = \{\cdots, C_{-1}^+, C_0^+,  C_1^+, C_2^+, \cdots\}
    \end{align*}
    and $C_i^\pm < C_{i+1}^\pm$ for all $i$.
\end{definition}

\begin{definition}\label{definition:sigma_bi_infinite_sequence}
Let $\sigma(\alpha) = (\dots, b_{-1}, b_0, b_1, b_2,\dots) $ where
\[
b_{2i} = \left| C_i^- \right|, \, b_{2i+1} = \left| C_i^+ \right| \text{ for } i\in \Z, 
\] 
\end{definition}

From the definitions of $\tau$ and $\sigma(\alpha)$, we get the following:

\begin{lemma}\label{lemma:A_action_sigma_periodic}
    The action of $A$ on $T$ defined in Lemma \ref{lemma:action_A_Tsets} induces an order preserving action on $\tau$. As a consequence, $\sigma(\alpha)$ is periodic.
\end{lemma}

\begin{proposition}\label{proposition:sigma_continued_fraction_equal}
    If $\alpha =[a_0;a_1, a_2, \dots] $ and $\sigma(\alpha) = (\dots, b_{-1}, b_0, b_1, b_2, \dots) $, then there exists $M\geq 0$ even, such that $b_k = a_{k + M}$ for all $k\geq 1$.
\end{proposition}

\begin{proof}
    Recall that Proposition \ref{proposition:characterization_good_approx} tells us that if $t_{(m,n)} \geq t_0$, then $R_{(m,n)}$ contains no points of $(\Z_{\geq 0})^2$ in its interior if and only if $\frac{n}{m}$ is a good lower approximation of the second kind for $\alpha$.
Define subsets $T_0^\pm \subset T^\pm$ as
\begin{align*}
    T^-_0 = \{t_{(m,n)}:  t_{(m,n)}\geq t_0,\, n-\alpha m < 0, \, R_{(m,n)} \text{ does not contain}\\ \text{points of } (\Z_{\geq 0})^2 \text{ in its interior}   \}\\
T^+_0 = \{ t_{(m,n)}: t_{(m,n)} \geq t_0,\, n-\alpha m > 0, \, R_{(m,n)} \text{ does not contain}\\ \text{points of } (\Z_{\geq 0})^2 \text{ in its interior}   \}\\
\end{align*}

 By definition of the sets $T_0^\pm$, we know that these sets contain all $t_{(m,n)}$ where $t_{(m,n)} \geq t_0$ for $(m,n) \in(\Z_{\geq 0})^2$ such that $R_{(m,n)}$ does not contain points of $\Z^2$ in its interior. Conversely, for any point $(m,n)\in (\Z_{\geq 0})^2$ such that $t_{(m,n)} \geq t_0$ and $R_{(m,n)}$ does not contain points of $\\Z^2$ in its interior, we know that $t_{(m,n)} \in T_0^\pm$.

 Therefore, an alternative characterization of $T_0^\pm$ is as follows:
 \[T_0^- = \left\{ t_{(m,n)}: (m,n)\in (\Z_{\geq 0})^2,\, t_{(m,n) } \geq t_0, \frac{n}{m} \text{ good lower approximation of } \alpha \right\}  \]
 \[T_0^+ = \left\{ t_{(m,n)}: (m,n)\in (\Z_{\geq 0})^2,\, t_{(m,n) } \geq t_0, \frac{n}{m} \text{ good upper approximation of } \alpha \right\}  \]

 For $t_i^\pm \in T_0^\pm$, let $t_i^\pm = t_{(m_i^\pm, n_i^\pm)} $. 

\begin{claim}
    $t_i^\pm \geq t_j^\pm$ if and only if $m_i^\pm \geq m_j^\pm$.
\end{claim}
 

First we prove the claim: 

 Recall from the statement of Proposition \ref{proposition:characterization_good_approx} that for $(m,n)\in (\Z_{\geq 0})^2$ such that $t_{(m,n)} \geq t_0$ and $\frac{n}{m}$ is a good lower or upper approximation for $\alpha$, we have that 
\begin{itemize}
    \item $d((m,n)), r_\alpha(0) \cap r_\beta(m,n))$,
    \item $d((m,n), r_\alpha(0) \cap r_{\infty}(m,n)). $

\end{itemize}
are all less than $\frac{1}{4}$, where $r_\gamma(x)$ denotes the line with slope $\gamma$ and containing $x\in \R^2$.

The claim follows from the above, since it implies that the intersections of the vertical lines $r_\infty(m,n)$ with $r_\alpha(0)$ are ordered (in $r_\alpha(0)$) in the same way as the intersections of the lines $r_\beta(m,n)$ with $r_\alpha(0)$. The latter are ordered (by definition of $t_{(m,n)}$) in the same way as the parameters $t_{(m,n)}$.


 Recall that Proposition \ref{proposition:good_approx_form_convergents} tells us that all good lower or upper approximations of the second kind for $\alpha$ are of the form
 $ \frac{p_{k} + r p_{k}}{q_{k} + rq_{k}}$, where $k\geq 0$ and $r\in \{ 0, 1,\dots, a_{k+2}\}$ and $\frac{p_k}{q_k}$ is the $k$-th convergent of $\alpha$.

 Moreover, we know by Proposition \ref{proposition:inequalities_convergents} that for $j<k$ we have $q_k + rq_{k+1} < q_{k+2}$ for $r \in \{ 0,1, \dots, a_{k+2} - 1\}$ and $q_k + a_{k+2}q_{k+1} = q_{k+2}$. Therefore, the denominators of the best lower and upper approximations for $\alpha$ are ordered as follows, for all $k\geq 0$:
 \begin{align*}
     q_{2k+1} &\leq q_{2k} + q_{2k+1} \leq q_{2k} + 2 q_{2k+1} \leq \dots \leq q_{2k} + a_{2k+2}q_{2k+1} = q_{2k+2}\\
     q_{2k+2} &\leq q_{2k+1} + q_{2k+2} \leq q_{2k+1} + 2 q_{2k+2} \leq \dots \leq q_{2k+1} + a_{2k+3} q_{2k+1} = q_{2k+3} 
 \end{align*}

 Thus, this and the Claim above allow us to conclude that if $k_1\geq 0$ and $r\in \{ 0,1,\dots, a_{k_1+2} - 1\}$ are such that $\frac{N_0}{M_0} =  \frac{p_{k_1} + r p_{k_1+ 1}}{q_{k_1} + r q_{k_1 + 1}}$ (this implies $k_1$ is even), we have the following: defining $m^k_r = q_{k} + rq_{k+1},\, n^k_r = p_k + rp_{k+1}$, we get for all $k\geq 1$ that
  \[
C_{k}^- = \{   t_{\left(m^{2k + k_0}_r, n^{2k + k_0}_r\right)} : r \in \{ 1,\dots, a_{2k+k_0+2}\}\}
  \]
 and \[
 C_k^+ = \{ t_{\left(m^{2k+k_0-1}_r, n^{2k+k_0-1}_r\right)} : r \in \{ 1,\dots, a_{2k + k_0 +1}\}   \},
 \]
 where we have used the claim to order the elements $t_{\left(m^k_r, n^k_r\right)}$ in $T$ by using their first coordinate. 

 Thus, we have shown that for $k\geq 1$, we have $b_{2k} = \left| C^-_{k} \right| = a_{2k+k_0+ 2} $ and $b_{2k+ 1} = \left|C^+_k  \right| = a_{2k+k_0 + 1} $. This is what we wanted to show, with $M= k_0$.

\end{proof}

\begin{corollary}\label{corollary:periodic_sigma_sequence}
    If $a_0, a_1, \dots$ are such that $$\alpha = [a_0 ; a_1, \dots, a_{k_0 - 1}, \overline{ a_{k_0}, a_{k_0 + 1}, \dots, a_{k_0 + l - 1} }],$$ then the sequence $\sigma(\alpha)$ is given by $\sigma(\alpha) = (\overline{a_{k_0}, a_{k_0 + 1}, \dots, a_{k_0 + l - 1}})$.
\end{corollary}

\begin{proof}
    We know that $\sigma(\alpha)$ is periodic by Lemma \ref{lemma:A_action_sigma_periodic}. Since the tails of the sequences $\sigma(\alpha)$ and $(a_0, a_1, a_2,\dots)$ coincide, the sequences must have the same periodic part. 
\end{proof}

\section{Proof of Theorem \ref{introthm:FW_main_theorem}} \label{subsection:proofs_of_FW_theorems}

In this section, we prove Theorem \ref{introthm:FW_main_theorem} using results from the previous sections. Then, we give an example of two non-isomorphic planes corresponding to different flows $\varphi_A$ and $\varphi_B$. Finally, we prove Corollary \ref{introcor:FW_cor}.

First, we restate Theorem \ref{introthm:FW_main_theorem} as Theorem \ref{theorem:FW_main_theorem} below, making use of the fact that the continued fraction expansion of quadratic integers is periodic.

\begin{definition}
    Let $A\in \mathrm{SL}(2,\Z)$ be a hyperbolic linear map. Let $\alpha(A)$ be the slope of the stable eigenspace of $A$. 
\end{definition}

\begin{theorem}\label{theorem:FW_main_theorem}
    Let $A$, $B \in \mathrm{SL}(2,\Z)$ be hyperbolic linear maps, and let \linebreak $(P_A, \F^+_A, \F^-_A), (P_B, \F^+_B, \F^-_B)$ be the bifoliated planes associated to the Franks-Williams flows $\varphi_A$ and $\varphi_B$. Let $\alpha(A) = \pm \left[a_0; a_1, \dots, a_{k_0 - 1}, \overline{a_{k_0}, a_{k_0+1}, \dots, a_{k_0+M - 1}}\right], \alpha(B) = \pm \left[b_0; b_1, \dots, b_{j_0-1}, \overline{b_{j_0}, \dots, b_{j_0 + N - 1}}\right]$. 

    Then, if $(P_A, \F^+_A, \F^-_A)$ and $ (P_B, \F^+_B, \F^-_B)$ are isomorphic, we must have $N = M$ and, possibly up to cyclic reordering, $a_{k_0} = b_{j_0}, a_{k_0 + 1} = b_{j_0 + 1}, \dots, a_{k_0+M -1} = b_{j_0 +N - 1} $.
\end{theorem}

\begin{proof}[Proof of Theorem \ref{theorem:FW_main_theorem}] 

Let $H: (P_A, \F^+_A, \F^-_A)\to(P_B, \F^+_B, \F^-_B)$ be an isomorphism, and let $\wr \subset l\in \F^-_A$ be a half-leaf based at a point $\walpha \in P_A$, where the orbit $\walpha$ is a lift to the universal cover $\widetilde{M}$ of $\alpha_1 = \mathcal{O}(z) \subset M_1 $ of the point $z\in M_1$.

By Corollary \ref{proposition:equivalence_crossing_integer_points}, there is an order preserving bijective correspondence between the set $C(\wr)$ of crossing infinite perfect fits for $\wr$ and the set \linebreak $ 
\left\{\wG^u(m,n) \cap \wG^s(0,0): (m,n) \in CP(\wz) \right\},$
which, recalling the definition of the set $CP(\wz)$ of crossing points for $\wz$ (Definition \ref{definition:FW_crossing_point}), is the same as the set
$$ 
\left\{\wG^u(m,n) \cap \wG^s(0,0): \wG^u(m,n)\cap r\neq \emptyset,\, R_{(m,n)}\cap \Z^2 = \{ (0,0), (m,n)\} \right\},$$
 where $r$ is the connected component of $\wG^s(0,0)\setminus \{ (0,0)\}$ that contains $\wz\in \R^2$.

This set is order isomorphic to the set $T^+_A \cup T^-_A$, which can be seen by simply mapping $\wG^u(m,n) \cap \wG^s(0,0)$ to $t_{(m,n)}$. Additionally, this order isomorphism maps intersections $\wG^u(m,n) \cap \wG^s(0,0)$ to $t_{(m,n)}$ where $(m,n)$ is to the right (left) of $\wG^s(0,0)$ to points in $T^-$ (resp. $T^+$).

Composing the order isomorphism given by Corollary \ref{proposition:equivalence_crossing_integer_points} with the identification between its target and $T^+ \cup T^-$, we obtain an order isomorphism $C^A = C(\wr) \to T^+_A \cup T^-_A$, which maps $C_r^A = C_r(\wr)$ to $T^-_A$ and $ C_l^A = C_l(\wr)$ to $T^+_A$. The same can be done with the bifoliated plane $(P_B, \F^+_B, \F^-_B)$, obtaining an order preserving bijection $C^B \to T_B$ which restricts to order preserving bijections $C^B_r \to T^-_B$, $C^B_l \to T^+_B$.

By Corollary \ref{corollary:FW_isomorphic_implies_equivalent_patterns}, $H$ induces an equivalence between the intersection patterns of crossing infinite perfect fits $\left( C^A, C^A_r, C^A_l \right)$ and $\left( C^B, C^B_r, C^B_l \right)$.

We conclude then that there is an order isomorphism between $T^+_A\cup T^-_A$ and $T^+_B \cup T^-_B$ which maps $T^+_A$ to $T^+_B$ and $T^-_A$ to $T^-_B$. This implies that $\sigma(\alpha(A)) = \sigma(\alpha(B))$. 

Therefore, we conclude by Corollary \ref{corollary:periodic_sigma_sequence} that the continued fraction expansions of $\alpha(A)$ and $\alpha(B)$ have the same periodic part, which shows the result.

\end{proof}

\begin{figure}[h!!]
  \centering
  \includegraphics[width=.9\linewidth]{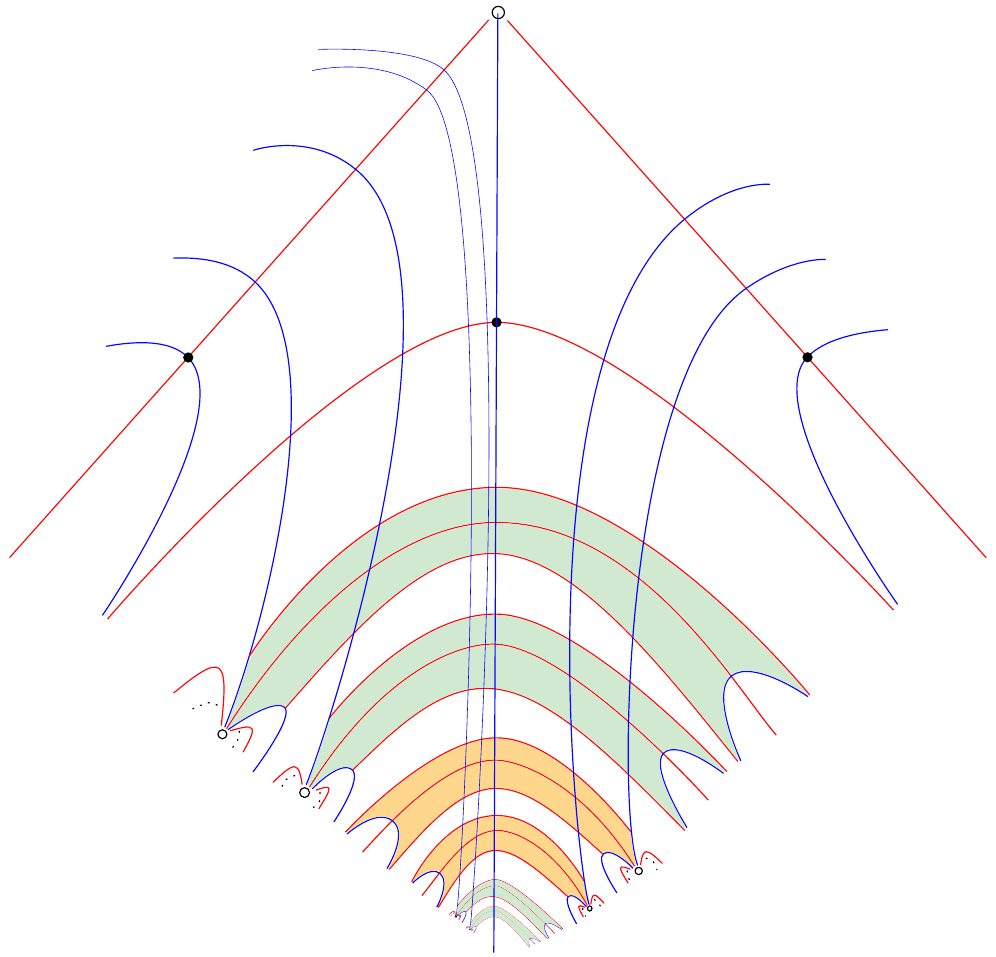}
  \caption{Left- and right-crossing infinite perfect fits for $A = \begin{pmatrix}
      1&2\\-2&5
  \end{pmatrix},$ c.f. Figure \ref{fig:eigenspaces_example}. }  
  \label{fig:crossing_pattern}
\end{figure}

\begin{example}
    Let $A  = \begin{bmatrix}
        2&1\\1&1
    \end{bmatrix}$ and $B = \begin{bmatrix}
        5&2\\2&1
    \end{bmatrix}$. 
    
    Then, $\alpha(A) = \frac{-1 - \sqrt{5}}{2}  = -[1;\overline{1}]$ and $\alpha(B) = -1 - \sqrt{2} = -[2;\overline{2}]$. Therefore  $(P_A, \F^+_A, \F^-_A) $ and $ (P_B, \F^+_B, \F^-_B)$ are not isomorphic.
\end{example}

\begin{example}
    More generally, let $R = \begin{pmatrix}
        1&1\\ 0&1
    \end{pmatrix},\, L = \begin{pmatrix}
        1&0\\1&1
    \end{pmatrix}$. 
    
    One can show via a straightforward calculation that if $x = [0; \overline{a_1, b_1, a_2, b_2,\dots, a_k, b_k}]$ then $x = \beta(A)$, where $A = R^{a_1}L^{b_1}\dots R^{a_k}L^{b_k}$. 
    
    This allows one to construct countably many matrices $A_n \in \mathrm{SL}(2,\Z)$ such that the slopes $\beta(A_n)$ have different periodic parts, and therefore countably many non-isomorphic bifoliated planes $\left( P_{A_n}, \F^+_{A_n}, \F^-_{A_n} \right)$.
\end{example}

Now, we prove Corollary \ref{introcor:FW_cor}. Recall that it states that if $(P_A, \F^+_A, \F^-_A)$ and $(P_B, \F^+_B, \F^-_B)$ are isomorphic, then there exist powers $A^k$ and $B^l$ that are conjugate in $\mathrm{GL}(2,\Z)$.

\begin{proof}[Proof of Corollary \ref{introcor:FW_cor}]

By Theorem \ref{theorem:FW_main_theorem}, we know that $\alpha(A)$ and $\alpha(B)$ have the same periodic part.

    Given $x\in \R$ and a matrix $Q\in \mathrm{SL}(2,\Z)$, we denote by $Q \cdot x$ the slope of the image of the line through the origin with slope $x$ under multiplication by $Q$. That is, if $Q = \begin{pmatrix}
        a&b\\c&d
    \end{pmatrix}$ then $Q\cdot x = \frac{c + dx}{a + bx}$. In particular, the action of $\mathrm{SL}(2,\Z)$ on the set $\R P^1$ of lines through the origin (where each line is identified with its slope) is conjugate via the matrix $\begin{pmatrix}
        0&1\\1&0
    \end{pmatrix}$ to the standard action of $\mathrm{SL}(2,\Z)$ on $\partial \H^2 = \R \cup \{\infty\} $ given by $\begin{pmatrix}
        a&b \\c&d
    \end{pmatrix} \cdot x = \frac{ax + b}{cx + d}$. It's enough then to study the latter action.

   Since $\alpha(A)$ and $\alpha(B)$ have the same periodic part, one can construct a matrix $C \in \mathrm{GL}(2,\Z)$ such that $C \cdot \alpha(A) = \pm \alpha(B)$ (see Section 10.11 of \cite{hardy75}). Therefore, 
    $$(C^{-1} B C) \cdot \alpha(A) = \alpha(A),$$
    so the matrices $A$ and $C^{-1}BC$ are both in the stabilizer of $\alpha(A)$. Since the projection of this stabilizer to $\mathrm{PSL}(2,\Z)$ is infinite cyclic (recall that a discrete subgroup of orientation preserving isometries of $\mathrm{H}^2$ that fixes a point at infinity is either trivial or cyclic), there exists $M \in \mathrm{SL}(2,\Z)$ such that $M^i = \pm A$, $M^j = \pm C^{-1}BC$. Therefore, we have proven the result.
\end{proof}


\printbibliography


\end{document}